\documentclass[11pt,oneside]{article}

\usepackage{mathptmx}

\usepackage{amsmath,amssymb}
\usepackage{mathtools}
\usepackage{psfrag}
\usepackage[english]{babel} 
\usepackage{graphicx}
\usepackage{caption}
\usepackage{amsmath}
\usepackage{amssymb}
\usepackage{verbatim, bbm}
\usepackage{tabu}

\usepackage{makecell}
\usepackage{colortbl}
\usepackage{multirow}
\usepackage{hhline}
\usepackage{arydshln}
\usepackage{xspace}

\usepackage{enumerate}
\usepackage{amsfonts}
\usepackage{color}
\usepackage[overload]{empheq}
\usepackage{subfig}
\usepackage{float}
\usepackage[section]{placeins}
\usepackage[hyphens]{url}
\usepackage{hyperref}
\usepackage{cite}
\usepackage{mathrsfs}

\newtheorem{proposition}{Proposition}[subsection]
\newtheorem{theorem}{Theorem}[subsection]

\newtheorem{corollary}[proposition]{Corollary}
\newtheorem{definition}{Definition}[subsection]
\newtheorem{remark}{Remark}

\newcommand\restr[2]{{
		\left.\kern-\nulldelimiterspace 
		#1 
		\vphantom{\big|}
		\right|_{#2} 
}}

\newtheorem{algorithm}{Algorithm}

\newenvironment{proof}{{\noindent \it Proof.}}{\hfill $\fbox{}$ \vspace*{5mm}}

\newcommand{\bx}{{\bf x}}

\newcommand{\N}{\mathbb{N}}
\newcommand{\R}{\mathbb{R}}

\newcommand{\diag}{{\rm diag}}

\newcommand{\rank}{{\rm rank}}

\newcommand{\BE}{\begin{equation}}
\newcommand{\EE}{\end{equation}}

\newcommand{\Eulerc}{\prescript{\textrm{e}^{\sqrt{\alpha}}}{1}{\mathcal{L}_{\textnormal{dir},\alpha x^2}}}
\newcommand{\Euler}{\prescript{\textrm{e}^{\sqrt{\alpha}}}{1}{\mathcal{L}_{\textnormal{dir},\alpha x^2}^{(n)}}}
\newcommand{\Eulerw}{\prescript{\textrm{e}^{\sqrt{\alpha}}}{1}{\hat{\mathcal{L}}_{\textnormal{dir},\alpha x^2}^{(n)}}}

\newcommand{\numerr}{\prescript{}{\alpha}{\textnormal{\textbf{err}}_{k}^{(n)}}}
\newcommand{\analerr}{\prescript{}{\alpha}{\tilde{\textnormal{\textbf{err}}}_{r,k}^{(n)}}}
\newcommand{\analerrExact}{\prescript{}{\alpha}{\tilde{\textnormal{\textbf{err}}}_{k}^{(n)}}}

\numberwithin{equation}{section}
\date{}

\begin{document}
	\title{Analysis of the spectral symbol function for spectral approximation of a differential operator}
	\author{Davide Bianchi\footnote{University of Insubria - Como (Italy). Email: d.bianchi9@uninsubria.it} \textsuperscript{ ,}\footnote{Partially supported by INdAM-GNCS Gruppo Nazionale per il Calcolo Scientifico}}
	\maketitle
	
	\begin{abstract}
Given a differential operator $\mathcal{L}$ along with its own eigenvalue problem $\mathcal{L}u = \lambda u$ and an associated algebraic equation $\mathcal{L}^{(n)} \mathbf{u}_n = \lambda\mathbf{u}_n$ obtained by means of a discretization scheme (like Finite Differences, Finite Elements, Galerkin Isogeometric Analysis, etc.), the theory of Generalized Locally Toeplitz (GLT) sequences serves the purpose to compute the spectral symbol function $\omega$ associated to the discrete operator $\mathcal{L}^{(n)}$

We prove that the spectral symbol $\omega$ provides a necessary condition for a discretization scheme in order to uniformly approximate the spectrum of the original differential operator $\mathcal{L}$. The condition measures how far the method is from a uniform relative approximation of the spectrum of $\mathcal{L}$.  Moreover, the condition seems to become sufficient if the discretization method is paired with a suitable (non-uniform) grid and an increasing refinement of the order of approximation of the method. 

On the other hand, despite the numerical experiments in many recent literature, we disprove that in general a uniform sampling of the spectral symbol $\omega$ can provide an accurate relative approximation of the spectrum, neither of $\mathcal{L}$ nor of the discrete operator $\mathcal{L}^{(n)}$.    
\end{abstract}



\section{Introduction}\label{sec:intro}
For simplicity, throughout this paper our main model problem will be the one dimensional Sturm-Liouville equation
 \begin{equation}\label{eq:S-L_general}
\begin{cases}-\partial_x\left(p(x)\partial_xu(x)\right) +q(x)u(x)= f(x) & x\in (a,b),\\
\textnormal{boundary conditions},
\end{cases}
\end{equation}
and its associated (regular) Sturm-Liouville eigenvalues problem (SLP)
\begin{equation}\label{eq:S-L_general_eig_prob}
\begin{cases}-\partial_x\left(p(x)\partial_xu(x)\right) +q(x)u(x)= \lambda w(x)u(x) & x\in (a,b),\\
\textnormal{boundary conditions}.
\end{cases}
\end{equation}

As it will be clear in Subsection \ref{ssec:generalization}, a generalization to other kind of differential operators in higher dimensions can be handled by the same techniques we are going to discuss for this model case.

There are several many approaches to compute the eigenvalues $\lambda_k$ related to Problem \eqref{eq:S-L_general_eig_prob}, such as matrix methods, Pr\"{u}fer's methods, sampling methods, etc. We refer to \cite{P93,SG05,ZD98} and \cite{BC96,A95} for a general basic overview and we cite some recent interesting papers about transmutation operator methods \cite{KT15,KNT17}.

Matrix methods consist in discretizing Problem \eqref{eq:S-L_general_eig_prob} in the form 
\begin{equation}\label{eq:eig_algebraic}
\mathcal{L}^{(n)} \mathbf{u}_n = \lambda^{(n)} \mathbf{u}_n,
\end{equation}
where $\mathcal{L}^{(n)}$ is a matrix of order $n\times n$, with $n$ a mesh fineness parameter, and then solving the associated algebraic eigenvalue problem. They generally suffer of a (fast) growth of the absolute (and relative) error, that is the quality of the numerical approximation they provide for the $k$-th eigenvalue deteriorates significantly as the index $k$ increases. 
Even if there have been developed suitable correction techniques for improving the accuracy of the computed eigenvalues (see as an example for the FD case the first paper appeared in this sense \cite{PHA81}, and a more recent development \cite{AGM09} for the Numerov's method), if the interest relies only on the computation of the eigenvalues/eigenfunctions, then other methods typically can be a better choice, such as shooting-type methods for example, which can provide more uniformly accurate estimates of the eigenvalues with respect to the index $k$, at the cost of a fair more effort in the implementation, see \cite[Chapter 5]{P93}. 

Nevertheless, many other problems demand to discretize equation \eqref{eq:S-L_general} in such a way that the spectrum of the discrete operator $\mathcal{L}^{(n)}$ preserves the spectrum of the continuous operator $\mathcal{L}$
 uniformly with respect to $k=1,\cdots,n$. Within this regard, see for example the spectral gap problem of the $1d$ wave equation for the uniform observability of the control waves, \cite{IZ99} and \cite{EMZ16,BS18}, or structural engineering problems, see \cite[Section 3.1]{HER}, or finally the discretization of the Laplace-Beltrami operator by means of the graph-Laplacian, see \cite{RBV10,BIK,LBB}.

In this setting, the theory of Generalized Locally Toeplitz (GLT) sequences provides the necessary tools to understand whether the matrix methods used to discretize the operator $\mathcal{L}$ are efficient or not to spectrally approximate it. The GLT theory originated from the seminal work of Tilli on Locally Toeplitz (LT) sequences \cite{Ti98} and then afterwards developed by Serra-Capizzano in \cite{Serra03,Serra06}. It was devised to compute and analyze the spectral distribution of matrices arising from the numerical discretization of integral equations and differential equations. One of the purposes of this spectral analysis is the design of efficient numerical methods for computing the related numerical solutions by fast iterative solvers (especially, multigrid and preconditioned Krylov methods), in this sense see for example the related works \cite{DGMS15,DMS16,DFFMS18}.

It often happens that the discrete operator $\mathcal{L}^{(n)}$ from \eqref{eq:eig_algebraic}, properly weighted by a power of $n$ (depending on the dimension of underlying space and on the maximum order of derivatives involved) and which we will denote as $\hat{\mathcal{L}}^{(n)}$, enjoys an asymptotic spectral distribution as $n\to \infty$, i.e., as the mesh is progressively refined. More precisely, for any test functions $F(t) \in C_c(\mathbb{C})$,
 \begin{equation}\label{eq:spectral_symbol_intro}
 \lim_{n \to \infty} \frac{1}{n} \sum_{k=1}^{n}F\left(\lambda_k\left(\hat{\mathcal{L}}^{(n)}\right)\right) = \frac1{\mu_m(D)}\int_D  F(\omega(\bx))d\mu_m(\bx),
 \end{equation}
where $\lambda_k\left(\hat{\mathcal{L}}^{(n)}\right)$ is the $k$-th eigenvalue of the weighted operator $\hat{\mathcal{L}}^{(n)}$ and $\omega: D\subset \R^m \to \mathbb{C}$ is referred to as the {\em spectral symbol} of the sequence $\left\{\hat{\mathcal{L}}^{(n)} \right\}_{n\in\N}$, see \cite[Definition 2.6]{BS13} in relation with Toeplitz operators. 

The GLT theory serves the purpose to compute the spectral symbol $\omega$ related to matrix methods employed to get the matrix equation \eqref{eq:eig_algebraic}, especially if the numerical method belongs to the family of the so-called {\em local methods}, such as FD methods,
Finite Element (FE) methods and collocation methods  with locally supported basis functions.

In several recent papers it was advanced the suggestion to exploit the sampling of the spectral symbol for approximating the spectrum of the discrete operator $\hat{\mathcal{L}}^{(n)}$, see as a main reference the paper \cite{GSERSH18} with all the references therein, which is a full review up to the state-of-the-art of the symbol-based analysis for the eigenvalues distribution carried on in the framework of the isogeometric Galerkin approximation (IgA). Unfortunately, this does not apply in general, as we are going to prove. 

Nevertheless, the spectral symbol provides to a matrix method a necessary condition for the uniform spectral approximation of the continuous operator $\mathcal{L}$, in the sense of the relative error. The condition is easily computable and it seems to be sufficient if the discretization method is paired with a suitable (non-uniform) grid and an increasing refinement of the order of approximation of the method. 

The paper is organized as follows. 
\begin{itemize}
	\item Section \ref{sec:model_equation} is a summary about the main properties concerning the SLP model equations \eqref{eq:S-L_general} and \eqref{eq:S-L_general_eig_prob}.  
	
	\item In Section \ref{sec:GLT}, the theory of GLT sequences is briefly introduced. In Subsections \ref{ssec:FD} and \ref{ssec:IsoG} we present the spectral symbols concerning the discretization of \eqref{eq:S-L_general_eig_prob2} by means of central FD and IgA schemes of varying order, respectively. In Subsection \ref{ss:rearrangment}  it is introduced the monotone rearrangement of the spectral symbol and it is explained how to compute an approximation of it, whenever it is not feasible to get an exact analytical expression. We highlight the connection with the theory of uniformly distributed sequences, see Theorem \ref{thm:discrete_Weyl_law}.
	
	\item Section \ref{sec:example} is mainly dedicated to numerical experiments: we analyze the SLP in the Euler-Cauchy equation case. Since the eigenvalues for this regular SLP are computable in closed form, it is a suitable toy-model example with variable coefficients. In Subsection \ref{ssec:example_uniform_3_points}, by easy calculations we show that a uniform sampling of the monotone rearrangement of the spectral symbol produces a positive lower bound in the relative error approximation, independent by the mesh finesse parameter $n$, which prevents to obtain an accurate approximation of the eigenvalues of the Euler-Cauchy differential operator. In Subsection \ref{ssec:L1}, furthermore we present a negative result in the case of spectral symbols $\omega$ belonging to the $L^1$ class. Finally, numerical validations are provided to the necessary condition for a uniform spectral approximation.  
	
	In Subsection \ref{ssec:FD_nonuniform} and Subsection \ref{ssec:galerkin} we generalize what observed in the previous subsection to the case of central FD and IgA methods of higher order. By means of a suitable non-uniform grid suggested by the spectral symbol, we observe by numerical experiments that the discretization methods produce a uniform relative approximation of the spectrum of the continuous differential operator, as the mesh fineness parameter $n$ and the order of the methods increase.  
	
	\item In Section \ref{sec:theory} we collect the proof of the results discussed previously. In particular:
	\begin{itemize}
		\item it is proved that in general a uniform sampling of the spectral symbol $\omega$ does not produce an accurate approximation nor of the eigenvalues of $\mathcal{L}$, nor of the eigenvalues of $\hat{\mathcal{L}}^{(n)}$;
		\item  it is provided a necessary condition to a numerical matrix method for the relative uniform approximation of the spectrum of $\mathcal{L}$;
	\end{itemize} 
All the results can be extended to different matrix methods and differential operators, once there are known the spectral symbol $\omega$ of the matrix method and the asymptotic expansion of the eigenvalues of the differential operator.   
\end{itemize}


\subsection{About the notations}
 We will use the notation $A^{(n)},B^{(n)},C^{(n)}, T^{(n)}$ to indicate general $n\times n$ square matrices, and we will use the notation $\mathcal{L}^{(n,\eta)}$ to indicate a $n\times n$ square matrix which is the discretization of a differential operator $\mathcal{L}$ by means of a numerical method of order of approximation $\eta$. In the case that the approximation order $\eta$ is implied by the context, then we will omit it. If the discretized operator $\mathcal{L}^{(n)}$ is weighted by a constant depending by the finesses mesh parameter $n$, then we will denote it with $\hat{\mathcal{L}}^{(n)}$. We will use the subscripts {\em dir} and {\em BCs} to indicate a (discretized) differential operator characterized by Dirichlet or generic BCs, respectively. When it will be necessary to highlight the dependency of the differential operator with respect to the coefficients $p(x),q(x),w(x)$ and the endpoints $a,b$, we will write them as right and left subscripts/supscripts. So, for example, the weighted discretization of a  Sturm-Liouville operator with Dirichlet BCs by means of the IgA method of order $\eta$ will be denoted by 
 $$
 \prescript{b}{a}{\hat{\mathcal{L}}^{(n,\eta)}_{\textnormal{dir},p(x),q(x),w(x)}}.
 $$
 In the special case of the (negative) Laplace operator we will use the symbol $-\Delta$, and all the other previous notations will apply.
 
 For all the constants we will use the letter $c$, making explicit the dependency to other parameters if needed.
 
 Finally, if not stated differently, we will always consider a non-decreasing ordering of the eigenvalues $\lambda_k$ of a given operator.

\section{The model equation}\label{sec:model_equation}
Our model equation will be the following Sturm-Liouville equation with separated boundary conditions (BCs),
 \begin{equation}\label{eq:S-L_general2}
\begin{cases}-\partial_x\left(p(x)\partial_xu(x)\right) +q(x)u(x)= 0 & x\in (a,b)\subset \R,\\
\sigma_1 u(a) - \sigma_2 p(a)\partial_xu(x)_{|x=a} =0 & \sigma_1^2+\sigma_2^2 >0,\\
\zeta_1 u(b) +\zeta_2 p(b)\partial_xu(x)_{|x=b} =0 & \zeta_1^2+\zeta_2^2 >0,
\end{cases}
\end{equation}
and its associated weighted Sturm-Liouville eigenvalue problem (SLP)

\begin{equation}\label{eq:S-L_general_eig_prob2}
\begin{cases}-\partial_x\left(p(x)\partial_xu(x)\right) +q(x)u(x)= \lambda w(x) u(x) & x\in (a,b)\subset \R,\\
\sigma_1 u(a) - \sigma_2 p(a)\partial_xu(x)_{|x=a} =0 & \sigma_1^2+\sigma_2^2 >0,\\
\zeta_1 u(b) +\zeta_2 p(b)\partial_xu(x)_{|x=b} =0 & \zeta_1^2+\zeta_2^2 >0.
\end{cases}
\end{equation}
There is an extensive literature concerning the Sturm-Liouville eigenvalue boundary problem \eqref{eq:S-L_general_eig_prob2}, see as a (not exhaustive) collection of references \cite{E05,S12,Z15}. We will write $\partial_x f$ and $f'$ equivalently to indicate the first derivative of a function $f$ with respect to its argument. 

If not otherwise explicitly stated, we will suppose that $p,p',w,w',q,(pw)',(pw)'' \in C([a,b])$, $p,w>0$, $q\geq 0$ and $\sigma_1^2+\sigma_2^2 >0, \zeta_1^2+\zeta_2^2>0$. Under these conditions, problem \eqref{eq:S-L_general_eig_prob2} is regular and the differential operator 
\begin{align*}
\mathcal{L}(\cdot) := -w(x)^{-1} \left[\partial_x \left(p(x) \partial_x (\cdot) \right) + q(x)(\cdot)\right],
\end{align*}
defined over an appropriate domain $\textnormal{dom}\left(\mathcal{L}\right)\subset L^1([a,b])$, depending on the boundary conditions, is self-adjoint and has an increasing sequence of positive real eigenvalues $0\leq \lambda_1 <\lambda_2 < \cdots<\lambda_n$ of multiplicity one and such that 
\begin{equation}\label{eq:weyl_asymptotic}
\lim_{n \to \infty}\frac{\lambda_n}{n^2} = \pi^2 \left(\int_a^b \sqrt{\frac{w(x)}{p(x)}}dx\right)^{-2}.
\end{equation}
By the Liouville transformation 
\begin{align}\label{eq:liouville_transform}
&y:= \int_a^x \sqrt{\frac{w(t)}{p(t)}} dt,\\
&v(y) = u(x(y)) (w(x(y)) p(x(y)))^{1/4},\nonumber
\end{align} 
the regular SLP \eqref{eq:S-L_general_eig_prob2} converts to the Liouville normal form
\begin{equation}\label{eq:S-L_normal_eig_prob}
\begin{cases}-\Delta v(y) + V(y)v(y) = \lambda v(y)& y\in (0,B)\subset \R,\\
\Sigma_1 v(0) - \Sigma_2 \partial_yv(y)_{|y=0} =0 & \Sigma_1^2+\Sigma_2^2 >0,\\
Z_1 v(B) +Z_2\partial_yv(y)_{|y=B} =0 & Z_1^2+Z_2^2 >0,
\end{cases}
\end{equation}
where $\Delta$ is the Laplace operator and 
\begin{align*}
&B = \int_a^b \sqrt{\frac{w(t)}{p(t)}} dt,\\
&V(y)= \frac{g''(y)}{g(y)} + \frac{q(x(y))}{w(x(y))} \in C([0,B]), \qquad &g(y):= (w(x(y)) p(x(y)))^{1/4},\\
&\Sigma_1=  \frac{\sigma_1}{(w(a)p(a))^{1/4}} + \sigma_2\left( \frac{(wp)'(a)}{4(p(a))^{1/4}(w(a))^{5/4}}\right), &\Sigma_2=(w(a)p(a))^{1/4}\sigma_2,\\
&Z_1=  \frac{\zeta_1}{(w(b)p(b))^{1/4}} - \zeta_2\left( \frac{(wp)'(b)}{4(p(b))^{1/4}(w(b))^{5/4}}\right), &Z_2=(w(b)p(b))^{1/4}\zeta_2.
\end{align*}
Indicating with $v(y,\lambda)$ a solution of the differential equation in \eqref{eq:S-L_normal_eig_prob} and normalizing it by the initial conditions at $y=0$, namely
$$
v(0,\lambda)=\Sigma_2, \qquad v'(0,\lambda)=\Sigma_1,
$$ 
then the corresponding solution $u(x,\lambda)= (r(x)p(x))^{-1/4}v(y(x),\lambda)$ of \eqref{eq:S-L_general_eig_prob} is normalized by
$$
u(a,\lambda)=\sigma_2, \qquad p(a)u'(a,\lambda)=\sigma_1.
$$
Therefore, the eigenvalues of the SLP \eqref{eq:S-L_general_eig_prob2} are the same of the eigenvalues of the SLP \eqref{eq:S-L_normal_eig_prob}, being the zeros of the entire function
$$
\Psi(\lambda) = \zeta_1 u(b,\lambda) + \zeta_2 p(b)u'(b,\lambda) = Z_1v(B,\lambda) + Z_2 v'(B,\lambda).
$$
In particular, the eigenvalues and eigenfunction norms remain invariant under the change from \eqref{eq:S-L_general_eig_prob2} to \eqref{eq:S-L_normal_eig_prob}. 

\section{Generalized Locally Toeplitz sequences and spectral symbol}\label{sec:GLT}
In this section we are going to provide a brief summary about the theory of GLT sequences, starting from the spectral symbol, see Definition \ref{def:ss_def}. For a detailed treatment of the theory of the GLT sequences we invite the reader to look at \cite{GS17,GS18} and all the references therein. In Subsections \ref{ssec:FD} and \ref{ssec:IsoG} we present the spectral symbols concerning the discretization of \eqref{eq:S-L_general_eig_prob2} by means of $(2\eta+1)$-points central FD approximation and Isogeometric Galerkin approximation (IgA) based on B-splines of degree $\eta$ and smoothness $C^{\eta-1}$. We will not discuss directly here the case of spectral symbols concerning Finite Elements method or IgA with more general smoothness assumptions, in order to avoid that this paper becomes unnecessarily long.

In Subsection \ref{ss:rearrangment}, we present a natural way to approximate the monotone rearrangement of a spectral symbol $\omega$. 

\subsection{Preliminaries on GLT sequences and spectral symbol}

A matrix $T^{(n)}$ is said to be Toeplitz if it is a matrix with constant coefficients along its diagonals, i.e., if it is of the form $T^{(n)}_{i,j}= t_{i-j}$ for all $i,j=1,\ldots,n$, and with $\mathbf{t}=\left[t_{-n+1}, \ldots, t_0, \ldots, t_{n-1}\right] \in \mathbb{C}^{2n-1}$. If $t_k$ is the $k$-th Fourier coefficient of a complex integrable function $f$ defined over the interval $[-\pi, \pi]$, then $T^{(n)}= T^{(n)}(f)$ is said to be the Toeplitz matrix generated by $f$. Matrices with a Toeplitz-related structure naturally arise when discretizing, over a uniform grid, problems which have a translation invariance property, such as differential operators with constant coefficients.

From \cite{GS17} we have the following definitions.
\begin{definition}[\textbf{Spectral symbol}]\label{def:ss_def}
	Let $\{T^{(n)}\}_{n\in\N}$ be a sequence of matrices tending to infinity as $n\to \infty$. Let $\omega:D \subset \R^m \to\mathbb C$ be a measurable function defined on a set $D$ with $0<\mu_m\left(D\right)<\infty$, where $\mu_m$ is the Lebesgue measure on $\R^m$, and 
	$$
	\bx := (x^1,\cdots,x^\nu,\theta^{\nu+1},\cdots,\theta^m), \quad \nu \in \{0,1\cdots,m\}.
	$$ 
	We say that $\{T^{(n)}\}_{n\in\N}$ has a {\em spectral (or eigenvalue) distribution} described by $\omega$, and we write 
	\begin{equation}\label{eq:symbol_distribution}
	\{T^{(n)}\}_{n}\sim_\lambda\omega,
	\end{equation}
	if for all functions $F\in  C_c(\mathbb{C})$ (i.e., continuous with compact support) we have
	\begin{equation} \label{def_asym-bis-Matrix}
	\lim_{n \to \infty} \frac{1}{n} \sum_{k=1}^{n}F(\lambda_k(T^{(n)})) = \frac{1}{\mu_m(D)}\int_D  F(\omega(\bx))d\mu_m(\bx).
	\end{equation}
	In this case, $\omega$ is called {\em spectral (or eigenvalue) symbol} of $\{T^{(n)}\}_{n\in\N}$. See Subsection \ref{ss:rearrangment} for its monotone rearrangement.
\end{definition}
Relation \eqref{def_asym-bis-Matrix} is satisfied for example by Hermitian Toeplitz matrices generated by real-valued functions $\omega\in L^1([-\pi,\pi])$, i.e., $T^{(n)}(\omega) \sim_\lambda \omega$, see \cite{Szego84,Tyrty96,Tyrty98}. For a general overview on Toeplitz operators and spectral symbol, see \cite{BS13}.

\begin{remark}\label{rem:ssymbol_sampling}
In particular, if $D$ is compact, $\omega$ real and continuous, and $\lambda_k \in \left[\min \omega, \max \omega\right]$ for every $k\in \N$, then 
 taking $F(\lambda)=\lambda\chi_D(\lambda)$, with $\chi_D$ a $C^\infty$ cut-off of $D$ such that $\chi_D(t)\equiv 1$ on $D$, then 
	\begin{equation}\label{ss_def2}
	\lim_{n \to \infty} \frac{1}{n} \sum_{k=1}^{n}\lambda_k(T^{(n)}) = \frac{1}{\mu_m(D)}\int_D \omega(\bx)d\mu_m(\bx).
	\end{equation}
	Because the Riemannian sum over equispaced points converges to the integral of the right hand side of the above formula, then \eqref{ss_def2} could suggest us that the eigenvalues $\lambda_k(T^{(n)})$ can be approximated by a pointwise evaluation of the symbol $\omega(\bx)$ over an equispaced grid of $D$, for $n\to \infty$, expect for at most an $o(n)$ of outliers, see Definition \ref{def:outliers}. This is mostly the content of \cite[Remark 3.2]{GS17} and \cite[Section 2.2]{GSERSH18}. 

We remark that in the special cases of Toeplitz matrices generated by enough regular functions $\omega : [0, \pi]\to \R$, then 
$$
\lambda_k\left(T^{(n)}(\omega)\right)= \omega\left(\frac{k\pi}{n+1}\right) + O(n^{-1}) \qquad \mbox{for every } k=1,\cdots,n,
$$
see \cite{W58} and \cite{BBGM15}. In some sense, this justifies the informal meaning given above to the spectral distribution $\omega$.
\end{remark}

An important consequence of \eqref{eq:symbol_distribution} is contained in the next Theorem  \ref{thm:clustering&spectral_attraction}, but we need a couple of definitions more. 
\begin{definition}[Clustering]\label{def:clustering}
Given a subset $\Omega\subset \mathbb{C}$ and a positive number $\epsilon>0$, define the $\epsilon$-expansion of $\Omega$ as
$$
\Omega_\epsilon := \bigcup_{z\in \Omega} B^{\mathbb{C}}_\epsilon (z).
$$
We say that a matrix sequence $\left\{ T^{(n)} \right\}_{n\in \N}$ is \emph{weakly clustered} at $\Omega$ if for every $\epsilon>0$ it holds that 
\begin{equation*}
\lim_{n \to \infty}\frac{\left|\left\{ k= 1,\ldots,n \, : \, \lambda_k \left(T^{(n)}\right) \notin \Omega_\epsilon  \right\} \right|}{n} = 0.
\end{equation*}
\end{definition}

\begin{definition}[Spectral attraction]\label{def:spectral_attraction}
Let $z\in \mathbb{C}$. Given a matrix sequence $\left\{ T^{(n)} \right\}_{n\in \N}$, let us order the eigenvalues of $T^{(n)}$ according to their distance from $z$, i.e.,
\begin{equation*}
\left|\lambda_{1} \left(T^{(n)}\right) -z \right| \leq \left|\lambda_{2} \left(T^{(n)}\right) -z \right|\leq \ldots\leq \left|\lambda_{n} \left(T^{(n)}\right) -z \right|.
\end{equation*}
We say that $z$ \emph{strongly attracts} the spectrum of $T^{(n)}$ with infinite order if
\begin{equation*}
\lim_{n \to \infty} \left|\lambda_{k} \left(T^{(n)}\right) -z \right| =0 \qquad \mbox{for every fixed }k.
\end{equation*}
\end{definition}

\begin{definition}[Essential range]\label{def:essential_range}
Let $\omega$ be a $\mu_m$-measurable function and define the set $R_\omega \subset \mathbb{C}$ as 
\begin{equation*}
t\in R_\omega \quad \Leftrightarrow \quad \mu_m\left(\left\{\bx \in D \, : \, \left| \omega(\bx) - t\right|< \epsilon   \right\}\right)>0 \quad \forall \epsilon >0.
\end{equation*}
We call $R_\omega$ the \emph{essential range} of $\omega$.  $R_\omega$ is closed.
\end{definition}

\begin{theorem}\label{thm:clustering&spectral_attraction}
Let $\{T^{(n)}\}$ be a matrix sequence such that 	
		\begin{equation*}
	\left\{T^{(n)}\right\}_{n\in \N} \sim_{\lambda} \omega(\bx), \qquad \bx \in D\subset \mathbb{R}^m.
	\end{equation*}
Then $\{T^{(n)}\}_n$ is weakly clustered at $R_\omega$ and every point of $R_\omega$ strongly attracts the spectrum of  $T^{(n)}$ with infinite order.
\end{theorem}
\begin{proof}
See \cite[Theorem 3.1]{GS17}.
\end{proof}

\begin{definition}[Outliers]\label{def:outliers}
Given  a matrix sequence $\{T^{(n)}\}$ such that $\left\{T^{(n)}\right\}_{n\in \N} \sim_{\lambda} \omega$, if $\lambda_k\left( T^{(n)}\right)\notin R_\omega$, we call it an \emph{outlier}.
\end{definition}

The matrices which represent the discrete version of a general differential operator do not always own a Toeplitz structure and therefore we cannot know beforehand whether they posses a spectral symbol or not. Hereafter we begin to introduce the concepts which will bring us to the definition of GLT sequences.
\begin{definition}[Approximating class of sequences]
Let $\{A^{(n)}\}_{n\in\N}$ be a matrix-sequence and let $\{\{B^{(n,m)}\}_{n\in\N}\}_{m\in\N}$ be a sequence of
matrix-sequences. We say that $\{\{B^{(n,m)}\}_n\}_m$ is an {\em approximating class of sequences (a.c.s.)} for $\{A^{(n)}\}_{n}$, and we write
$$
\{B^{(n,m)}\}_n \to \{A^{(n)}\}_{n} \qquad  \mbox{a.c.s.},
$$
if the following condition is met: for every $m$ there exists $n_m$ such that, for $n>n_m$,
$$
A^{(n)} =  B^{(n,m)} + R^{(n,m)} + N^{(n,m)}, \qquad \rank\left(R^{(n,m)}\right) \leq c_1(m)n, \qquad \|N^{(n,m)}\|\leq c_2(m),
$$
where $n_m, c_1(m), c_2(m)$ depend only on $m$, and
$$
\lim_{m \to \infty} c_1(m) = \lim_{m \to \infty} c_2(m) =0.
$$
\end{definition}
Roughly speaking, $\{\{B^{(n,m)}\}_n\}_m$ is an a.c.s. for $\{A^{(n)}\}_{n}$ if, for all sufficiently large $m$, the sequence $\{B^{(n,m)}\}_n$ approximates $\{A^{(n)}\}_{n}$ in the sense that $A^{(n)}$ is eventually equal to $B^{(n,m)}$ plus a small-rank matrix (with respect to the matrix size $n$) plus a small-norm matrix.

\begin{definition}[Locally Toeplitz operators and Locally Toeplitz sequences]
Let $m,n \in \N$, let $a:[0,1] \to \mathbb{C}$ and let $f \in L^1[-\pi,\pi]$. The {\em locally Toeplitz (LT) operator} is defined as the following
$n \times n$ matrix:
\begin{align*}
LT^m_n(a,f) &:= \diag_{i=1,\cdots,m}\left[a\left(\frac{i}{m}\right)T^{\left(\lfloor n/m \rfloor\right)}(f)\right]\oplus O_{n\bmod m}\\
&= \begin{bmatrix}
a\left(\frac{1}{m}\right)T^{\left(\lfloor n/m \rfloor\right)}(f)  & & & & \\
 & a\left(\frac{2}{m}\right)T^{\left(\lfloor n/m \rfloor\right)}(f) & & & \\
 & & \ddots & & \\
 & &   & a\left(1\right)T^{\left(\lfloor n/m \rfloor\right)}(f) &\\
 & & & & O_{n\bmod m}
\end{bmatrix},
\end{align*}
where $T^{\left(\lfloor n/m \rfloor\right)}(f)$ is a Toeplitz matrix of size $\lfloor n/m \rfloor$ generated by $f$ and $O_{n\bmod m}$ is the zero matrix of size $n\bmod m$.

Let $\{A^{(n)}\}_{n\in\N}$ be a matrix-sequence. We say that $\{A^{(n)}\}_{n}$ is {\em locally Toeplitz (LT) sequence} with symbol $a\otimes f$, and we write $\{A^{(n)}\}_{n} \sim_{\textnormal{LT}}  a\otimes f$, if 
$$
\{LT_n^m(a,f)\}_{n} \to \{A^{(n)}\}_{n}, \quad \mbox{a.c.s.}
$$
The functions $a$ and $f$ are, respectively, the {\em weight function} and the {\em generating function} of $\{A^{(n)}\}_{n}$.
\end{definition}
We can finally give the definition of GLT sequence.
\begin{definition}[GLT sequence]
Let $\{A^{(n)}\}_{n\in\N}$ be a matrix-sequence and let $\omega :[0,1]\times [-\pi, \pi] \to \mathbb{C}$ be a measurable function. We say that   $\{A^{(n)}\}_{n}$ is a {\em GLT sequence with symbol} $\omega$, and we write $\{A^{(n)}\}_{n} \sim_{\textnormal{GLT}} \omega$, if the following condition is met: for every $\epsilon >0$ there exists a finite number of LT sequences  $\{A^{(n)}_{(i,\epsilon)}\}_{n} \sim_{\textnormal{LT}} a_{i,\epsilon}\otimes f_{i,\epsilon}$, $i=1,\cdots, N_\epsilon$, such that
\begin{enumerate}[(1)]
	\item $\sum_{i=1}^{N_\epsilon} a_{i,\epsilon}\otimes f_{i,\epsilon} \to \omega $ in measure as $\epsilon \to 0$;
	\item $\left\{\sum_{i=1}^{N_\epsilon} A_{(i,\epsilon)}^{(n)}\right\} \to \{A_n\}_{n}$ a.c.s. as $\epsilon \to 0$.
\end{enumerate}
\end{definition}
We have the following main property (see \cite[Property \textbf{GLT 1} p. 170]{GS17}) which connects the GLT symbol with the spectral symbol of Definition \ref{def:ss_def}.
\begin{proposition}\label{prop:GLT1}
If $\{A^{(n)}\}_{n} \sim_{\textnormal{GLT}} \omega$ and $A^{(n)}$ are Hermitian, then $\{A^{(n)}\}_{n} \sim_\lambda \omega$.
\end{proposition}

\subsection{$(2\eta+1)$-points central FD discretization of the SLP \eqref{eq:S-L_general_eig_prob2}}\label{ssec:FD}

The (central) Finite Difference method basically consists in the approximation of the $j\textsuperscript{th}$-derivative $u^{(j)}(x_0)$ by means of $2\eta$ Taylor expansions, centered at $x_0 \in (a,b)$, at $2\eta$ points $\{x_{-\eta},\ldots, x_{-1}, x_{1}, \ldots, x_{\eta}\}$ with $2\eta+1\geq j$ and such that $x_{-\eta}  < \ldots < x_0 < \ldots <  x_{\eta}$. Given a standard equispaced grid $\bx=\left\{ x_j \right\}_{j=1-\eta}^{n+\eta}\subseteq \left[\bar{a},\bar{b}\right]$, with
$$
\bar{a}=a-(b-a)\frac{\eta-1}{n+1}<a,\quad \bar{b}=b + (b-a)\frac{n+\eta}{n+1}>b,\quad  x_j = a + (b-a)\frac{j}{n+1},
$$ 
let us consider a $C^1$-diffeomorphism $\tau : [a,b]\to [a,b]$ such that $\tau'(x)\neq 0$, $\tau(a)=a, \tau(b)=b$ and let us consider its piecewise $C^1$-extension $\bar{\tau}: \left[\bar{a}, \bar{b}\right] \to  \left[\bar{a}, \bar{b}\right]$ such that
$$
\bar{\tau}(x) =\begin{cases}
x & \mbox{if } x \leq a,\\
\tau(x) &\mbox{if } x \in (a,b)\\
x & \mbox{if } x\geq b.
\end{cases}
$$
By means of the piecewise $C^1$-diffeomorphism $\bar{\tau}$ we have a new grid $\bar{\bx}=\left\{ \bar{x}_j \right\}_{j=1-\eta}^{n+\eta}\subset [\bar{a},\bar{b}]$, non necessarily uniformly equispaced, with $\bar{x}_j = \bar{\tau}(x_j)$. Combining together the high-order central FD schemes in \cite{AS05,AS11,Li05}, it is not difficult to obtain the following general matrix eigenvalue problem to approximate the SLP \eqref{eq:S-L_general_eig_prob2} in the case of Dirichlet BCs:

\begin{equation*}
 \left(\prescript{a}{b}{\mathcal{L}^{(n,\eta)}_{\textnormal{dir},p(\bar{x})}} + Q^{(n)}\right)\bar{\textbf{u}}_n = \lambda^{(n)}W^{(n)}\bar{\textbf{u}}_n , \qquad \bar{\textbf{u}}_n^j=u(\bar{x}_j), \; j=1,\ldots,n,
\end{equation*}

with 

\begin{equation*}
Q^{(n)} =  \begin{bmatrix}
q_1 & 0 & &  \\
0 & q_2 & 0 & \\
& \ddots & \ddots & \ddots  \\
& & 0 & q_n
\end{bmatrix}, \qquad q_j = q\left(\bar{x}_j\right), \,\, j=1,\ldots,n,
\end{equation*}
\begin{equation*}
W^{(n)} =  \begin{bmatrix}
w_1 & 0 & &  \\
0 & w_2 & 0 & \\
& \ddots & \ddots & \ddots  \\
& & 0 & w_n
\end{bmatrix}, \qquad w_j = w\left(\bar{x}_j\right), \,\, j=1,\ldots,n,
\end{equation*}
and finally
\begin{equation*}
\prescript{a}{b}{\mathcal{L}^{(n,\eta)}_{\textnormal{dir},p(\bar{x})}} = \begin{bmatrix}
l_{1,1} & \cdots & l_{1,1+\eta} & 0 & \cdots & 0\\
l_{2,1} & l_{2,2} &  & l_{2,2+\eta}  &\cdots &0\\
 0& \ddots & \ddots & & \ddots & 0\\
 &   &  & l_{n-\eta,n-\eta}& \ddots & l_{n-\eta,n}\\
 0& & & & & \\
0&\cdots&0&l_{n,n-\eta}& \cdots & l_{n,n}
\end{bmatrix},
\end{equation*}
where, if we define the $C^0$-extension of $p(x)$ to $[\bar{a},\bar{b}]$ as
\begin{equation*}
\bar{p}(x) = \begin{cases}
p(a) & \mbox{for } x\leq a,\\
p(x) & \mbox{for } x \in (a,b)\\
p(b) & \mbox{for } x\geq b,
\end{cases}
\end{equation*}
and the element $l_{i,j}$ as 
\begin{equation*}
l_{i,j}=\frac{2\bar{p}\left(\frac{\bar{x}_j + \bar{x}_i}{2}\right)\sum_{\substack{m=i-\eta\\m\neq i,j}}^{i+\eta} \prod_{\substack{k=i-\eta\\k\neq i,j,m}}^{i+\eta}\left(\bar{x}_k - \bar{x}_i\right)}{\left(\bar{x}_j-\bar{x}_i\right)\prod_{\substack{k=i-\eta\\k\neq i,j}}^{i+\eta}\left(\bar{x}_k - \bar{x}_j\right)}, \qquad \mbox{for } i\neq j,\; \begin{cases}
i= 1,\ldots,n,\\
j=i-\eta,\ldots, i+\eta,
\end{cases}
\end{equation*}
then the generic matrix element of $\prescript{a}{b}{\mathcal{L}^{(n,\eta)}_{\textnormal{dir},p(\bar{x})}}$ is given by
\begin{equation*}
\left(\prescript{a}{b}{\mathcal{L}^{(n,\eta)}_{\textnormal{dir},p(\bar{x})}}\right)_{i,j} = \begin{cases}l_{i,j} & \mbox{for }  i\neq j,\; |i-l|\leq\eta,\; i,j=1,\ldots,n,\\
\sum_{\substack{k=i-\eta\\k\neq i}}^{i+\eta} l_{i,k} & \mbox{for } i=j,\; i=1,\ldots,n,\\
0 & \mbox{for } i\neq j,\; |i-l|>\eta,\; i,j=1,\ldots,n.
\end{cases}
\end{equation*}
With abuse of notation, we will call $\eta$ the \emph{order of approximation} of the central FD method. We have the following results.
\begin{theorem}\label{thm:FD_symbol}
In the above assumptions, for $\eta\geq 1$, defining
$$
\prescript{a}{b}{\mathcal{L}^{(n,\eta)}_{\textnormal{dir},p(\bar{x}),q(\bar{x}),w(\bar{x})}}:= \left(W^{(n)}\right)^{-1}\left(\prescript{a}{b}{\mathcal{L}^{(n,\eta)}_{\textnormal{dir},p(\bar{x})}} + Q^{(n)}\right),
$$
 it holds that
\begin{enumerate}[(i)]
	\item\label{item_spectral_conv_thm:FD_symbol} the eigenvalues $\lambda_k\left(\prescript{a}{b}{\mathcal{L}^{(n,\eta)}_{\textnormal{dir},p(\bar{x}),q(\bar{x}),w(\bar{x})}}\right)$ are real for every $k$ and
	\begin{equation*}
	\lim_{n \to \infty}\lambda_k\left(\prescript{a}{b}{\mathcal{L}^{(n,\eta)}_{\textnormal{dir},p(\bar{x}),q(\bar{x}),w(\bar{x})}}\right) = \lambda_k\left(\prescript{a}{b}{\mathcal{L}_{\textnormal{dir},p(\bar{x}),q(\bar{x}),w(\bar{x})}}\right) \qquad \mbox{for every fixed }k;
	\end{equation*}
	\item\label{item_spectral_symbol_thm:FD_symbol}
	\begin{equation}\label{eq:FD_symbol}
\left\{ (n+1)^{-2}\prescript{a}{b}{\mathcal{L}^{(n,\eta)}_{\textnormal{dir},p(\bar{x}),q(\bar{x}),w(\bar{x})}} \right\} \sim_{\lambda} \omega_{\eta} \left(x,\theta\right) \qquad (x,\theta) \in [a,b]\times[0,\pi], 
\end{equation}
where
\begin{equation*}
 \omega_{\eta} \left(x,\theta\right)=\frac{p\left(\tau(x)\right)}{w\left(\tau(x)\right)\left(\tau'(x)\right)^2(b-a)^2}f_{\eta}(\theta), 
\end{equation*}
and
\begin{equation}\label{FD_coefficients}
f_\eta (\theta) = d_{\eta,0} + 2\sum_{k=1}^\eta d_{\eta,k}\cos(k\theta), \qquad d_{\eta,k}= \begin{cases}
(-1)^{k} \frac{\eta!\eta!}{(\eta-k)!(\eta+k)!}\frac{2}{k^2} & \mbox{for } k=1,\cdots \eta,\\
-2\sum_{j=1}^\eta d_{j,\eta} &\mbox{for } k=0.
\end{cases}
\end{equation}
\end{enumerate}
\end{theorem}
\begin{proof}
See Appendix \ref{ssec:proof_FD}. 
\end{proof}


\begin{corollary}\label{cor:FD_uniform}
The function $f_\eta$ from Theorem \ref{thm:FD_symbol} is differentiable, nonnegative, monotone increasing on $[0,\pi]$ and it holds that
\begin{equation*}
f_\eta(\theta)\sim \theta^2 \quad \mbox{as }\theta \to 0, \qquad \lim_{\eta \to \infty}\sup_{\theta \in [0,\pi]}\left|f_\eta(\theta) - \theta^2\right|=0.
\end{equation*}
\end{corollary}
\begin{proof}
See Appendix \ref{ssec:proof_FD}. 
\end{proof}

\subsection{Isogeometric Galerkin discretization of the SLP \eqref{eq:S-L_general_eig_prob2} by B-slpine of degree $\eta$ and smoothness $C^{\eta-1}$}\label{ssec:IsoG}

The Galerkin discretization approach deals with the SLP in its weak formulation. For simplicity, let $q(x)\equiv0$ and $\sigma_2=\zeta_2=0$ in \eqref{eq:S-L_general_eig_prob2}. In the standard Galerkin method, fix a set of basis functions $\{\vartheta_1,\ldots,\vartheta_N\}\subset H^1_0((a,b))$ defined on the reference domain $[a,b]$ and vanishing on the boundary, and look for approximations of the exact eigenpairs $\left\{\lambda_k;u_k \right\}$ of Problem \eqref{eq:S-L_general_eig_prob2} by solving the following matrix eigenvalue problem:
\begin{equation*}\label{matrix_eig}
K_N\textbf{u}_N=\lambda_{k}^{(N)}M_N\textbf{u}_N,
\end{equation*}
where 
\begin{align*}\label{KM}
K_N=\left[\int_a^b p(x)\vartheta_j'(x)\vartheta_i'(x)dx\right]_{i,j=1}^N,\qquad M_N=\left[\int_a^b w(x)\vartheta_j(x)\vartheta_i(x)dx\right]_{i,j=1}^N,
\end{align*}
are the stiffness and mass matrices, respectively. In the isogeometric Galerkin method (IgA), the physical domain $[a,b]$ is described by a global geometry map $\tau:[a,b] \to [a,b]$ such that $\tau \in C^1([a,b])$, $\tau(a)=a, \tau(b)=b$ and $\tau'(x)\neq 0$. Within this reparametrization of the domain, the stiffness and mass matrices assume the following form
\begin{align*}\label{KM-non_uniform}
K_N=\left[\int_a^b \frac{p(\tau(x))}{|\tau'(x)|}\vartheta_j'(x)\vartheta_i'(x)dx\right]_{i,j=1}^N,\qquad M_N=\left[\int_a^b w(\tau(x))|\tau'(x)|\vartheta_j(x)\vartheta_i(x)dx\right]_{i,j=1}^N.
\end{align*}
The numerical eigenvalues are simply the eigenvalues of $L_N=M_{N}^{-1}K_{N}$. Due to the assumption of strictly positivity of $p$ and $w$, then both $K_N$ and $M_N$ are symmetric positive definite matrices for any basis functions and diffeomorphism $\tau$. In the next theorem  we collect the results of \cite[Theorem~10.15]{GS17} in a shortened way and according to our notations. For a detailed treatment of the IgA, see \cite{CHB}. 
\begin{theorem}\label{thm:Galerkin_symbol}
Let $(a,b)$ be discretized by a uniform mesh $\left\{ x_j \right\}_{j=1}^n$ of step-size $(n+1)^{-1}$ (i.e., it divides the interval $(a,b)$ into $n+1$ equispaced subintervals) and let $\tau : [a,b]\to [a,b]$ a $C^1$-diffeomorphism such that $\tau(a)=a,\tau(b)=b$ and $\tau'(x)\neq 0$ for every $x \in [a,b]$. Let $\eta\geq 1$ and let the basis functions $\{\vartheta_1,\ldots,\vartheta_N\}$ be taken as $\{B_{1,[\eta,\eta-1]},\ldots,B_{n+\eta-2,[\eta,\eta-1]}\}$, where $B_{j,[\eta,\eta-1]}$ for $j=1,\cdots,(n+1)+\eta-2$ are the B-splines of degree $\eta$ and smoothness $C^{\eta-1}([0,1])$ defined on the knot sequence
$$
\left\{ \underbrace{a,\ldots,a}_{\hbox{$\eta+1$}}, x_1,x_2, \ldots,x_{n}, \underbrace{b,\ldots,b}_{\hbox{$\eta+1$}}  \right\}.
$$

Indicating with $\psi_{[s]}$ the cardinal B-spline of degree $s$, then
\begin{enumerate}[(i)]
	\item\label{spect_conv_IgA} for every fixed $k$
	$$
	\lim_{n \to \infty}\lambda_k\left(\mathcal{L}^{(n+\eta-1,\eta)}_{\textnormal{dir},p(x),w(x)}\right)= \lambda_k\left(\mathcal{L}_{\textnormal{dir},p(x),w(x)}\right);
	$$
	\item\label{symb_IgA} it holds the following spectral distribution
	$$ 
	\left\{(n+1)^{-2}\mathcal{L}^{(n+\eta-1,\eta)}_{\textnormal{dir},p(x),w(x)}\right\}_n\sim_{\lambda}\omega_{\eta}(x,\theta), \qquad (x,\theta) \in [0,1]\times [0,\pi],
	$$ 
	where
	\begin{align}
	\omega_\eta( x,\theta)& =\frac{p(\tau(x))}{w(\tau(x)\left(\tau'(x)\right))^2(b-a)^2}\,f_\eta(\theta),\label{e-symbol}\\
	f_\eta(\theta)&=(h_\eta(\theta))^{-1}g_\eta(\theta),\label{e-symbol0}
	\end{align}
	and $g_\eta(\theta)$ and $h_\eta(\theta)$ are given by
	\begin{align*}
	g_\eta(\theta)&=-\psi_{[2\eta+1]}''(\eta+1)-2\sum_{k=1}^\eta\psi_{[2\eta+1]}''(\eta+1-k)\cos(k\theta),\\
	h_\eta(\theta)&=\psi_{[2\eta+1]}(\eta+1)+2\sum_{k=1}^\eta\psi_{[2\eta+1]}(\eta+1-k)\cos(k\theta).
	\end{align*}
\end{enumerate}
For  $\eta=1,2,3,4$, equation~\eqref{e-symbol0} gives
\begin{align*}
f_1(\theta)&=\frac{6(1-\cos\theta)}{2+\cos\theta},\\[3pt]
f_2(\theta)&=\frac{20(3-2\cos\theta-\cos(2\theta))}{33+26\cos\theta+\cos(2\theta)},\\[3pt]
f_3(\theta)&=\frac{42(40-15\cos\theta-24\cos(2\theta)-\cos(3\theta))}{1208+1191\cos\theta+120\cos(2\theta)+\cos(3\theta)},\\[3pt]
f_4(\theta)&=\frac{72(1225-154\cos\theta-952\cos(2\theta)-118\cos(3\theta)-\cos(4\theta))}{78095+88234\cos\theta+14608\cos(2\theta)+502\cos(3\theta)+\cos(4\theta)}.
\end{align*}
\end{theorem}%
\begin{proof}
For item \eqref{spect_conv_IgA} see \cite{BBCHS06,PDC}. For item \eqref{symb_IgA}, see \cite[Theorem 10.15]{GS17}.
\end{proof}

We have an analogue of Corollary \ref{cor:FD_uniform}.
\begin{corollary}\label{cor:Galerkin_uniform}
The function $f_\eta$ is differentiable, nonnegative, monotone increasing on $[0,\pi]$ and it holds that
\begin{equation*}
f_\eta(\theta)\sim \theta^2 \quad \mbox{as }\theta\to 0, \qquad \lim_{\eta \to \infty}\sup_{\theta \in [0,\pi]}\left|f_\eta(\theta) - \theta^2\right|=0.
\end{equation*}
\end{corollary}
\begin{proof}
See \cite[Theorem 1, Theorem 2 and Lemma 1]{EFGMSS18}.
\end{proof}

With abuse of notation, we will call $\eta$ the \emph{order of approximation} of the IgA method.

\subsection{The monotone rearrangement of $\omega$}\label{ss:rearrangment}
From here on, if not differently stated, we will consider $\omega : D \subset \R^m \to \R$ to be a real valued measurable function in $L^1\left(D,\mu_m\right)$.

It is always nice to deal with an univariate and monotone spectral symbol $\omega(\bx) = \omega(\theta^1)$. 
Unfortunately, in general $\omega$ is multivariate or not monotone. Nevertheless, in such cases, it is possible to consider a rearrangement $\tilde{\omega}: [0,1] \to[\inf R_\omega, \sup R_\omega]$, where $R_\omega$ is the essential range of $\omega$ as in Definition \ref{def:essential_range}, such that $\tilde{\omega}$ is univariate, monotone nondecreasing and 
\begin{equation}\label{eq:weyl-SLLN}
\lim_{n \to \infty} \frac{1}{n} \sum_{k=1}^{n}F\left((\lambda_k\left(T^{(n)}\right)\right) = \int_0^1  F(\tilde{\omega}(x))d\mu_1(x),
\end{equation} 
i.e.,
$$
\left\{T^{(n)}\right\}_n \sim_\lambda \tilde{\omega}.
$$
This can be achieved by defining 
\begin{equation}\label{eq:rearrangment}
\tilde{\omega} :  [0,1]\to [\inf R_\omega, \sup R_\omega], \qquad\tilde{\omega}(x) = \inf\left\{ t \in [\inf R_\omega, \sup R_\omega]\,:\, \phi(t)\geq x\right\}  
\end{equation}
where 
\begin{equation}\label{eq:rearrangment2}
\phi : [\inf R_\omega,\sup R_\omega] \to [0,1], \qquad \phi(t) := \frac{1}{\mu_m(D)}\mu_m \left\{\bx = (x^1,\cdots,x^\nu,\theta^{\nu+1},\cdots,\theta^m) \in D \, : \, \omega(\bx) \leq t  \right\}.
\end{equation}
Clearly, $\tilde{\omega}$ is a.e. unique, univariate, monotone strictly increasing and right-continuous, which make it a good choice for an equispaced sampling. On the other hand, $\tilde{\omega}$ could not have an analytical expression or it could be not feasible to calculate, therefore it is often needed an approximation $\tilde{\omega}_r$. Hereafter we summarize the steps presented in \cite[Example 10.2]{GS17} and \cite[Section 3]{GSERSH18} in order to approximate the eigenvalues $\lambda_k\left(T^{(n)}\right)$ by means of an equispaced sampling of the rearranged symbol $\tilde{\omega}$ (or its approximated version $\tilde{\omega}_r$). For the sake of clarity and without loss of generality, we suppose that $D=[0,1]\times [0,\pi]$ and $\omega$ continuous.  
\begin{algorithm}\label{alg:omega}{\ } \\ 
	\begin{enumerate}[1)]
		\item Fix $r \in \N$ such that $r=r(n)\geq n$, and fix the equispaced grid $\{(x_i, \theta_j)\}$ over $[0,1]\times[0,\pi]$, where $x_i= \frac{i}{r+1}$, $\theta_j = \frac{j\pi}{r+1}$ for $i,j=1,\cdots, r$;  
		\item Get the set of samplings $\left\{ \omega(x_i,\theta_j) \right\}_{i,j=1}^r$ and form a nondecreasing sequence $\{\omega_1\leq \omega_2\leq \cdots\leq \omega_{r^2} \}$;
		\item Define $\tilde{\omega}_r :[0,1] \to [\min\omega, \max\omega]$ as the piecewise linear nondecreasing function which interpolates the samples $\{\min\omega=\omega_0\leq\omega_1\leq \omega_2\leq \cdots\leq \omega_{r^2}\leq\omega_{r^2+1}=\max\omega \}$ over the nodes $\{0,\frac{1}{r^2+1},\frac{2}{r^2+1},\cdots,\frac{r^2}{r^2+1},1\}$;
		\item Sample $\tilde{\omega}_r$ over the set $\left\{ \frac{k}{n+1}\right\}_{j=1}^n$ and define
		$$
		\tilde{\omega}_{r,k}^{(n)}:= \tilde{\omega}_r\left( \frac{k}{n+1}\right).
		$$
		\item Finally, approximate the eigenvalues of $T_n$ by $\tilde{\omega}_{r,k}^{(n)}$, i.e.,
		$$
		\lambda_k\left(T_n\right) \approx \tilde{\omega}_{r,k}^{(n)}.
		$$
		Obviously, if $\tilde{\omega}$ is available then use it instead of $\tilde{\omega}_r$ and define 
		$$
		\tilde{\omega}_{k}^{(n)}:= \tilde{\omega}\left( \frac{k}{n+1}\right).
		$$
	\end{enumerate}
\end{algorithm}

As standard result in approximations of monotone rearrangements, it holds that $\| \tilde{\omega}_{r(n)} - \tilde{\omega} \|_\infty \to 0$ as $n\to \infty$, see \cite{CP79,T86}. See moreover \cite[Definition 3.1 and Theorem 3.3]{DBFS93} and \cite{Serra98}, where the monotone rearrangement were first introduced in the context of spectral symbol. We have the following limit relation.

\begin{theorem}[Discrete Weyl's law]\label{thm:discrete_Weyl_law}
Let $\tilde{\omega}: [0,1] \to [\min R_\omega,\max R_\omega]$ be the monotone rearrangement \ref{eq:rearrangment} of a spectral symbol $\omega$ of the matrix sequence $\left\{ T^{(n)}\right\}_n$. Let $\tilde{\omega}$ be piecewise Lipschitz continuous. Then 
\begin{equation}\label{eq:discrete_Weyl_law1}
\lim_{n \to \infty}\frac{\left| \left\{ k=1,\ldots,n \, : \, \lambda_k \left(T^{(n)}\right) \leq t \right\} \right|}{n} = \tilde{\omega}^{-1}(t).
\end{equation}
In particular, let $k=k(n)$ be such that $k(n)/n \sim x$ as $n\to \infty$ for a fixed $x \in [0,1]$ and $\lambda_{k(n)}\left(T^{(n)}\right)\in \tilde{\omega}\left([0,1]\right)=R_\omega$ definitely. Then
\begin{align}
&\lambda_{k(n)}\left(T^{(n)}\right) \sim \tilde{\omega}\left(\frac{k(n)}{n}\right) \qquad \mbox{as } n\to \infty,\label{eq:discrete_Weyl_law2_1}\\
&\left(\frac{k(n)}{n}, \lambda_{k(n)}\left(T^{(n)}\right)  \right) \to \left(x, \tilde{\omega}(x)\right) \qquad \mbox{as } n\to \infty. \label{eq:discrete_Weyl_law2_2}
\end{align}
\end{theorem}
\begin{proof}
By our assumptions, $\tilde{\omega}\left([0,1]\right) = R_\omega=\bigcup_{j=1}^m [a_j,b_j]$ such that $a_1=\min R_\omega$, $b_m=\max R_\omega$, and since $\tilde{\omega}$ is rectifiable then it induces a positive Borel measure on $\tilde{\omega}\left([0,1]\right)$, which we call $d\tilde{\omega}(t)$. If the subsets $[a_j,b_j]$ are disjoint, then $\tilde{\omega}^{-1}$ can be extended on all over $[\min R_\omega,\max R_\omega]$ by defining $\tilde{\omega}^{-1}(t) = \tilde{\omega}^{-1}(a_j)$ for every $t\in [b_{j-1},a_{j})$. The same applies to $d\tilde{\omega}(t)$, which can be extended to a positive measure defined on $[\min R_\omega,\max R_\omega]$, which we keep calling $d\tilde{\omega}(t)$, such that $d\tilde{\omega}\left((b_{j-1},a_{j})\right)=0$. It holds that
\begin{equation*}
\int_0^1 F\left(\tilde{\omega}(x)\right) d\mu_1(x) = \int_{\min R_\omega}^{\max R_\omega} F(t) d\tilde{\omega}(t), \qquad \forall F \in C([\min R_\omega,\max R_\omega]).
\end{equation*}
Then equation \eqref{eq:discrete_Weyl_law1} follows immediately from \cite[Theorem 7.1]{Kuipers-Niederreiter} and the definition of asymptotic distribution function mod $1$, \cite[Definition 7.1]{Kuipers-Niederreiter}, and recalling that $\left\{ T^{(n)}\right\}_n$ is weakly clustered at $R_\omega$ by Theorem \ref{thm:clustering&spectral_attraction}. Equation \eqref{eq:discrete_Weyl_law2_1} follows instead from \eqref{eq:discrete_Weyl_law1} and again by Theorem \ref{thm:clustering&spectral_attraction}, taking $t= \lambda_{k(n)}\left(T^{(n)}\right)$ and applying $\tilde{\omega}$ to both sides. 
\end{proof}

\begin{corollary}\label{cor:discrete_Weyl_law}
In the same hypothesis of Theorem \ref{thm:discrete_Weyl_law} and in presence of no outliers as in Definition \ref{def:outliers}, then the absolute error between a uniform sampling of $\tilde{\omega}$ and the eigenvalues of $T^{(n)}$ converges to zero, namely
$$
\left\|\lambda_{k}\left(T^{(n)}\right) - \tilde{\omega}\left(\frac{k}{n+1}\right)\right\|_\infty := \max_{k=1,\ldots,n} \left\{\left|\lambda_{k}\left(T^{(n)}\right) - \tilde{\omega}\left(\frac{k}{n+1}\right)\right|\right\} \to 0 \qquad \mbox{as } n\to \infty.
$$
\end{corollary}
\begin{proof}
Let us suppose that there exists a sequence $\left\{k(n) \right\}_n$ such that $\left|\lambda_{k(n)}\left(T^{(n)}\right) - \tilde{\omega}\left(\frac{k(n)}{n+1}\right)\right|\geq c>0$ for every $n$. Since $\left\{\frac{k(n)}{n}\right\}\subset [0,1]$ is bounded then there exists a convergent subsequence, which we keep calling $k(n)$, such that $k(n)/n \to x \in [0,1]$. Therefore $\left|\lambda_{k(n)}\left(T^{(n)}\right) - \tilde{\omega}\left(x\right)\right|\geq c>0$, which contradicts \eqref{eq:discrete_Weyl_law2_1}.
\end{proof}

In some sense, $\tilde{\omega}$ can be intended as the inverse cumulative distribution function of the eigenvalue distribution of $\left\{T^{(n)}\right\}_n$, and the limit relation \eqref{eq:weyl-SLLN} as the strong law for large numbers for specially chosen sequences of dependent complex-valued random variables. See \cite{LL18} as a recent survey about equidistributions from a probabilistic point of view.

\section{Application to the Euler-Cauchy differential operator}\label{sec:example}
We are now in position to begin our analysis with respect to a toy-model example. The main tasks of this section are:
\begin{itemize}
	\item to disprove that, in general, a uniform sampling of the spectral symbol $\omega(\bx)$ can provide an accurate approximation of the eigenvalues of the weighted and un-weighted discrete operators $\hat{\mathcal{L}}^{(n)}$ and $\mathcal{L}^{(n)}$, respectively;
	\item to show numerical evidences of Theorem \ref{thm:necessary_cond_for_uniformity}, i.e., that the spectral symbol $\omega(\bx)$ measures how far a matrix discretization method is from a uniform approximation of the spectrum of the continuous operator $\mathcal{L}$, in the sense of the relative error. 
	\item to show that a matrix discretization method, if paired with a suitable (non-uniform) grid can produce a uniform relative approximation of $\mathcal{L}$, as the order of approximation of the method increases.  
\end{itemize}

Let us fix $\alpha>0$ and let us consider the following SLP with Dirichlet BCs,
\begin{equation}\label{eq:Euler-Cauchy}
\begin{cases}-\partial_x\left(\alpha x^2\partial_xu(x)\right) =\lambda u(x), & x\in (1,\textrm{e}^{\sqrt{\alpha}}),\\
u(1)=u(\textrm{e}^{\sqrt{\alpha}})=0,
\end{cases}
\end{equation}
that is \eqref{eq:S-L_general_eig_prob2} with $p(x)=\alpha x^2$, $q(x)\equiv 0$, $(a,b)=(1,\textrm{e}^{\sqrt{\alpha}})$ and $\sigma_1=\zeta_1=1$,$\sigma_2=\zeta_2=0$.        
It is an Euler-Cauchy differential equation and by means of the transformation \eqref{eq:liouville_transform}, i.e., $y(x)= \int_1^x \left(\sqrt{\alpha}t\right)^{-1}dt$, Problem \eqref{eq:Euler-Cauchy} is (spectrally) equivalent to 
\begin{equation}\label{eq:Euler-Cauchy_normal_form}
\begin{cases} 
-\Delta_{\textnormal{dir}}v(y) - \frac{\alpha}{4}v(y) = \lambda v(y) & y\in (0,1),\\
v(0)=v(1)=0,
\end{cases}
\end{equation} 
which is the normal form of a Sturm-Liouville equation, with constant reaction term (or potential) $V(y)\equiv -\alpha/4$. If we write
\begin{equation}\label{eq:Eulerc}
\Eulerc(\cdot) = -\partial_x\left(\alpha x^2 \partial_x(\cdot)\right)_{|\left\{ u \in C^2([1,\textrm{e}^{\sqrt{\alpha}}])\, : \,  u(a)=u(b)=0\right\}}
\end{equation}
for the self-adjoint operator in Equation \eqref{eq:Euler-Cauchy} and 
$$-\Delta_{\textnormal{dir},\alpha}(\cdot) = -\Delta_{\textnormal{dir}}(\cdot) - \frac{\alpha}{4}(\cdot)$$ 
for the self-adjoint operator in the form of Equation \eqref{eq:Euler-Cauchy_normal_form}, then it is clear that
$$
\lambda_k \left(\Eulerc\right) = \lambda_k\left(-\prescript{1}{0}{\Delta_{\textnormal{dir},\alpha}} \right) = \lambda_k\left(-\prescript{1}{0}{\Delta_{\textnormal{dir}}} \right) +\frac{\alpha}{4} = k^2 \pi^2 +\frac{\alpha}{4}\quad \mbox{for every } k\geq 1.
$$    
For later reference, notice that
\begin{equation}\label{eq:shift}
\lim_{\alpha \to 0} \lambda_k \left(\Eulerc \right) = \lambda_k\left(-\prescript{1}{0}{\Delta_{\textnormal{dir}}}  \right) = k^2 \pi^2 \quad \mbox{for every } k\geq 1,
\end{equation}
namely, the diffusion coefficient $p(x)=\alpha x^2$ produces a constant shift of $\alpha/4$ to the eigenvalues of the unperturbed Laplacian operator with Dirichlet BCs, i.e., $-\Delta_{\textnormal{dir}}$.

We introduce the following definition.

\begin{definition}[Numerical and analytic relative error]\label{def:num_anal_error}
	Let $\Euler$ be the discrete differential operator obtained from \eqref{eq:Eulerc} by means of a generic numerical discretization method, and suppose that 
	$$
	\left\{(n+1)^{-2}\Euler\right\}_n \sim_{\lambda} \prescript{}{\alpha}{\omega}(x,\theta), \qquad (x,\theta) \in [1,\textrm{e}^{\sqrt{\alpha}}]\times [0,\pi].
	$$
	Fix $n,n',r \in \N$, with $n'>>n$ and $r=r(n)\geq n$, and compute the following quantities
	$$
	\numerr = \left| \frac{\lambda_k\left(\Euler\right)}{\lambda_k\left(\prescript{\textrm{e}^{\sqrt{\alpha}}}{1}{\mathcal{L}_{\textnormal{dir},\alpha x^2}^{(n')}}\right)} -1\right|, \qquad 
	\analerr = \left| \frac{(n+1)^2\prescript{}{\alpha}{\tilde{\omega}_{r,k}^{(n)}}}{\lambda_k\left(\prescript{\textrm{e}^{\sqrt{\alpha}}}{1}{\mathcal{L}_{\textnormal{dir},\alpha x^2}^{(n')}}\right)} -1 \right|\qquad \mbox{for }k=1,\cdots, n,
	$$
	where $\prescript{}{\alpha}{\tilde{\omega}_{r,k}^{(n)}}= \prescript{}{\alpha}{\tilde{\omega}_{r}}\left(\frac{k}{n+1}\right)$ and $\prescript{}{\alpha}{\tilde{\omega}_{r}}$ is the (approximated) monotone rearrangement of the spectral symbol $\prescript{}{\alpha}{\omega}$ obtained by the procedure described in Algorithm~\ref{alg:omega}. We call $\numerr$ the {\em numerical relative error} and $\analerr$ the {\em analytic relative error}. We say that $\prescript{}{\alpha}{\omega}$ \emph{spectrally approximates} the discrete differential operator $\Euler$ if 
	$$
	\lim_{n \to \infty}\; \left(\analerr\right) =0 \qquad \mbox{for every fixed } k.
	$$
\end{definition}

\subsection{Approximation by $3$-points central FD method on uniform grid}\label{ssec:example_uniform_3_points}

In our example, if we apply the standard central $3$-point FD scheme as in \eqref{eq:FD_symbol} with $\eta=1$ and $\tau(x)=x$, then the weighted discretization matrix $\Euler$ of Problem \eqref{eq:Euler-Cauchy} has spectral symbol 
$$
\prescript{}{\alpha}{\omega(x,\theta)}= \frac{\alpha x^2}{\left(\textrm{e}^{\sqrt{\alpha}}-1\right)^2}\left(2-2\cos(\theta)\right), \qquad \mbox{with}\quad D=[1,\textrm{e}^{\sqrt{\alpha}}]\times [0,\pi].
$$

Working with this toy-model problem in the $3$-points central FD scheme provides us a further advantage, since we can calculate the exact monotone rearrangement $\prescript{}{\alpha}{\tilde{\omega}}$, or at least a finer approximation than $\prescript{}{\alpha}{\tilde{\omega}_r}$  which does not depend on the extra parameter $r$ and which is less computationally expensive. Indeed, from equation \eqref{eq:rearrangment2} we have that
\begin{equation}\label{eq:phi_alpha}
\prescript{}{\alpha}{\phi} : \left[0,\frac{4\alpha\textrm{e}^{2\sqrt{\alpha}}}{\left(\textrm{e}^{\sqrt{\alpha}}-1\right)^2}\right] \to \left[0,1\right], \qquad 
\end{equation}
where
\begin{align*}
\prescript{}{\alpha}{\phi}(t)  &= \frac{1}{\pi\left(\textrm{e}^{\sqrt{\alpha}}-1\right)} \mu_2 \left\{(x,\theta) \in [1,\textrm{e}^{\sqrt{\alpha}}]\times [0,\pi] \, : \, \frac{\alpha x^2\left(2-2\cos(\theta)\right)}{\left(\textrm{e}^{\sqrt{\alpha}}-1\right)^2} \leq t  \right\} \\
&= \frac{1}{\pi\left(\textrm{e}^{\sqrt{\alpha}}-1\right)}\cdot \begin{cases}
\int_1^{\textrm{e}^{\sqrt{\alpha}}} 2\arcsin\left( \frac{\left(\textrm{e}^{\sqrt{\alpha}}-1\right)\sqrt{t}}{2\sqrt{\alpha}x}\right) \, d\mu_1(x) & \mbox{if } t \in \left[0, \frac{4\alpha}{\left(\textrm{e}^{\sqrt{\alpha}}-1\right)^2}\right],\nonumber\\
\pi\left(\frac{\left(\textrm{e}^{\sqrt{\alpha}}-1\right)\sqrt{t}}{2\sqrt{\alpha}}-1\right) + \int_{\frac{\left(\textrm{e}^{\sqrt{\alpha}}-1\right)\sqrt{t}}{2\sqrt{\alpha}}}^{\textrm{e}^{\sqrt{\alpha}}} 2\arcsin\left( \frac{\left(\textrm{e}^{\sqrt{\alpha}}-1\right)\sqrt{t}}{2\sqrt{\alpha}x}\right) \, d\mu_1(x) & \mbox{if } t \in \left[\frac{4\alpha}{\left(\textrm{e}^{\sqrt{\alpha}}-1\right)^2}, \frac{4\alpha\textrm{e}^{2\sqrt{\alpha}}}{\left(\textrm{e}^{\sqrt{\alpha}}-1\right)^2}\right]
\end{cases}\nonumber\\
&= \frac{1}{\pi\left(\textrm{e}^{\sqrt{\alpha}}-1\right)}\cdot \begin{cases}
\Phi\left(t,\textrm{e}^{\sqrt{\alpha}}\right) - \Phi\left(t,1\right)& \mbox{if } t \in \left[0, \frac{4\alpha}{\left(\textrm{e}^{\sqrt{\alpha}}-1\right)^2}\right],\nonumber\\
\pi\left(\frac{\left(\textrm{e}^{\sqrt{\alpha}}-1\right)\sqrt{t}}{2\sqrt{\alpha}}-1\right) + 
\left[ \Phi\left(t,\textrm{e}^{\sqrt{\alpha}}\right) - \Phi\left(t,\frac{\left(\textrm{e}^{\sqrt{\alpha}}-1\right)\sqrt{t}}{2\sqrt{\alpha}}\right)  \right]& \mbox{if } t \in \left[\frac{4\alpha}{\left(\textrm{e}^{\sqrt{\alpha}}-1\right)^2}, \frac{4\alpha\textrm{e}^{2\sqrt{\alpha}}}{\left(\textrm{e}^{\sqrt{\alpha}}-1\right)^2}\right],
\end{cases}\nonumber
\end{align*}
and where
\begin{equation*}
\Phi(t,x) = \frac{\left(\textrm{e}^{\sqrt{\alpha}}-1\right)\sqrt{t}\log\left(2\alpha x\left(\sqrt{1-\frac{\left(\textrm{e}^{\sqrt{\alpha}}-1\right)^2}{4\alpha x^2}}\right) +1\right)}{\sqrt{\alpha}} + 2x\arcsin\left( \frac{\left(\textrm{e}^{\sqrt{\alpha}}-1\right)\sqrt{t}}{2\sqrt{\alpha}x}\right).
\end{equation*}

 Clearly, $\prescript{}{\alpha}{\tilde{\omega}} = \prescript{}{\alpha}{\phi}^{-1}$, and having an analytic expression for $\prescript{}{\alpha}{\phi}(t)$, it is then possible to compute a numerical approximation of its inverse $\prescript{}{\alpha}{\phi}^{-1} : [0,1] \to \left[0,\frac{4\alpha\textrm{e}^{2\sqrt{\alpha}}}{\left(\textrm{e}^{\sqrt{\alpha}}-1\right)^2}\right]$ over the uniform grid $\left\{ \frac{k}{n+1} \right\}_{k=1}^n$, for example by means of a Newton method. This approximation of $\prescript{}{\alpha}{\phi}^{-1}$ does not depend on the extra parameter $r$ and with abuse of notation we will call it $\prescript{}{\alpha}{\tilde{\omega}}$. Henceforth, we will work with both $\prescript{}{\alpha}{\tilde{\omega}_r}$ and $\prescript{}{\alpha}{\tilde{\omega}}$, and when we will compute the analytical relative error with respect to $\prescript{}{\alpha}{\tilde{\omega}}$, we will write $\analerrExact$ without the subscript $r$.

\begin{center}
	\begin{figure}[ht]
	\includegraphics[width=15cm]{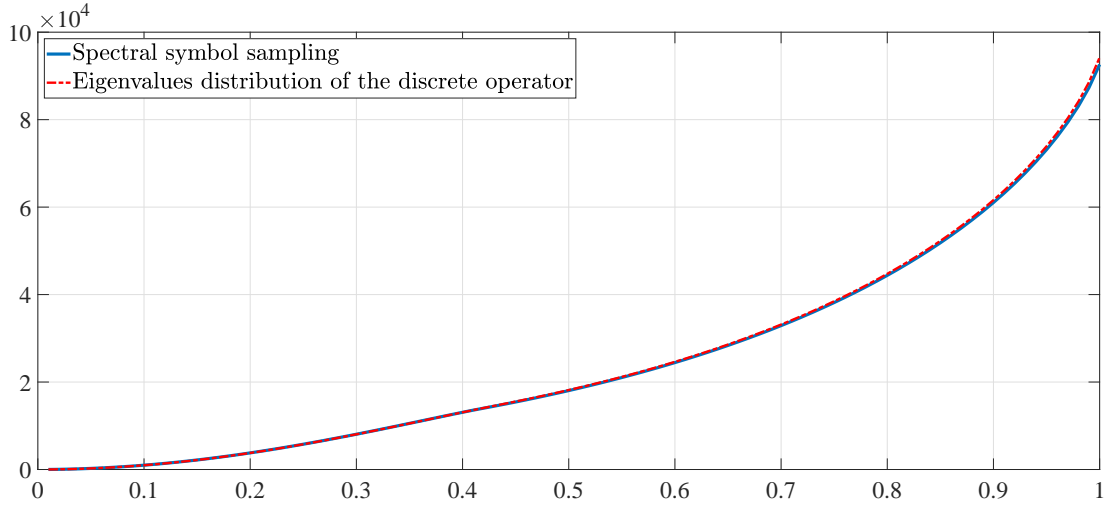}
	\captionof{figure}{For $\alpha=1$, comparison between the distribution of the first $n=10^2$ eigenvalues of the discrete operator $\Euler$ (red-dotted line) and the $n$-equispaced samples of $(n+1)^2\prescript{}{\alpha}{\tilde{\omega}_r}$ with $r=10^3$ (blue-continuous line). On the $x$-axis is reported the quotient $k/n$, for $k=1,\ldots,n$. The sovrapposition of the graphs is explained by Theorem \ref{thm:discrete_Weyl_law} and the limits \eqref{convergence_to_ICDF}, \eqref{convergence_to_ICDF2}.}\label{fig:eig_symbol_comparison}
	\end{figure}
\end{center}
Finally, let us begin our analysis. In the above Figure \ref{fig:eig_symbol_comparison} it is possible to check how an equispaced sampling of $(n+1)^2\prescript{}{\alpha}{\tilde{\omega}_r}$ seems to distribute exactly as the eigenvalues of the unweighted discrete operator $\Euler$. 

Moreover, according with equation \eqref{eq:shift}, we observe that 
\begin{equation*}
\lim_{\alpha \to 0} \prescript{}{\alpha}{\phi(t)} = \frac{2}{\pi}\arcsin\left(\frac{\sqrt{t}}{2}\right) \qquad t\in [0,4],
\end{equation*}
and so
\begin{equation*}
\lim_{\alpha \to 0}\prescript{}{\alpha}{\tilde{\omega}} (\theta) = 4\sin^2\left(\frac{\pi x}{2}\right) \qquad \theta \in [0,1],
\end{equation*}
which means that the  monotone rearrangement $\prescript{}{\alpha}{\tilde{\omega}}$ converges to the spectral symbol $\omega$ as $\alpha \to 0$, with
$$
\omega(\theta) = 4\sin^2\left(\frac{\pi \theta}{2}\right)= 2 - 2\cos\left(\theta\pi\right), 
$$
that is the spectral symbol which characterizes the differential operator $-\prescript{1}{0}{\Delta_{\textnormal{dir}}}$ discretized by means of a 3-points FD scheme.  In Figure \ref{fig:symbols_distribution_comparison} and Table \ref{table:uniform_convergence_symbol_alpha} it is visually and numerically summarized this last observation. The eigenvalues of $(n+1)^{-2}\left(-\prescript{1}{0}{\Delta_{\textnormal{dir}}^{(n)}}\right)$ are the exact sampling of $\omega(\theta)=2-2\cos(\theta\pi)$ over the uniform grid $\left\{ \frac{k}{n+1} \right\}_{k=1}^n$, see \cite[p. 154]{Smith85}.

All these remarks would suggest that $\prescript{}{\alpha}{\tilde{\omega}}$, or equivalently $\prescript{}{\alpha}{\tilde{\omega}_r}$, spectrally approximates the weighted discrete operator $\Eulerw$.  

\begin{center}
	\begin{figure}[ht]
	\includegraphics[width=15cm]{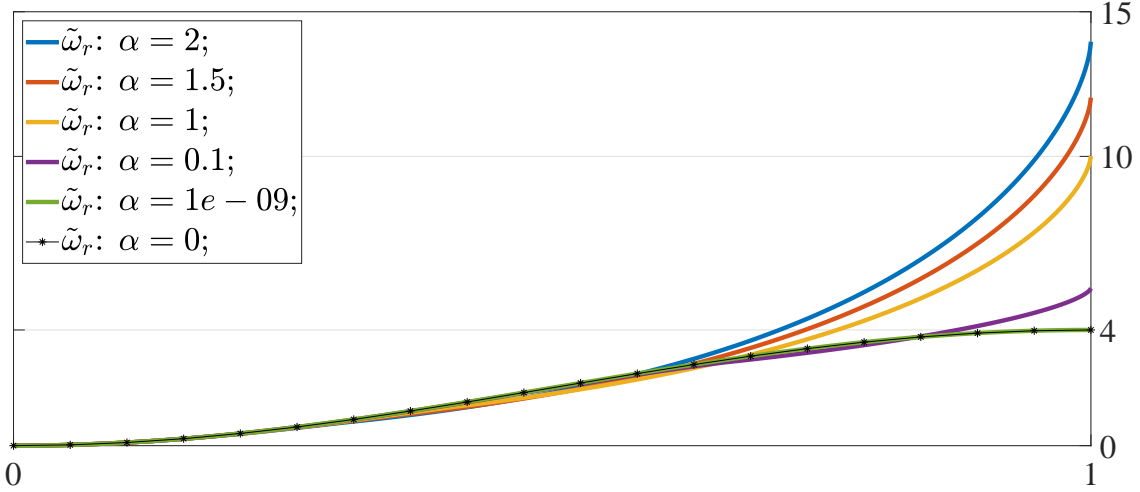}
	\captionof{figure}{Graphs of the monotone rearrangment of the spectral symbols $\prescript{}{\alpha}{\tilde{\omega}_r}$ for different values of $\alpha$ and fixed $n=10^2, r=10^3$. In a black-line with star-shaped dots it is drawn the spectral symbol $\omega(\theta)=2-2\cos(\theta\pi)$, which is the spectral symbol for the limit case $\alpha=0$.}\label{fig:symbols_distribution_comparison}
	\end{figure}
\end{center}

\begin{table}[ht]
	\centering
	\begin{tabular}{c|c|c|c|c|} 
		\cline{2-5}
		& $\alpha=1$         & $\alpha=10e-02$    & $\alpha=10e-05$    & $\alpha=10e-10$     \\ 
		\hline
		\multicolumn{1}{|l|}{$\left\|\prescript{}{\alpha}{\tilde{\omega}} - \omega\right\|_\infty$} & 5.8049  & 0.3912 &  0.0103&  4.2392e-06 \\
		\hline
	\end{tabular}\captionof{table}{for $\alpha \to 0$, we confront the sup-norm of the difference between the monotone rearrangment $\prescript{}{\alpha}{\tilde{\omega}}$ and the spectral symbol $\omega(\theta)=2-2\cos(\theta\pi)$, which characterizes the 3-point FD scheme discretization of the $1d$ Dirichlet Laplacian over $[0,1]$. The pointwise evaluation of $\prescript{}{\alpha}{\tilde{\omega}}$ and $\omega$ is made over the uniform grid $\left\{\frac{k}{n+1}\right\}_{k=1}^n$ with $n=10^3$, whereas the approximation of $\prescript{}{\alpha}{\tilde{\omega}}$ is obtained evaluating $\prescript{}{\alpha}{\phi^{-1}}$ from \eqref{eq:phi_alpha} by means of the \texttt{fzero()} function from \textsc{Matlab} r2018b.}\label{table:uniform_convergence_symbol_alpha}
\end{table}

Unfortunately, this conjecture looks to be partially proven wrong by Figure \ref{fig:analytical_err_1}. There it is shown the comparison between the graphs of the numerical relative error $\numerr$ and the analytic relative error $\analerr$, for several different increasing values of the parameter $r$. We observe a discrepancy in the analytical prediction of the eigenvalue error $\analerr$, for small $k<<n$, with respect to the numerical relative error $\numerr$. 

\begin{center}
	\begin{figure}[ht]
	\includegraphics[width=15cm]{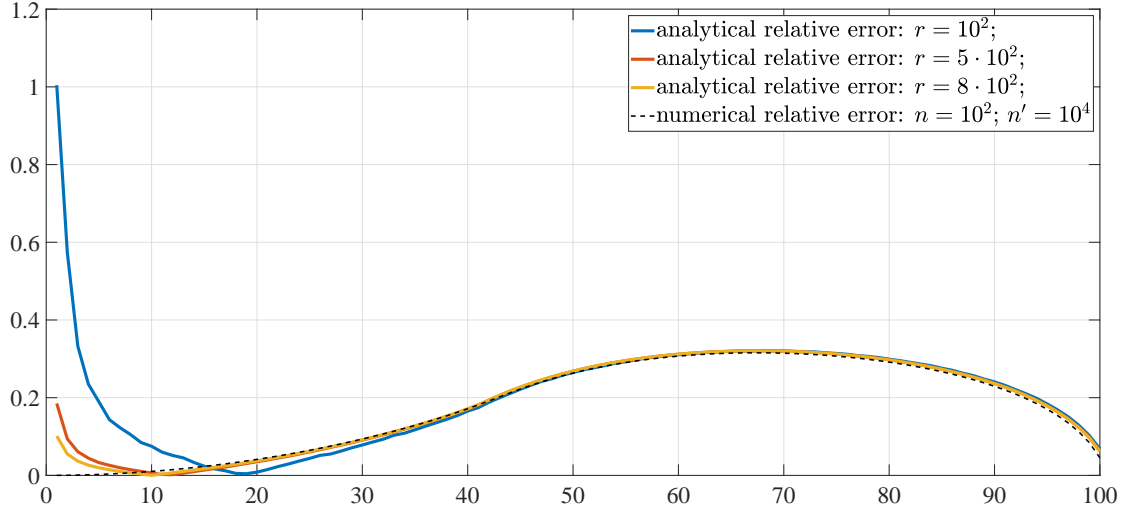}
	\captionof{figure}{Comparison between the numerical relative error $\textbf{err}_k^{(n)}$ and the analytic relative error $\tilde{\textbf{err}}_{k,r}^{(n)}$ for increasing $r=10^2,5\cdot 10^2,8\cdot10^2$. The values of $n$ and $n'$ are fixed at $10^2$ and $10^4$, respectively, and $\alpha=1$. The maximum discrepancy between the numerical relative error and the analytical relative errors is achieved for $k=1$ in all the three cases, and apparently it decreases as $r$ increases.}\label{fig:analytical_err_1}
		\end{figure}
\end{center}

In particular, the maximum discrepancy is achieved at the first eigenvalue approximation $k=1$, for every $r$. The discrepancy apparently decreases as the number of grid points $r$ increases, as well observed in \cite[Figure 48]{GSERSH18} for some test-problems in the setting of Galerkin discretization by linear $C^0$ B-spline. In that same paper, some plausible hypothesis and suggestions were advanced:
\begin{itemize}
	\item the discrepancy could depend on the fact that it has been used $\prescript{}{\alpha}{\tilde{\omega}_r}$ instead of $\prescript{}{\alpha}{\tilde{\omega}}$, and then that discrepancy should tend to zero in the limit $r\to \infty$, since $\prescript{}{\alpha}{\tilde{\omega}_r} \to \prescript{}{\alpha}{\tilde{\omega}}$.
	\item numerical instability of the analytic relative error $\analerr$ for small eigenvalues,\cite[Remark 3.1]{GSERSH18}.
	\item Change the sampling grid into an \textquotedblleft almost uniform\textquotedblright grid: see \cite[Rermark 3.2]{GSERSH18} for details.
\end{itemize}

The problem is that these hypothesis, which stem from numerical observations, cannot be validated: the descent to zero of the observed  discrepancy as $r$ increases is only apparent. Indeed, as we are going to show below, for every fixed $k$ it holds that 
$$
\left|\analerr \right| \to c_{\alpha,k}>0 \quad \mbox{as }r,n \to \infty,
$$
with $c_{\alpha,k}$ independent of $n$ and $n'$. Let us explain here the reason for this mismatch between the numerical relative error $\numerr$ and the analytic relative error $\analerr$. 

First of all, as we already observed in Subsection \ref{ss:rearrangment}, for $r\to \infty$ it holds that $\|\tilde{\omega}_r - \tilde{\omega}\|_\infty \to 0$. Then we can work directly with $\tilde{\omega}$, avoiding us to pass through its approximation $\tilde{\omega}_r$. Moreover, for fixed $n<<n'$ and $n'$ large enough, it holds that
$$
\lambda_k\left(\prescript{\textrm{e}^{\sqrt{\alpha}}}{1}{\mathcal{L}_{\textnormal{dir},\alpha x^2}^{(n')}}\right) \approx k^2\pi^2 +\alpha/4 \qquad \mbox{for } k=1,\cdots n,
$$     
see for example \cite[Lemma 7.1 and Lemma 7.2]{Gary65}. So, from here after we will consider $n'$ large enough such that the above approximation holds. We can then rewrite $\analerr$ as
$$
\analerr \approx \analerrExact = \left| \frac{(n+1)^2\prescript{}{\alpha}{\tilde{\omega}_{k}^{(n)}}}{k^2\pi^2+\frac{\alpha}{4}} -1\right|.
$$ 
From \eqref{eq:phi_alpha}, for $0\leq t\leq \frac{4\alpha}{\left(\textrm{e}^{\sqrt{\alpha}}-1\right)^2}$, it holds that
\begin{align*}
\prescript{}{\alpha}{\phi}(t) &= \frac{1}{\pi\left(\textrm{e}^{\sqrt{\alpha}}-1\right)}\int_1^{\textrm{e}^{\sqrt{\alpha}}} 2\arcsin\left( \frac{\left(\textrm{e}^{\sqrt{\alpha}}-1\right)\sqrt{t}}{2\sqrt{\alpha}x}\right) \, dx\\
&= \frac{\sqrt{t}\left(1 + \Theta(t,\alpha)\right)}{\pi},
\end{align*}
with $\Theta(t,\alpha)$ a monotone increasing function with respect to $t$ and such that $\Theta(t,\alpha) = o(t)$, for every fixed $\alpha$. In particular, for $0<t<<1$ it holds that 
$$
\prescript{}{\alpha}{\phi}(t)\approx \frac{\sqrt{t}}{\pi}.
$$
Therefore, for every $0\leq x<1$ sufficiently small, from \eqref{eq:rearrangment} we have that
$$
\prescript{}{\alpha}{\tilde{\omega}}(x) \approx \left(x\pi\right)^2 
$$  
and then, provided that  $\frac{N}{n}<<1$, for every $k=1,\ldots,N$
\begin{align}
\analerr &\approx  \left|\frac{(n+1)^2\prescript{}{\alpha}{\tilde{\omega}_{k}^{(n)}}}{k^2\pi^2+\frac{\alpha}{4}} -1\right|\nonumber\\
&= \left| \frac{(n+1)^2\prescript{}{\alpha}{\tilde{\omega}}\left(\frac{k}{n+1}\right)}{k^2\pi^2+\frac{\alpha}{4}} -1\right| \label{eq:no_uniform_grid}\\
&\approx \left|\frac{k^2\pi^2}{k^2\pi^2+\frac{\alpha}{4}} -1\right| \nonumber\\
&= 1-\frac{1}{1+\frac{\alpha}{4k^2\pi^2}} >0.\nonumber
\end{align}
If now we let $r=r(n)\to \infty$ as $n\to \infty$, then for every fixed $k\in \N$ it holds that
\begin{equation}\label{eq:c_alpha,k}
\lim_{n \to \infty}\tilde{\textbf{err}}_{k,r(n)}^{(n)} = \frac{\frac{\alpha}{4}}{k^2\pi^2+\frac{\alpha}{4}} = c_{\alpha,k}>0.
\end{equation}
We then have a lower bound for the analytic relative error which can not be avoided by refining the grid points. Of course, as $n\to \infty$, then $c_{\alpha,k} \to 0$ as $k$ increases. This is summarized below by Figure \ref{fig:analytical_error_saturation} and Table \ref{table:table_saturation}.

\begin{center}
	\begin{figure}[ht]
	\includegraphics[width=15cm]{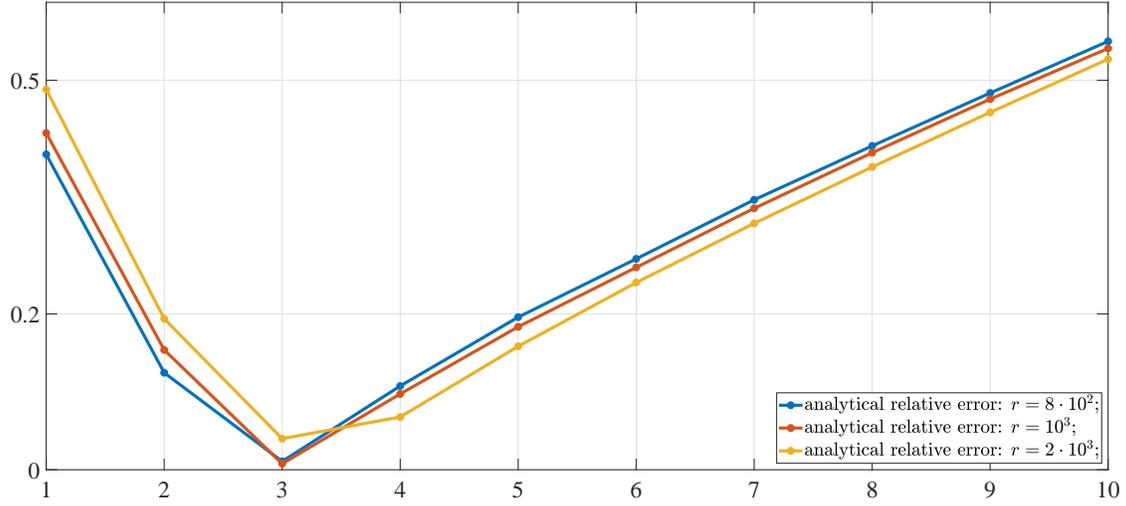}
	\captionof{figure}{Comparison between the analytic relative errors $\analerr$ for differents $r=8\cdot 10^2,10^3,2\cdot 10^3$ with $n=10^2$ and with fixed $\alpha=4\pi^2$. Observe that for $k=1$ and $k=2$ then $c_{4\pi^2,1}=0.5$ and $c_{4\pi^2,2}=0.2$, respectively, with $c_{\alpha,k}$ given in \eqref{eq:c_alpha,k}. As $r$ increases we can see that $\analerr$ tends to $c_{\alpha,k}$ for $k=1,2$.}\label{fig:analytical_error_saturation}
\end{figure}
\end{center}

\begin{table}
	\centering
	\begin{tabular}{|c|l|c|c|c|} 
		\cline{3-5}
		\multicolumn{2}{l|}{}                                                                              & \multicolumn{3}{c|}{$|\tilde{\textbf{err}}^{(n)}_k/c_{\alpha,k} -1|$ }        \\ 
		\cline{3-5}
		\multicolumn{2}{c|}{}   & \multicolumn{1}{l|}{$n=10^2$ } & $n=10^3$             & $n=10^4$              \\ 
		\hline
		\multirow{3}{*}{\rotcell{$\alpha=0.1$ }}                                 & $c_{\alpha,1}=0.0025$                                                                                         & 0.0326                & 3.3223e-04  & 3.3283e-06   \\ 
		\cline{2-5}
		& $c_{\alpha,5}=1.0131e-04$                                                                                     & 20.3811               & 0.2076      & 0.0021       \\ 
		\cline{2-5}
		& $c_{\alpha,10}=2.5330e-05$                                                                                    & 325.3811              & 3.3222      & 0.0333       \\ 
		\hline\hline
		\multirow{3}{*}{\rotcell{$\alpha=1$ }}                                   & $c_{\alpha,1}=0.0247$                                                                                         & 0.0041               & 4.1363e-05  & 4.1438e-07   \\ 
		\cline{2-5}
		& $c_{\alpha,5}=0.0010$                                                                                         & 2.5395                & 0.0259      & 2.5899e-04   \\ 
		\cline{2-5}
		& $c_{\alpha,10}=2.5324e-04$                                                                                    & 40.6422               & 0.4136      & 0.0041       \\ 
		\hline\hline
		\multirow{3}{*}{\rotcell{$\alpha=2$}} & $c_{\alpha,1}=0.0482$                                                                                         & 0.0026                & 2.6120e-05  & 2.6167e-07   \\ 
		\cline{2-5}
		& $c_{\alpha,5}=0.0020$                                                                                         & 1.6056                & 0.0163      & 1.6354e-04   \\ 
		\cline{2-5}
		& $c_{\alpha,10}=5.0635e-04$                                                                                    & 25.7979               & 0.2612      & 0.0026       \\ 
		\hline\hline
		\multirow{3}{*}{\rotcell{$\alpha=5$}} & $c_{\alpha,1}=0.1124$                                                                                         & 0.0020              & 2.0008e-05  & 2.0044e-07   \\ 
		\cline{2-5}
		& $c_{\alpha,5}=0.0050$                                                                                         & 1.2389                &0.0125      & 1.2528e-04   \\ 
		\cline{2-5}
		& $c_{\alpha,10}=0.0013$                                                                                        & 20.4017               &0.2002      & 1.2528e-04   \\
		\hline
	\end{tabular}
\captionof{table}{For every fixed $k$ and $\alpha$, the analytic relative error $\analerrExact$ converges to the lower bound $c_{\alpha,k}$ as $n$ increases, where $c_{\alpha,k}$ is given in \eqref{eq:c_alpha,k}. Observe that $\analerrExact$ seems to be monotone decreasing of order $O(n^{-2})$. The approximation of $\prescript{}{\alpha}{\tilde{\omega}}$ is obtained evaluating $\prescript{}{\alpha}{\phi^{-1}}$ from \eqref{eq:phi_alpha} by means of the \texttt{fzero()} function from \textsc{Matlab} r2018b.}\label{table:table_saturation}
\end{table}

The problem lies on the wrong informal interpretation given to the limit relation in Definition \ref{def:ss_def}, and suggested by Remark \ref{rem:ssymbol_sampling}. Indeed, that limit relation tells us that 
\begin{equation}\label{convergence_to_ICDF}
\left(\frac{k(n)}{n}, \lambda_{k(n)}\left(\Eulerw\right) \right)  \to \left(x, \prescript{}{\alpha}{\tilde{\omega}}(x) \right) \qquad \mbox{as }n\to \infty
\end{equation}
or equivalently
\begin{equation}\label{convergence_to_ICDF2}
\left(\frac{k(n)}{n}, \lambda_{k(n)}\left(\Euler\right) \right)  \to \left(x, (n+1)^2\prescript{}{\alpha}{\tilde{\omega}}(x) \right) \qquad \mbox{as }n\to \infty,
\end{equation}
for every $k(n)$ such that $k(n)/n\to x \in [0,1]$, see Theorem \ref{thm:discrete_Weyl_law}. Therefore, since $\prescript{}{\alpha}{\tilde{\omega}}\in C^1([0,1])$, it follows that 
$$
\prescript{}{\alpha}{\textbf{abs}_{r}^{(n)}}=\|\prescript{}{\alpha}{\tilde{\omega}_{r,k}^{(n)}}-\lambda_k\left(\Eulerw\right)\|_\infty\to 0 \quad \mbox{as } n\to \infty,
$$
as observed for example in \cite[Example 10.2 p. 198]{GS17}. On the contrary, a uniform sampling of the symbol $\omega$ does not necessarily provide an accurate approximation of the eigenvalues of the weighted operator $\Eulerw$, in the sense of the relative error. The uniform sampling of the symbol works perfectly only on certain subclasses of symbol functions $\omega$, but it fails in general. We loose convergence even for the absolute error as soon as we relax the hypothesis on $\omega$, requiring it to be just Lebesgue integrable over $D$, see Subsection \ref{ssec:L1}.

As a last remark, in general there does not exist an \textquotedblleft almost\textquotedblright uniform grid as well, nor in an asymptotic sense as described in \cite[Rermark 3.2]{GSERSH18}. 
Knowing the exact sampling grid $\left\{\tau\left(\frac{j}{n+1}\right)\right\}_{j=1}^n$ which guarantees $\prescript{}{\alpha}{\tilde{\omega}}$ to spectrally approximate the discrete differential operator $\Euler$ is equivalent to know the eigenvalue distribution of the continuous differential operator $\Eulerc$.

What we can say instead is that every $t \in [\min \prescript{}{\alpha}{\omega}, \max \prescript{}{\alpha}{\omega}]$ strongly attracts the spectrum of $\Eulerw$ with infinite order, see Definition \ref{def:spectral_attraction} and Theorem \ref{thm:clustering&spectral_attraction}. In particular, in presence of no outliers, it holds that  
\begin{align}\label{eq:spectral_attraction_FD_3_points}
\lim_{n \to \infty} \lambda_{n}\left(\Eulerw\right) = \max_{[0,1]\times [0,\pi]}\prescript{}{\alpha}{\omega}(x,\theta)
= \prescript{}{\alpha}{\tilde{\omega}}(1)=\frac{4\alpha\textrm{e}^{2\sqrt{\alpha}}}{\pi^2(\textrm{e}^{\sqrt{\alpha}}-1)^2},
\end{align}
and so
$$
\lim_{n \to \infty}\left|\prescript{}{\alpha}{\tilde{\textbf{err}}_{r(n),n}^{(n)}}\right| =0.
$$
In Table \ref{table:table_limit_n} we can see that \eqref{eq:spectral_attraction_FD_3_points} is confirmed.

Therefore, any numerical scheme which is characterized by a spectral symbol of the form $\prescript{}{\alpha}{\omega(x,\theta)}= \frac{p(x)}{(\textrm{e}^{\sqrt{\alpha}}-1)^2}f(\theta)$ such that
\begin{equation}\label{eq:necessary_condition_uniform}
\max_{(x,\theta)\in [1,\textrm{e}^{\sqrt{\alpha}}]\times[0,\pi]} \frac{p(x)}{(\textrm{e}^{\sqrt{\alpha}}-1)^2}f(\theta) \neq \pi^2 = \lim_{n \to \infty} \frac{\lambda_n\left(\Eulerc\right)}{(n+1)^2}
\end{equation} 
will not provide a relative uniform approximation of the eigenvalues  $\lambda_k\left(\Eulerc\right)$ of the continuous differential operator of Problem \ref{eq:Euler-Cauchy}, namely
\begin{equation}\label{eq:no_uniform_approx}
\lim_{n \to \infty}\left\{\max_{k=1,\ldots,n}\left| \frac{\lambda_k\left(\Euler\right)}{\lambda_k\left(\Eulerc\right)}  -1\right|\right\} = c >0,
\end{equation}
see Corollary \ref{cor:necessary_cond_for_uniformity}. This is the case of the method we implemented in this subsection, i.e., the 3-points uniform FD scheme, as it can be easily checked by Table \ref{table:table_limit_n} and Figure \ref{fig:eig_distribution_VS_exact}.

Actually, Theorem \ref{thm:necessary_cond_for_uniformity} says more: it says that if $\prescript{}{\alpha}{\tilde{\omega}}(x)\neq x^2\pi^2$ and in presence of no outliers, it holds that
\begin{equation}\label{eq:limit_relative_error_FD_3_point}
\lim_{n \to \infty}\left\{\max_{k=1,\ldots,n}\left| \frac{\lambda_k\left(\Euler\right)}{\lambda_k\left(\Eulerc\right)}  -1\right|\right\} = \max_{x\in[0,1]} \left| \frac{\prescript{}{\alpha}{\tilde{\omega}}(x)}{x^2\pi^2} -1\right|.
\end{equation}

In Figure \ref{fig:eig_distribution_VS_exact} and Table \ref{table:maximum_rel_error} it is validated numerically \eqref{eq:limit_relative_error_FD_3_point}, i.e., the thesis of Theorem \ref{thm:necessary_cond_for_uniformity} set in this specific case.

\begin{table}[ht]
	\centering
	\begin{tabular}{|l|c|c|c|c|} 
		\cline{3-5}
		\multicolumn{1}{l}{}           &                                                                                               & \multicolumn{1}{c|}{$n=10^2$ } & $n=10^3$  & $n=3\cdot10^3$   \\ 
		\hline
		\multirow{2}{*}{$\alpha=0.5$ } &  $\prescript{}{\alpha}{\tilde{\textbf{err}}_{r,n}^{(n)}}$                                     & 0.0210                         & 0.0036    & 0.0016           \\ 
		\cline{2-5}
		& $|\frac{(n+1)^2\lambda^{(n)}_n/\lambda_n}{\prescript{}{\alpha}{\tilde{\omega}(1)}/\pi^2}-1|$  & 0.0482                         & 0.0128    & 0.0065           \\ 
		\hline\hline
		\multirow{2}{*}{$\alpha=1.2$ } & $\prescript{}{\alpha}{\tilde{\textbf{err}}_{r,n}^{(n)}}$                                      & 0.0249                         & 0.0044    & 0.0020           \\ 
		\cline{2-5}
		& $|\frac{(n+1)^2\lambda^{(n)}_n/\lambda_n}{\prescript{}{\alpha}{\tilde{\omega}(1)}/\pi^2}-1|$  & 0.0615                         & 0.0157    & 0.0079           \\ 
		\hline\hline
		\multirow{2}{*}{$\alpha=4$ }   &  $\prescript{}{\alpha}{\tilde{\textbf{err}}_{r,n}^{(n)}}$                                     & 0.0297                         & 0.0052    & 0.0024           \\ 
		\cline{2-5}
		& $|\frac{(n+1)^2\lambda^{(n)}_n/\lambda_n}{\prescript{}{\alpha}{\tilde{\omega}(1)}/\pi^2}-1|$  & 0.0765                         & 0.0191    & 0.0095           \\
		\hline
	\end{tabular}
	\captionof{table}{In this table it is possible to check that for every fixed $\alpha$, the $n$-th term of the analytic relative error $\prescript{}{\alpha}{\tilde{\textnormal{err}}_{r,n}^{(n)}}$ converges to zero as $n$ increases, and that the quotient between $(n+1)^2\lambda_n\left(\Eulerw\right)$ and $\lambda_n\left(\Eulerc\right)$ converges to $\max_{(x,\theta)\in [1,\textrm{e}^{\sqrt{\alpha}}]\times[0,\pi]} \prescript{}{\alpha}{\omega(x,\theta)}/\pi^2 = \prescript{}{\alpha}{\tilde{\omega}(1)}/\pi^2$. In this case we choose $r=n$ to keep computational costs at minimum.}\label{table:table_limit_n}
\end{table}

\begin{table}\vspace*{-2cm}
	\centering
	\begin{tabular}{|l|c|c|c|c|} 
		\cline{3-5}
		\multicolumn{1}{c}{}           &                                                                                                        & $n=10^2$              & $n=10^3$              & $n=5\cdot10^3$         \\ 
		\hline
		\multirow{2}{*}{$\alpha=0.5$ } & $|\frac{\max|\lambda^{(n)}_k/\lambda_k|}{\max| \prescript{}{\alpha}{\tilde{\omega}}(x)/x^2\pi^2|}-1|$  & 0.0104                & 0.0010                & 2.0853e-04             \\ 
		\cline{2-5}
		& $\bar{k}/n$                                                                                            & 0.7900                & 0.7880                & 0.7878                 \\ 
		\hline\hline
		\multirow{2}{*}{$\alpha=1$ }   & $|\frac{\max|\lambda^{(n)}_k/\lambda_k|}{\max| \prescript{}{\alpha}{\tilde{\omega}}(x)/x^2\pi^2|}-1|$  & 0.0158                & 0.0016                & 3.1754e-04             \\ 
		\cline{2-5}
		& $\bar{k}/n$                                                                                            & 0.6700                & 0.6680                & 0.6676                 \\ 
		\hline\hline
		\multirow{2}{*}{$\alpha=1.2$ } & $|\frac{\max|\lambda^{(n)}_k/\lambda_k|}{\max| \prescript{}{\alpha}{\tilde{\omega}}(x)/x^2\pi^2|}-1|$  & 0.0180                & 0.0018                & 3.6226e-04             \\ 
		\cline{2-5}
		& $\bar{k}/n$                                                                                            & 0.64                  & 0.6310                & 0.6302                 \\ 
		\hline\hline
		\multirow{2}{*}{$\alpha=3$ }   & $|\frac{\max|\lambda^{(n)}_k/\lambda_k|}{\max| \prescript{}{\alpha}{\tilde{\omega}}(x)/x^2\pi^2|}-1|$  & 0.0518                & 0.0097               &0.0032             \\ 
		\cline{2-5}
		& $\bar{k}/n$                                                                                            & 1               & 1              & 1                 \\ 
		\hline
	\end{tabular}\caption{In this table we check numerically the validity of Theorem \ref{thm:necessary_cond_for_uniformity} for different values of $\alpha$ and $n$. It can be seen that for every $\alpha$, as $n$ increases then the relative error between $\max_{k=1,\ldots,n}\left|\lambda_k\left(\Euler\right)/\lambda_k\left(\Eulerc\right)\right|$ and $\max_{x\in[0,1]}\left|\prescript{}{\alpha}{\tilde{\omega}}(x)/x^2\pi^2\right|$ decreases, confirming \eqref{eq:limit_relative_error_FD_3_point}. In the table is reported as well the quotient $\bar{k}/n$, where $\bar{k}$ is the $k$-th eigenvalue which achieves the maximum relative error between $\lambda_k\left(\Euler\right)$ and $\lambda_k\left(\Eulerc\right)$. We can notice that $\bar{k}/n$ is always bounded and it tends to a finite value in $(0,1]$ as $n$ increases. The approximation of $\prescript{}{\alpha}{\tilde{\omega}}$ is obtained evaluating $\prescript{}{\alpha}{\phi^{-1}}$ from \eqref{eq:phi_alpha} by means of the \texttt{fzero()} function from \textsc{Matlab} r2018b.}\label{table:maximum_rel_error}
\end{table}

\begin{center}
	\begin{figure}\vspace{-1cm}
		\includegraphics[width=15cm]{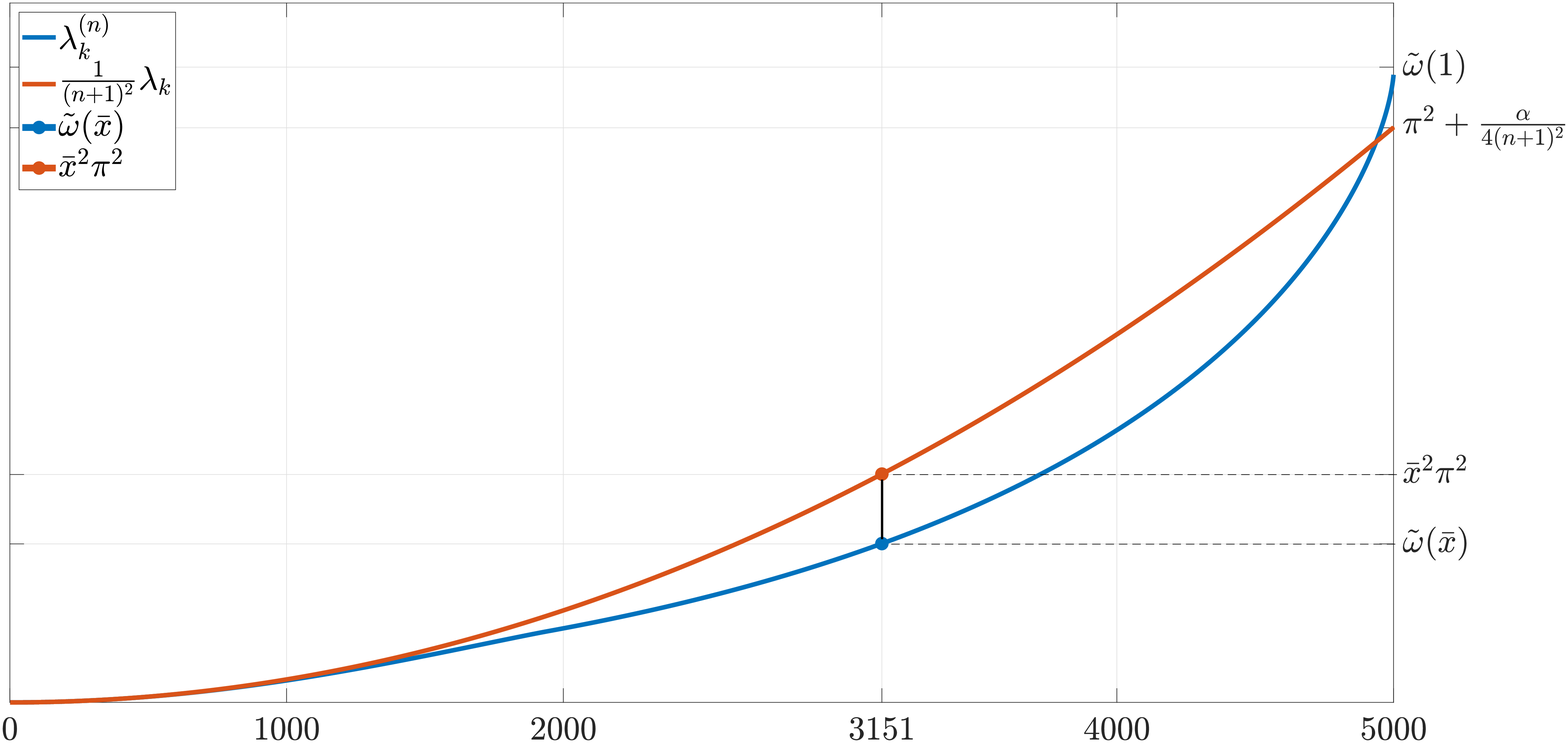}
		\captionof{figure}{For $\alpha=1.2$ and $n=5\cdot10^3$, comparison between the eigenvalues distribution of the weighted discrete differential operator $\Eulerw$ and the exact eigenvalues of the continuous differential operator $\Eulerc$, weighted by $(n+1)^2$. As observed in \eqref{eq:spectral_attraction_FD_3_points}, $\lambda_n^{(n)}\left(\Eulerw\right)\approx \prescript{}{\alpha}{\tilde{\omega}(1)}= \frac{4\alpha\textrm{e}^{2\sqrt{\alpha}}}{(\textrm{e}^{\sqrt{\alpha}}-1)^2}$. Moreover, the maximum relative error \eqref{eq:no_uniform_approx} is obtained for $\bar{k}\approx 3151$, which corresponds to the maximum relative error of $|\prescript{}{\alpha}{\tilde{\omega}}(x)/x^2\pi^2 -1|$ achieved at $\bar{x}\approx 0.6301 \approx \bar{k}/n$.}\label{fig:eig_distribution_VS_exact}
	\end{figure}
\end{center}

\FloatBarrier
\subsection{Discretization by $(2\eta+1)$-points central FD method on non-uniform grid}\label{ssec:FD_nonuniform}
Clearly, everything said in the preceding Subsection \ref{ssec:example_uniform_3_points} remains valid even if we increase the order of accuracy of the FD method, namely, the spectral symbol $\omega_\eta$ of equation \eqref{eq:FD_symbol} does not spectrally approximate the discrete differential operator $\prescript{\textrm{e}^{\sqrt{\alpha}}}{1}{\mathcal{L}_{\textnormal{dir},\alpha \tau(x)^2}^{(n,\eta)}}$, in the sense of the relative error, for any $\eta\geq 1$. See Figure \ref{fig:analytic_num_comparison_eta} in relation with Figure \ref{fig:analytical_err_1}.

\begin{center}
	\begin{figure}[ht]
		\centering
		\subfloat[Uniform central FD with $\eta=4$]{
			\includegraphics[height=33mm]{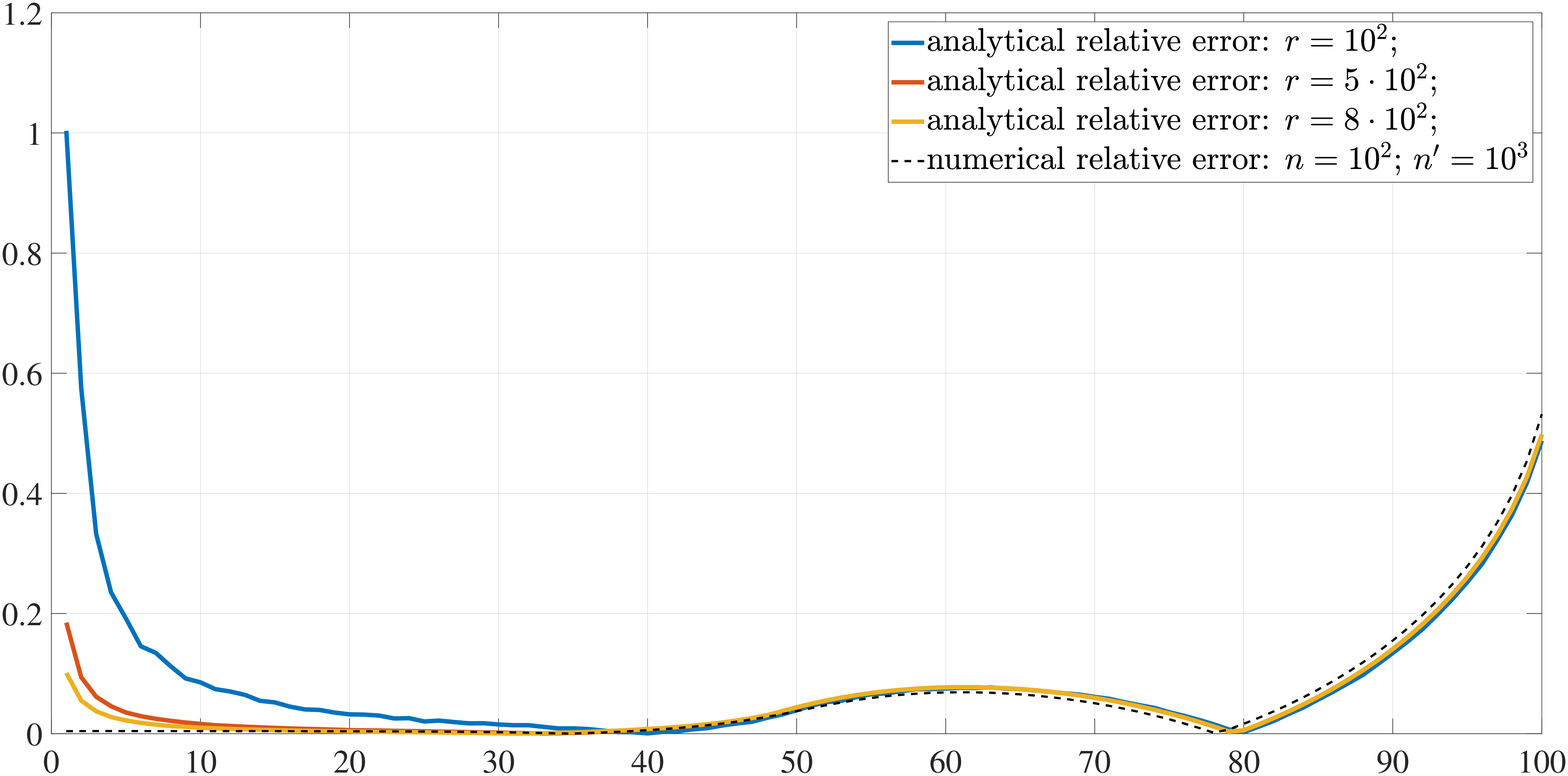}  
			\label{subfig:eta_4}
		}
		\subfloat[Uniform central FD with $\eta=8$]{
			\centering
			\includegraphics[height=33mm]{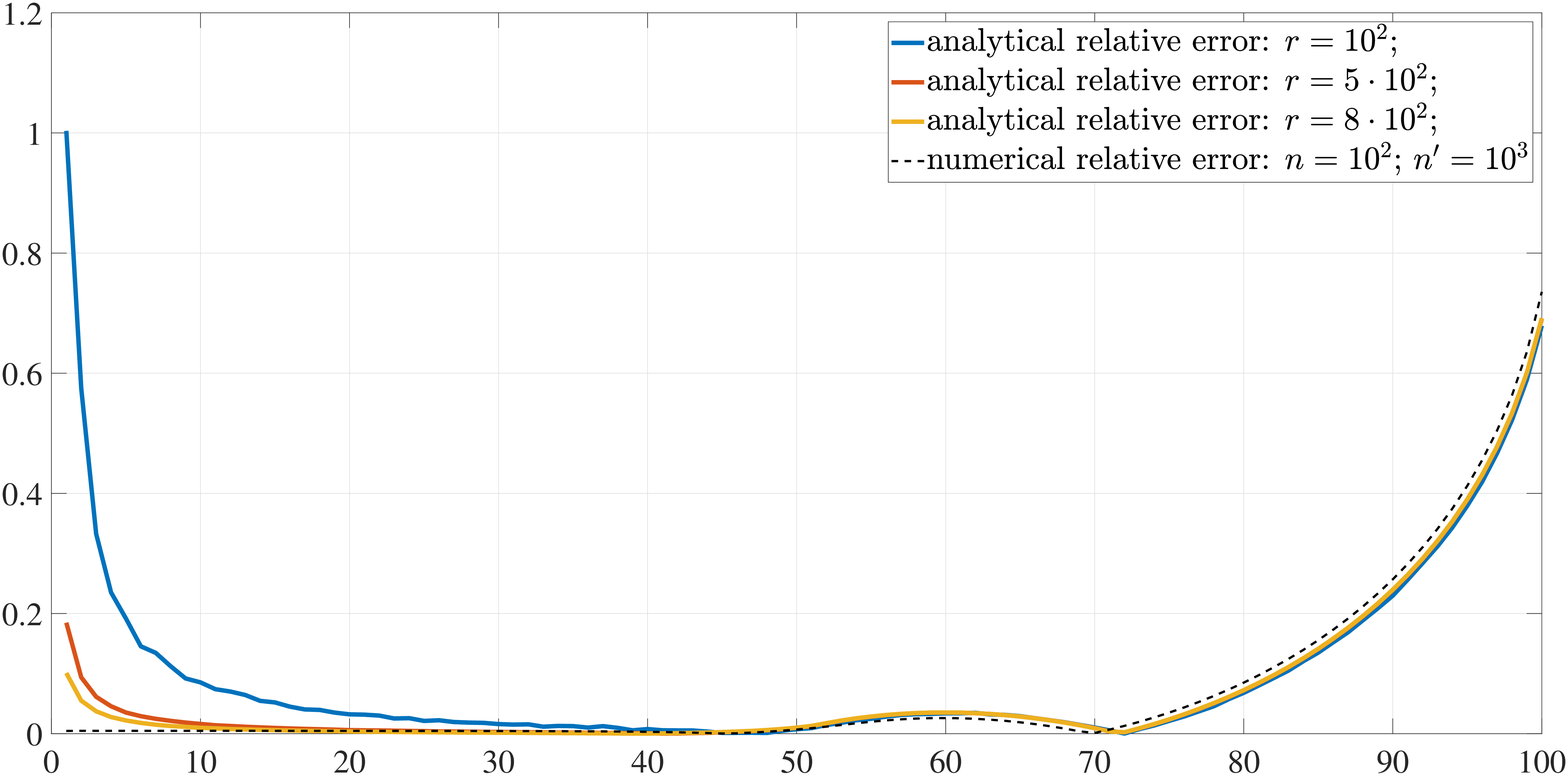}  
			\label{subfig:eta_8}
		}
		\captionof{figure}{Graphic comparison between the numerical relative error $\textbf{err}_k^{(n)}$ and the analytic relative error $\tilde{\textbf{err}}_{k,r}^{(n)}$ as in Definition \ref{def:num_anal_error}, for increasing $r=10^2,5\cdot 10^2,8\cdot10^2$. The values of $n$ and $n'$ are fixed at $10^2$ and $10^3$, respectively, and $\alpha=1$. In Figure \ref{subfig:eta_4} it has been used a $5$-points central FD discretization on uniform grid, while in Figure \ref{subfig:eta_8} it has been used a $9$-points central FD discretization on uniform grid. As it happened in Figure \ref{fig:analytical_err_1}, it is displayed an evident discrepancy between the numerical relative error and the analytical relative errors for the first eigenvalues, which is explained by \eqref{eq:c_alpha,k} and Proposition \ref{prop:non_good_approximation}.}\label{fig:analytic_num_comparison_eta}
	\end{figure}
\end{center}

What is interesting instead is to change the sampling grid and to increase the order of accuracy $\eta$ of the FD discretization method. Indeed, as it was observed in the relations \eqref{eq:necessary_condition_uniform}, \eqref{eq:no_uniform_approx} and \eqref{eq:limit_relative_error_FD_3_point}, it is not possible to achieve a relative uniform approximation of the eigenvalues  $\lambda_k\left(\Eulerc\right)$ if $\prescript{}{\alpha}{\tilde{\omega}_\eta}(x) \neq x^2\pi^2$. From \eqref{eq:FD_symbol}, for every $\eta\geq 1$ it is easy to check that $\max_{(x,\theta)\in[a,b]\times [0,\pi]}\prescript{}{\alpha}{\omega_{\eta}}(x,\theta)= \prescript{}{\alpha}{\tilde{\omega}_\eta}(1)\neq \pi^2$, and so we do not have any improvement by just increasing the order of accuracy $\eta$. But let us observe that:
\begin{itemize}
	\item from Corollary \ref{cor:FD_uniform} and equation \eqref{eq:FD_symbol}, 
	$$
	\lim_{\eta\to \infty} \omega_{\eta}(x,\theta) = \frac{\alpha \tau(x)^2}{\left(\tau'(x)\right)^2\left(\textrm{e}^{\sqrt{\alpha}}-1\right)^2} \theta^2 \qquad \mbox{for every } (x,\theta) \in [1,\textrm{e}^{\sqrt{\alpha}}]\times [0,\pi];
	$$
	\item $\tau(y) = \textrm{e}^{\sqrt{\alpha}y}$ is such that $\frac{\alpha \tau(y)^2}{\left(\tau'(y)\right)^2}\equiv 1$.
\end{itemize}

In some sense, the spectral symbol $\omega_{\eta}$ suggests us to change the uniform grid
$$
\left\{x_j\right\}_{j=1}^n=\left\{1+ \left(\textrm{e}^{\sqrt{\alpha}}-1\right)\frac{j}{n+1}\right\}_{j=1}^n \subset [1,\textrm{e}^{\sqrt{\alpha}}]
$$
by means of the diffeomorphism induced by the Liouville transformation. Indeed, from \eqref{eq:liouville_transform} we have that
\begin{equation*}
y(x)= \frac{\log(x)}{\sqrt{\alpha}} \quad \mbox{for } x \in [1,\textrm{e}^{\sqrt{\alpha}}], \qquad x(y)= \textrm{e}^{\sqrt{\alpha}y} \quad \mbox{for } y \in [0,1],
\end{equation*}  
and therefore we can construct a $C^1$-diffeomorphism $\tau: [1,\textrm{e}^{\sqrt{\alpha}}]\to [1,\textrm{e}^{\sqrt{\alpha}}]$ such that
\begin{equation*}
\tau :  [1,\textrm{e}^{\sqrt{\alpha}}] \xrightarrow[]{\tau_1} [0,1] \xrightarrow[]{\tau_2}  [1,\textrm{e}^{\sqrt{\alpha}}],
\end{equation*}
\begin{equation*}
\tau_1(x) = \frac{1}{\textrm{e}^{\sqrt{\alpha}}-1}\left(x-1\right), \qquad \tau_2(y)= \textrm{e}^{\sqrt{\alpha}y}.
\end{equation*}
The new non-uniform grid is then given by
\begin{equation}\label{eq:non_uniform_grid}
\left\{\bar{x}_j\right\}_{j=1}^n = \left\{\tau(x_j)\right\}_{j=1}^n,
\end{equation}
and it holds that
\begin{equation}\label{eq:nec_condition}
\lim_{\eta\to \infty}\prescript{}{\alpha}{\omega_{\eta}}(\tau(x),\theta) = \lim_{\eta\to \infty} f_\eta(\theta)=\theta^2 \qquad \theta \in [0,\pi],
\end{equation}
and therefore
$$
\lim_{\eta\to \infty}\prescript{}{\alpha}{\tilde{\omega_{\eta}}}(x) =x^2\pi^2 \qquad x \in [0,1].
$$
Of course, this last equality is not enough to guarantee that 
\begin{equation}\label{eq:uniform_approx}
\lim_{n,\eta \to \infty}\left\{\max_{k=1,\ldots,n}\left| \frac{\lambda_k\left(\prescript{\textrm{e}^{\sqrt{\alpha}}}{1}{\mathcal{L}_{\textnormal{dir},\alpha \tau(x)^2}^{(n,\eta)}}\right)}{\lambda_k\left(\Eulerc\right)}  -1\right|\right\} =0,
\end{equation}
since 
$$
 \max_{x\in[0,1]} \left| \frac{\prescript{}{\alpha}{\tilde{\omega}_\eta}(x)}{x^2\pi^2} -1\right|=0
$$
is only a necessary condition, see Theorem \ref{thm:necessary_cond_for_uniformity}. Nevertheless, from figures \ref{fig:comparison_non_uniform_FD}, \ref{fig:comparison_FD_err} and Table \ref{tab:uniform_approx} we can see that \eqref{eq:uniform_approx} seems to be validated numerically. This would suggest that condition \eqref{eq:nec_condition} becomes sufficient if paired by a suitable rearranged (non-uniform) grid and an increasing refinement of the method order. This phenomenon is not confined to this specific case but it occurs in several many different other cases we tested. More investigations will be carried on in future works. 

Finally, in Table \ref{table:maximum_rel_error_FD_eta} we check numerically the validity of Theorem \ref{thm:necessary_cond_for_uniformity}.

\begin{figure}[ht]
	\centering
	\subfloat[Uniform central FD with $\eta=1$]{
		\includegraphics[width=8cm]{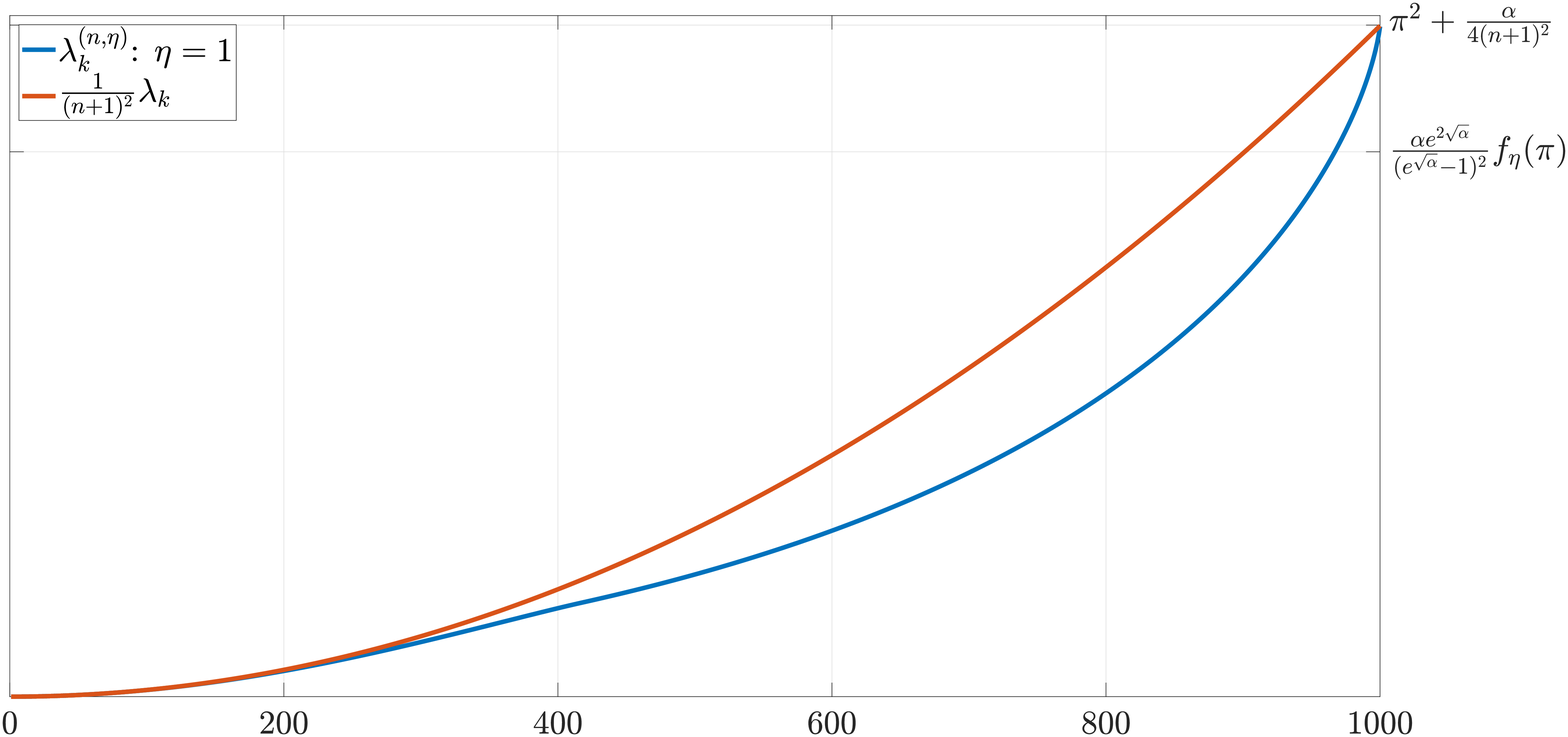}  
		\label{subfig:unif_FD_eta_1}
	}
	\subfloat[Uniform central FD with $\eta=15$]{
		\centering
		\includegraphics[width=8cm]{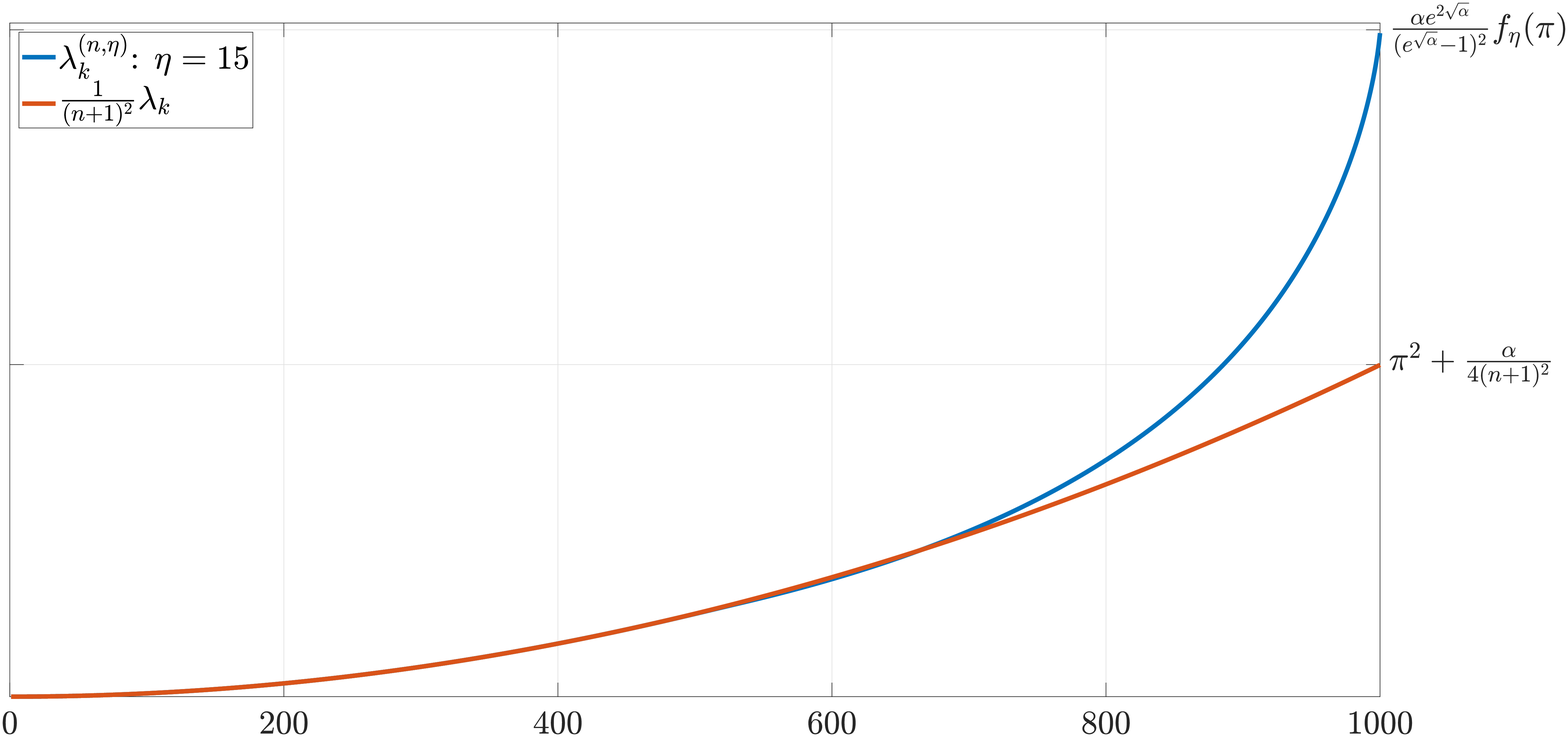}  
		\label{subfig:unif_FD_eta_15}
	}

\centering
\subfloat[Non-uniform central FD with $\eta=1$]{
	\includegraphics[width=8cm]{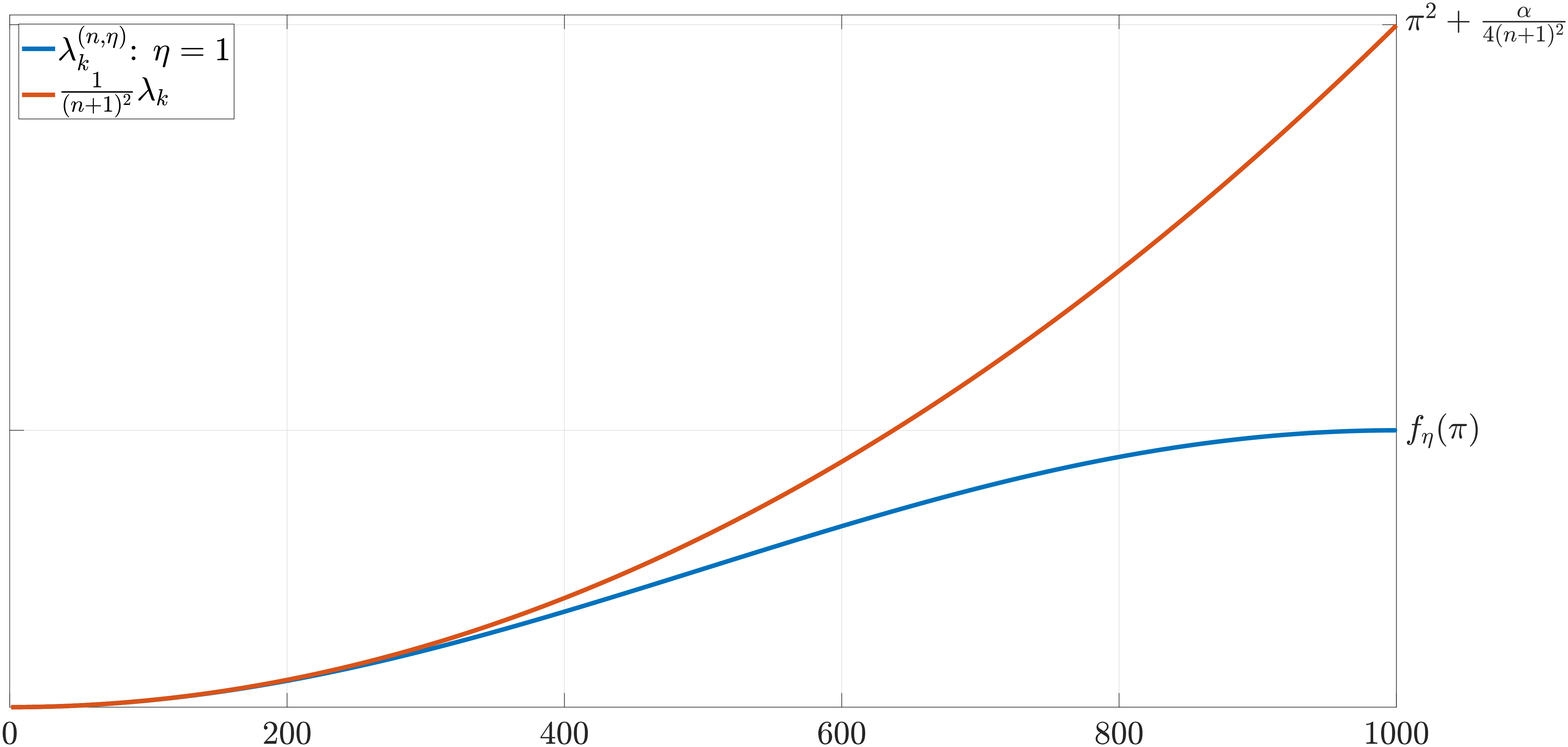}  
	\label{subfig:non_unif_FD_eta_1}
}
\subfloat[Non-uniform central FD with $\eta=15$]{
	\centering
	\includegraphics[width=8cm]{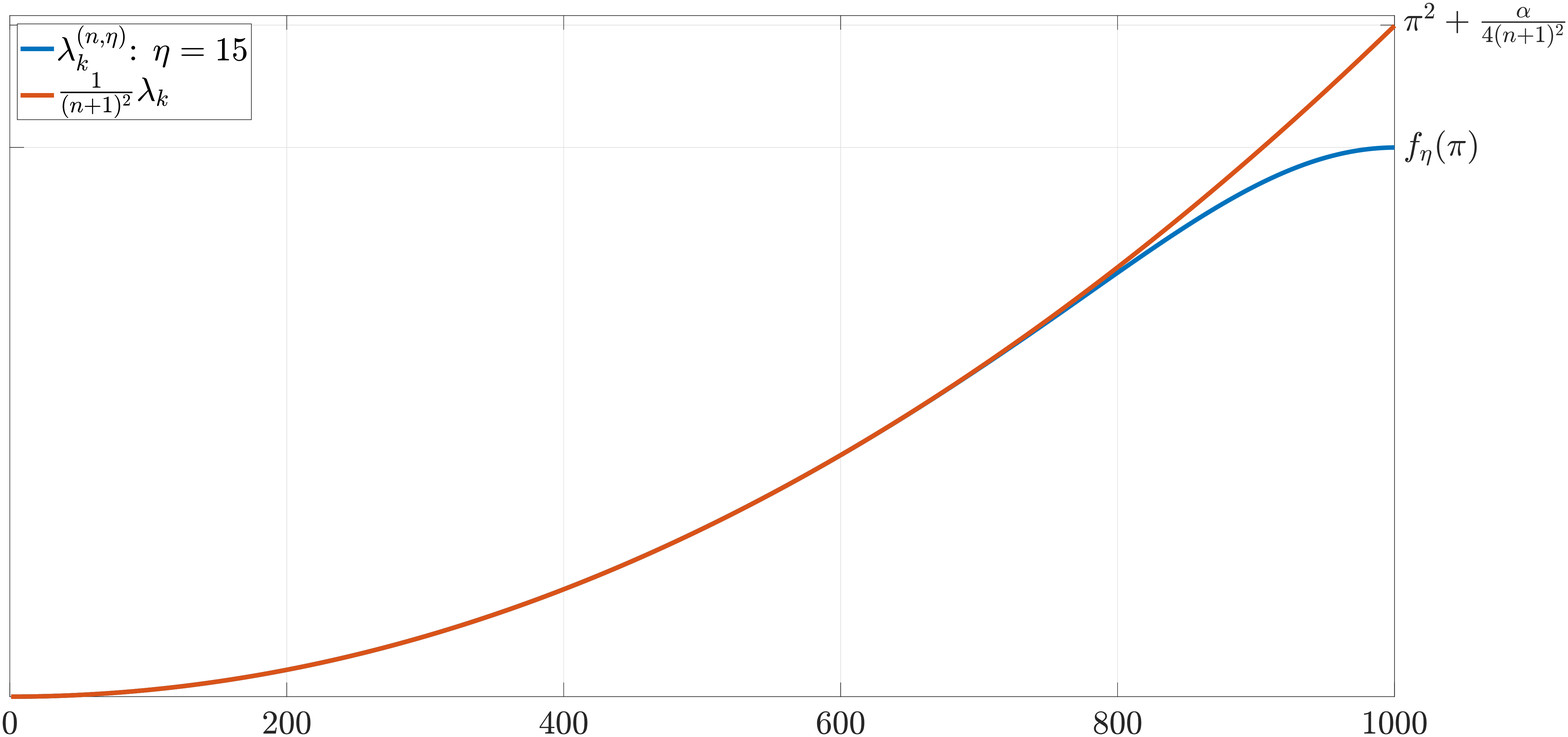}  
	\label{subfig:non_unif_FD_eta_15}
}
\captionof{figure}{Graphic comparison between the eigenvalues distribution of the weighted discrete differential operators $\prescript{\textrm{e}^{\sqrt{\alpha}}}{1}{\hat{\mathcal{L}}_{\textnormal{dir},\alpha \tau(x)^2}^{(n,\eta)}}$ obtained by means of $(2\eta+1)$-points central FD discretization on uniform and non-uniform grids. The parameters $\alpha$ and $n$ are fixed, with $\alpha=1$ and $n=10^3$, while $\eta$ changes. Let us observe that in figures \ref{subfig:non_unif_FD_eta_1}, \ref{subfig:non_unif_FD_eta_15}, i.e., in the case of central FD discretization on the non-uniform grid given by \eqref{eq:non_uniform_grid}, the graph of the eigenvalue distribution seems to converge uniformly to the graph of the exact eigenvalues $(n+1)^{-2}\lambda_k$, as $\eta$ increases. The same phenomenon does not happen in the case of central FD discretization on uniform grid, as it is clear from figures \ref{subfig:unif_FD_eta_1},\ref{subfig:unif_FD_eta_15}. See Table \ref{tab:uniform_approx} for a numerical comparison of the maximum relative errors.}\label{fig:comparison_non_uniform_FD}
\end{figure}
\begin{figure}[ht]
	\centering
	\subfloat[Relative errors for uniform central FD]{
		\includegraphics[width=8cm]{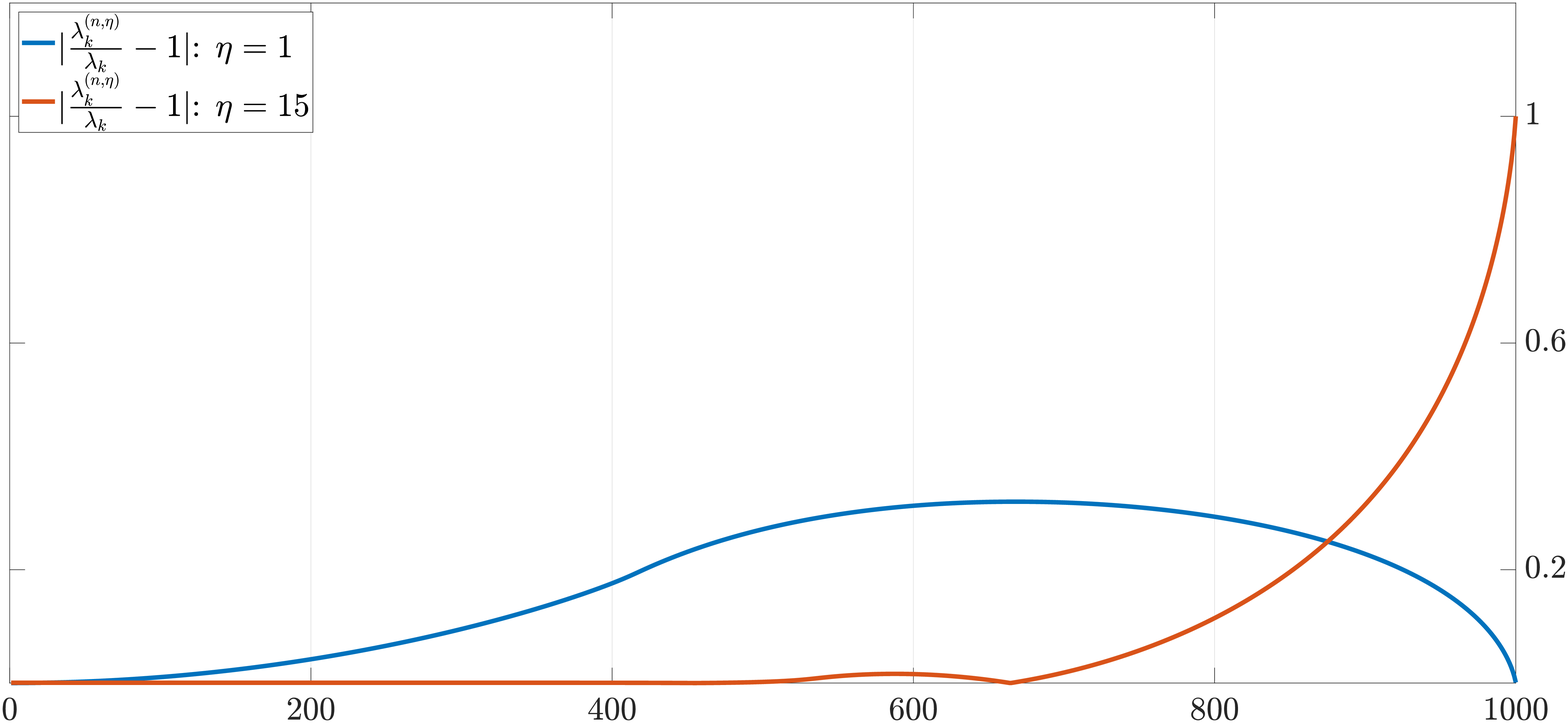}  
		\label{subfig:FD_uniformGrid_error}
	}
	\subfloat[Relative errors for non-uniform central FD]{
		\centering
		\includegraphics[width=8cm]{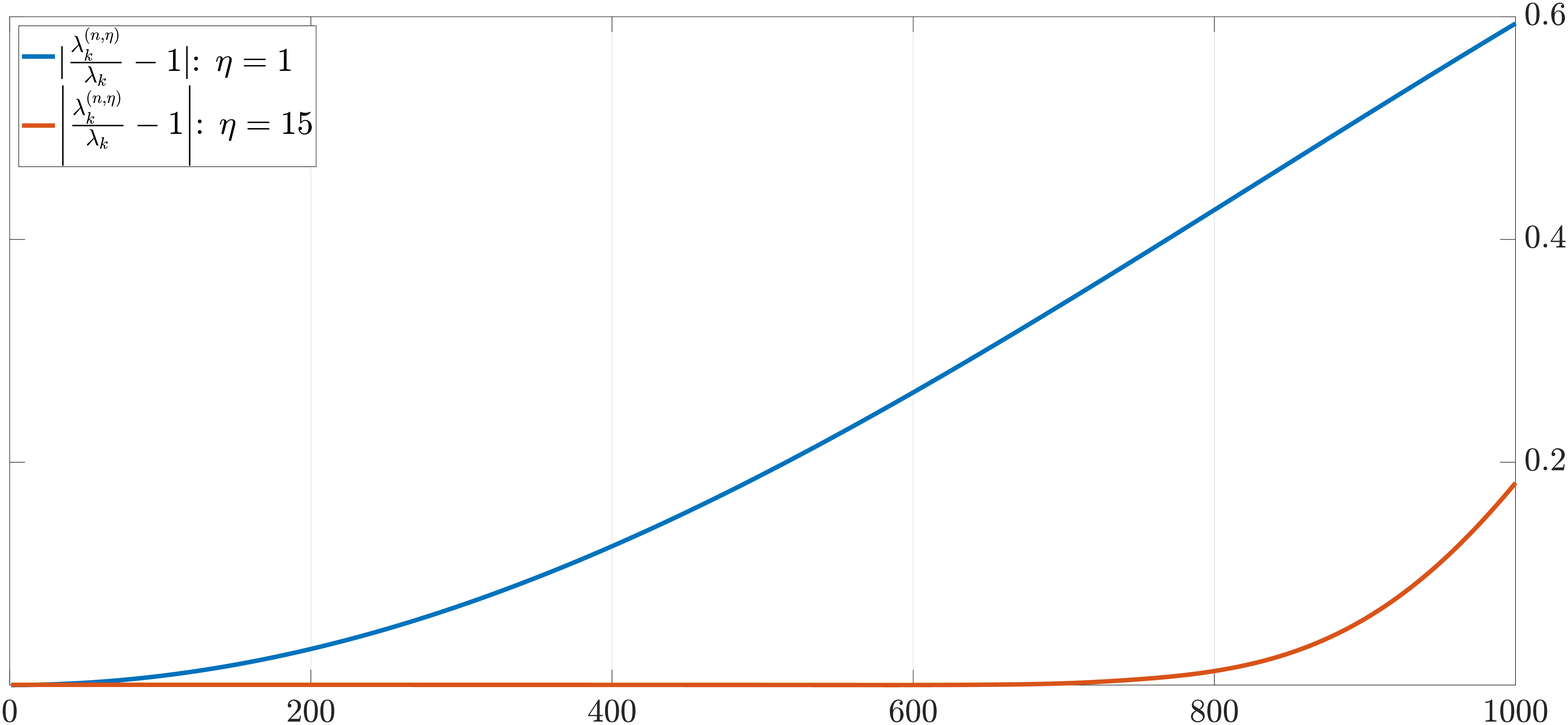}  
		\label{subfig:FD_NonuniformGrid_error}
	}
	\captionof{figure}{Graphic comparison between the eigenvalues relative errors of the discrete differential operators $\prescript{\textrm{e}^{\sqrt{\alpha}}}{1}{\mathcal{L}_{\textnormal{dir},\alpha \tau(x)^2}^{(n,\eta)}}$ obtained by means of $(2\eta+1)$-points central FD discretization on uniform and non-uniform grids, for different values of $\eta$. The parameters $\alpha$ and $n$ are fixed, with $\alpha=1$ and $n=10^3$. In Subfigure \ref{subfig:FD_uniformGrid_error}, where it is used the standard uniform grid, we notice that increasing the order $\eta$ produces a better approximation for the first half eigenvalues but it worsen the approximation in the last half part. On the contrary, from Subfigure \ref{subfig:FD_NonuniformGrid_error} where it is used the non-uniform grid given by \eqref{eq:non_uniform_grid}, increasing the order $\eta$ produces a well-behaved uniform relative approximation.}\label{fig:comparison_FD_err}
\end{figure}
\begin{table}[ht]
	\centering
	\begin{tabular}{l|c|c|c|} 
		\cline{2-4}
		& \multicolumn{3}{c|}{$\max_{k=1,\ldots,n}|\frac{\lambda_k^{(n,\eta)}}{\lambda_k} -1|$ }  \\ 
		\cline{2-4}
		& \multicolumn{1}{c||}{$\eta=1$ } & \multicolumn{1}{c||}{$\eta=10$ } & $\eta=15$          \\ 
		\cline{2-4}
		& $n=10^3$                        & $n=10^3$                         & $n=10^3$           \\ 
		\hline
		\multicolumn{1}{|l|}{uniform grid}     & 0.3201                          & 0.9057                          & 1.0000             \\ 
		\hline
		\multicolumn{1}{|l|}{non-uniform grid} & 0.5939
		                          & 0.2210                           & 0.1814             \\
		\hline
	\end{tabular}
	\captionof{table}{Comparison between the maximum of the eigenvalues relative errors of the discrete differential operators $\prescript{\textrm{e}^{\sqrt{\alpha}}}{1}{\mathcal{L}_{\textnormal{dir},\alpha \tau(x)^2}^{(n,\eta)}}$ obtained by means of $(2\eta+1)$-points central FD discretization on uniform and non-uniform grids, for different values of $\eta$. The parameter $\alpha$ is fixed, with $\alpha=1$. We observe that in the uniform grid case, as $\eta$ increases the maximum increases as well. On the contrary, in the non-uniform grid case given by \eqref{eq:non_uniform_grid}, the maximum decreases significantly as $\eta$ increases. See Figure \ref{fig:comparison_FD_err} for a general overview of the error distribution.}\label{tab:uniform_approx}
\end{table}

\begin{table}[th]
	\centering
	\begin{tabular}{|c|c|c|c||c|c|} 
		\cline{3-6}
		\multicolumn{1}{l}{}           & \multicolumn{1}{l|}{} & \multicolumn{2}{c||}{Uniform grid $\eta=5$} & \multicolumn{2}{c|}{Nonuniform grid $\eta=5$}  \\ 
		\cline{3-6}
		\multicolumn{1}{l}{}           & \multicolumn{1}{l|}{} & $n=10^2$ & $n=10^3$                        & $n=10^2$  & $n=10^3$                           \\ 
		\hline
		\multirow{2}{*}{$\alpha=0.1$ } &  $|\frac{\max|\lambda^{(n,\eta)}_k/\lambda_k|}{\max| \prescript{}{\alpha}{\tilde{\omega}_{\eta,r}}(x)/x^2\pi^2|}-1|$                     & 0.0369   & 0.0018                    &     2.7821e-04      &        2.7510e-06                            \\ 
		\cline{2-6}
		&     $\bar{k}/n$                     & 0.9100  & 0.9050                          &     1      &               1                     \\ 
		\hline\hline		
		\multirow{2}{*}{$\alpha=3$ }   &  $|\frac{\max|\lambda^{(n,\eta)}_k/\lambda_k|}{\max| \prescript{}{\alpha}{\tilde{\omega}_{\eta,r}}(x)/x^2\pi^2|}-1|$                     & 0.0628   & 0.0096                        &     0.0086      &         8.2742e-05                           \\ 
		\cline{2-6}
		&        $\bar{k}/n$                  & 1        & 1                               &  1         &       1                             \\
		\hline
	\end{tabular}\caption{In this table we check numerically the validity of Theorem \ref{thm:necessary_cond_for_uniformity} for different values of $\alpha$ and $n$. The discretization has been made by means of central FD of order $\eta=5$. It can be seen that for every $\alpha$, as $n$ increases then the relative error between $\max_{k=1,\ldots,n}\left|\lambda_k\left(\prescript{\textrm{e}^{\sqrt{\alpha}}}{1}{\mathcal{L}_{\textnormal{dir},\alpha \tau(x)^2}^{(n,\eta)}}\right)/\lambda_k\left(\Eulerc\right)\right|$ and $\max_{x\in[0,1]}\left|\prescript{}{\alpha}{\tilde{\omega}_{\eta,r}}(x)/x^2\pi^2\right|$ decreases, confirming \eqref{eq:limit_relative_error_FD_3_point}. In the table is reported as well the quotient $\bar{k}/n$, where $\bar{k}$ is the $k$-th eigenvalue which achieves the maximum relative error between $\lambda_k\left(\prescript{\textrm{e}^{\sqrt{\alpha}}}{1}{\mathcal{L}_{\textnormal{dir},\alpha \tau(x)^2}^{(n,\eta)}}\right)$ and $\lambda_k\left(\Eulerc\right)$. We can notice that $\bar{k}/n$ is always bounded and it tends to a finite value in $(0,1]$ as $n$ increases. The approximation of $\prescript{}{\alpha}{\tilde{\omega}_{\eta,r}}$ is obtained by means of Algorithm \ref{alg:omega} with $r$ fixed, $r= n$.}
\label{table:maximum_rel_error_FD_eta}
\end{table}

\FloatBarrier

\subsection{IgA discretization by B-slpine of degree $\eta$ and smoothness $C^{\eta-1}$}\label{ssec:galerkin}
In this subsection we continue our analysis in the IgA framework. We just collect all the numerical results of the tests, which confirm again what observed in subsections \ref{ssec:example_uniform_3_points} and \ref{ssec:FD_nonuniform}. The only difference relies on the fact that we took out the largest eigenvalues of the discrete operator $\prescript{\textrm{e}^{\sqrt{\alpha}}}{1}{\mathcal{L}_{\textnormal{dir},\alpha \tau(x)^2}^{(n+\eta-1,\eta)}}$. This is due to the fact that the IgA discretization suffers of a fixed number of outliers which depends on the degree $\eta$ and it is independent of $n$, see \cite[Chapter 5.1.2 p. 153]{CHB}. So, we consider only the eigenvalues $\lambda_k\left(\prescript{\textrm{e}^{\sqrt{\alpha}}}{1}{\mathcal{L}_{\textnormal{dir},\alpha \tau(x)^2}^{(n+\eta-1,\eta)}}\right)$ such that
$$
\lambda_k\left(\prescript{\textrm{e}^{\sqrt{\alpha}}}{1}{\mathcal{L}_{\textnormal{dir},\alpha \tau(x)^2}^{(n+\eta-1,\eta)}}\right) \in R_{\prescript{}{\alpha}{\omega_{\eta}}}.
$$
We stress out the fact that the number of outliers is fixed for every $n$, in accordance with Theorem \ref{thm:clustering&spectral_attraction}.

In Figure \ref{fig:analytic_num_comparison_IgA_eta} we observe again the discrepancy between the analytic relative errors and the numerical relative error. 
\begin{center}
	\begin{figure}[th]
		\centering
		\subfloat[Uniform IgA with $\eta=1$]{
			\includegraphics[height=33mm]{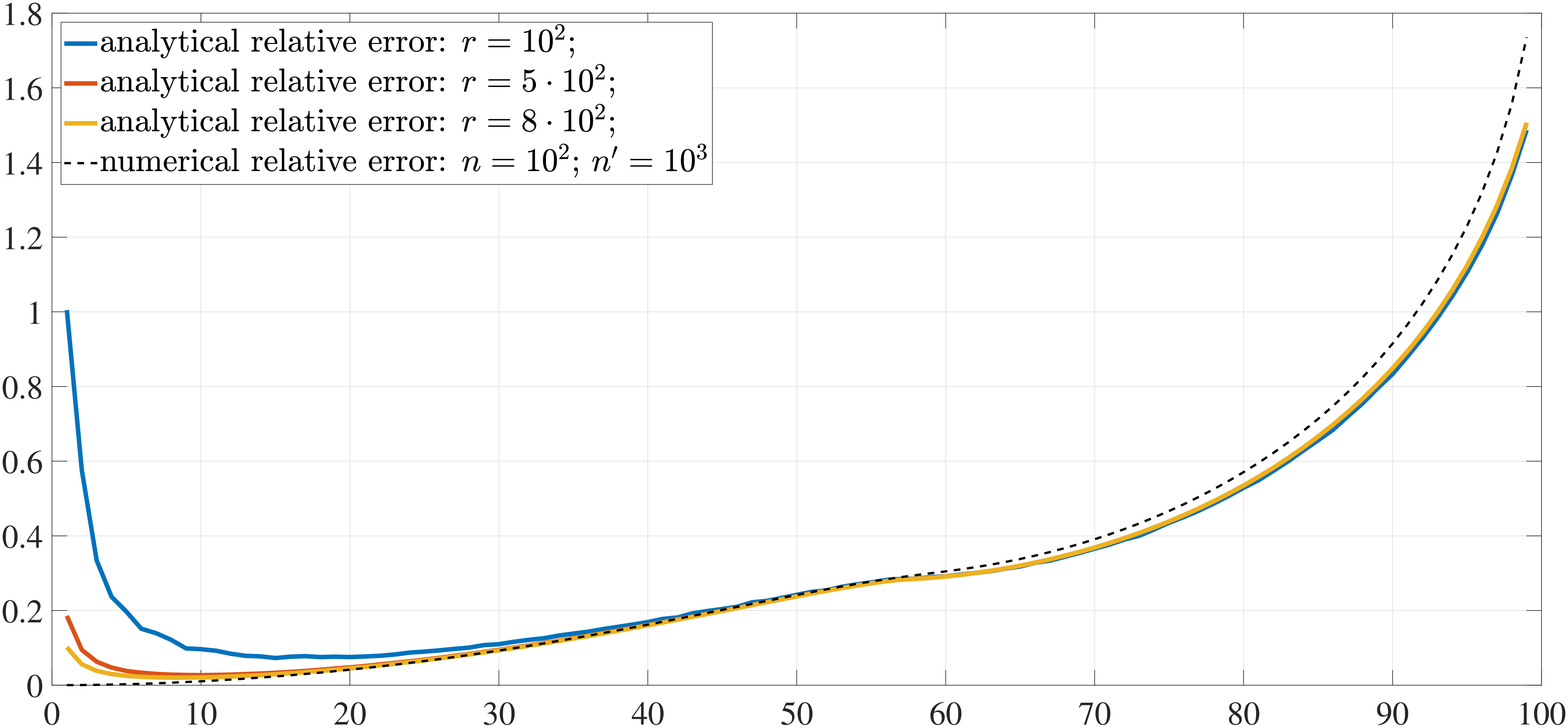}  
			\label{subfig:anal_IgA_eta_1}
		}
		\subfloat[Uniform IgA with $\eta=4$]{
			\centering
			\includegraphics[height=33mm]{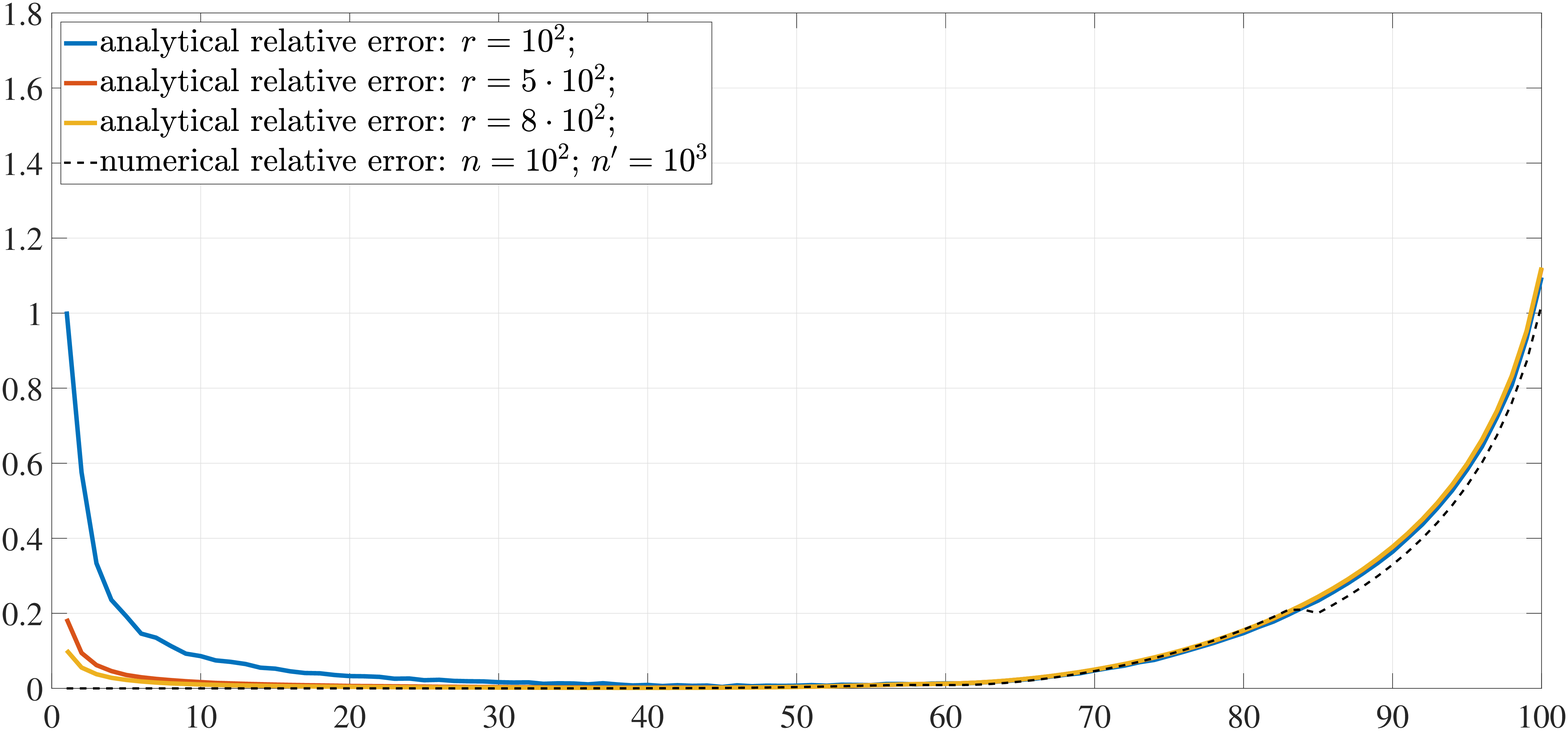}  
			\label{subfig:anal_IgA_eta_4}
		}
		\captionof{figure}{Graphic comparison between the numerical relative error $\textbf{err}_k^{(n)}$ and the analytic relative error $\tilde{\textbf{err}}_{k,r}^{(n)}$ as in Definition \ref{def:num_anal_error}, for increasing $r=10^2,5\cdot 10^2,8\cdot10^2$. The values of $n$ and $n'$ are fixed at $10^2$ and $10^3$, respectively, and $\alpha=1$. In Figure \ref{subfig:anal_IgA_eta_1} it has been used an IgA discretization of order $\eta=1$, while in Figure \ref{subfig:anal_IgA_eta_4} it has been used an IgA discretization of order $\eta=4$. Both of them are made on a uniform grid $\tau(x)=x$. As it happened in Figure \ref{fig:analytical_err_1}, it is displayed an evident discrepancy between the numerical relative error and the analytical relative errors for the first eigenvalues, which is explained by Proposition \ref{prop:non_good_approximation}.}\label{fig:analytic_num_comparison_IgA_eta}
	\end{figure}
\end{center}

In Figure \ref{fig:comparison_non_uniform_IgA} and Figure \ref{fig:comparison_IgA_err} we compare the graphs of the eigenvalue distributions and the relative errors, respectively, between the discrete eigenvalues $\lambda_k\left(\prescript{\textrm{e}^{\sqrt{\alpha}}}{1}{\mathcal{L}_{\textnormal{dir},\alpha \tau(x)^2}^{(n+\eta-1,\eta)}}\right)$ and the exact eigenvalues $\lambda_k\left(\Eulerc\right)$, for different values of $\eta$ on uniform and non-uniform grids. They line up with the numerics of Table \ref{table:maximum_rel_error_IgA_eta}: if the sampling grid is given by \eqref{eq:non_uniform_grid}, the maximum relative error decreases as the order of approximation increases.

	\begin{table}[th]
		\centering
		\begin{tabular}{l|c|c|c|} 
			\cline{2-4}
			& \multicolumn{3}{c|}{$\max_{k=1,\ldots,n}|\frac{\lambda_k^{(n,\eta)}}{\lambda_k} -1|$ }  \\ 
			\cline{2-4}
			& \multicolumn{1}{c||}{$\eta=1$ } & \multicolumn{1}{c||}{$\eta=5$ } & $\eta=10$           \\ 
			\cline{2-4}
			& $n=10^2$                        & $n=10^2$                        & $n=10^2$            \\ 
			\hline
			\multicolumn{1}{|l|}{uniform grid}     & 1.7653                          & 1.0688                          & 1.1465              \\ 
			\hline
			\multicolumn{1}{|l|}{non-uniform grid} & 0.4433                          & 0.0483                          & 0.0265              \\
			\hline
		\end{tabular}
	\captionof{table}{Comparison between the maximum of the eigenvalues relative errors of the discrete differential operators $\prescript{\textrm{e}^{\sqrt{\alpha}}}{1}{\mathcal{L}_{\textnormal{dir},\alpha \tau(x)^2}^{(n,\eta)}}$ obtained by means of IgA discretization of order $\eta$ on uniform and non-uniform grids, for different values of $\eta$. The parameter $\alpha$ is fixed, with $\alpha=1$. We observe that in the non-uniform grid case given by \eqref{eq:non_uniform_grid}, the maximum decreases significantly as $\eta$ increases. See Figure \ref{fig:comparison_IgA_err} for a general overview of the error distribution. Let us notice that we did not take in consideration the outliers, see Definition \ref{def:outliers}.}\label{table:maximum_rel_error_IgA_eta}
\end{table}

\begin{figure}[th]
	\centering
	\subfloat[Uniform IgA with $\eta=1$]{
		\includegraphics[width=8cm]{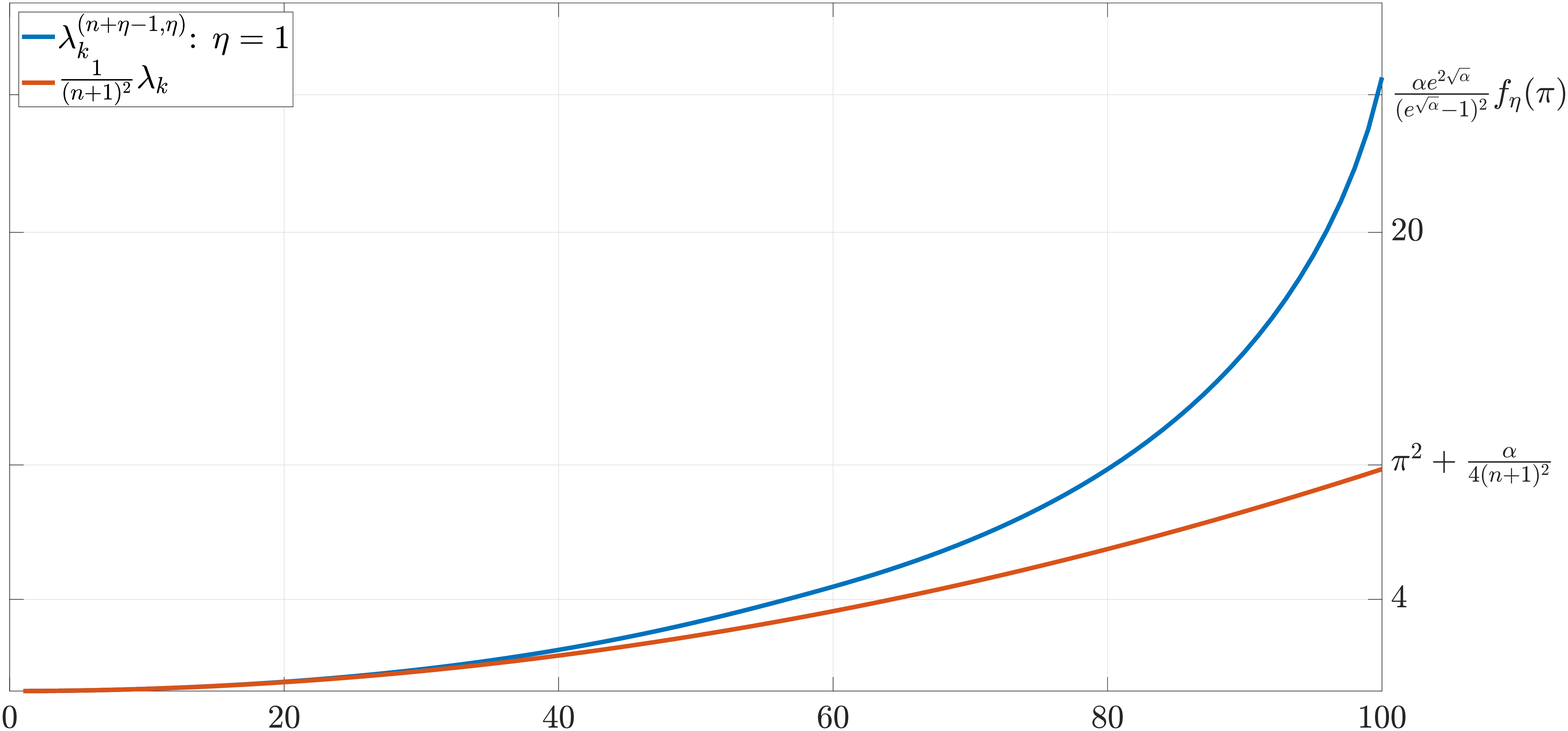}  
		\label{subfig:unif_IgA_eta_1}
	}
	\subfloat[Uniform IgA with $\eta=10$]{
		\centering
		\includegraphics[width=8cm]{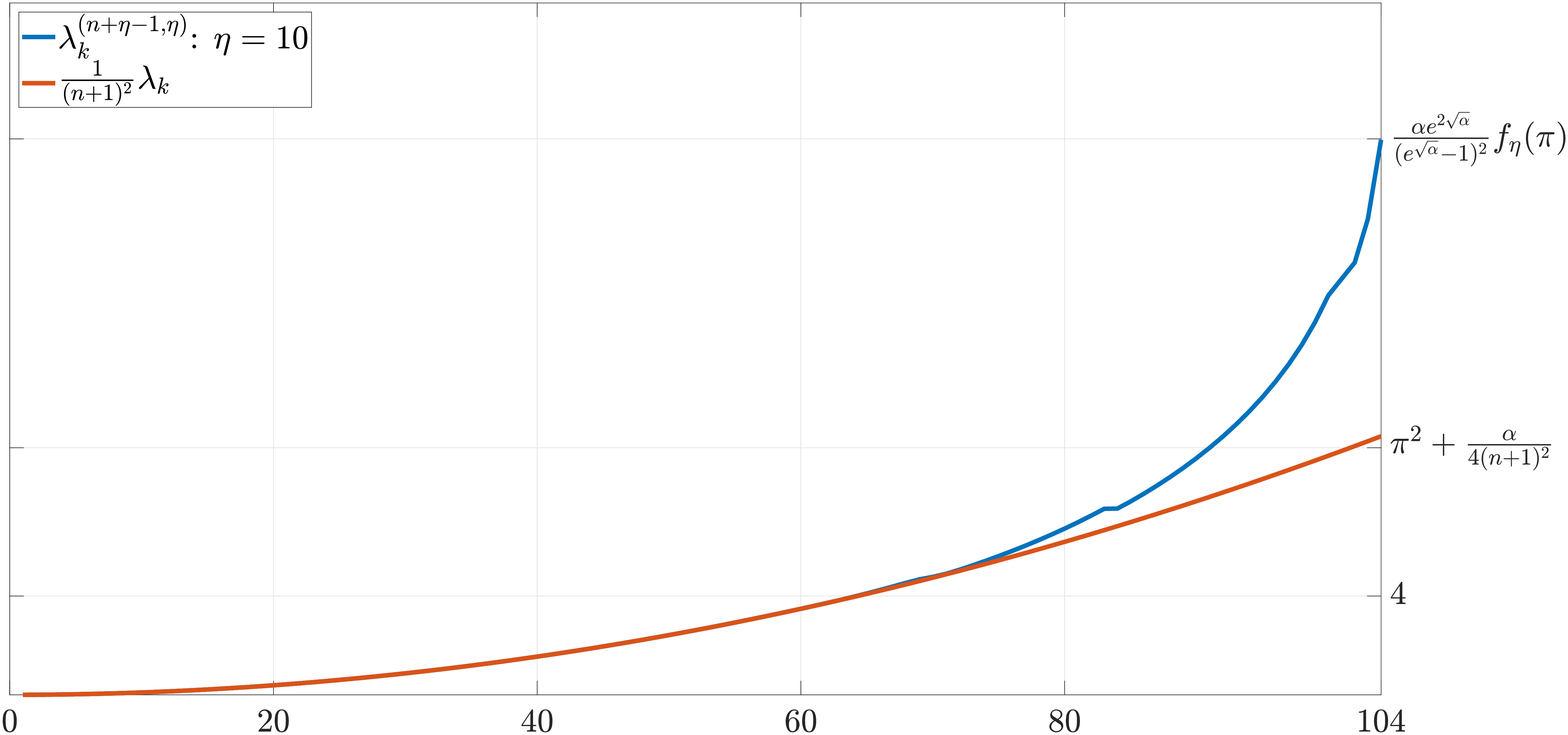}  
		\label{subfig:unif_IgA_eta_10}
	}
	
	\centering
	\subfloat[Non-uniform IgA with $\eta=1$]{
		\includegraphics[width=8cm]{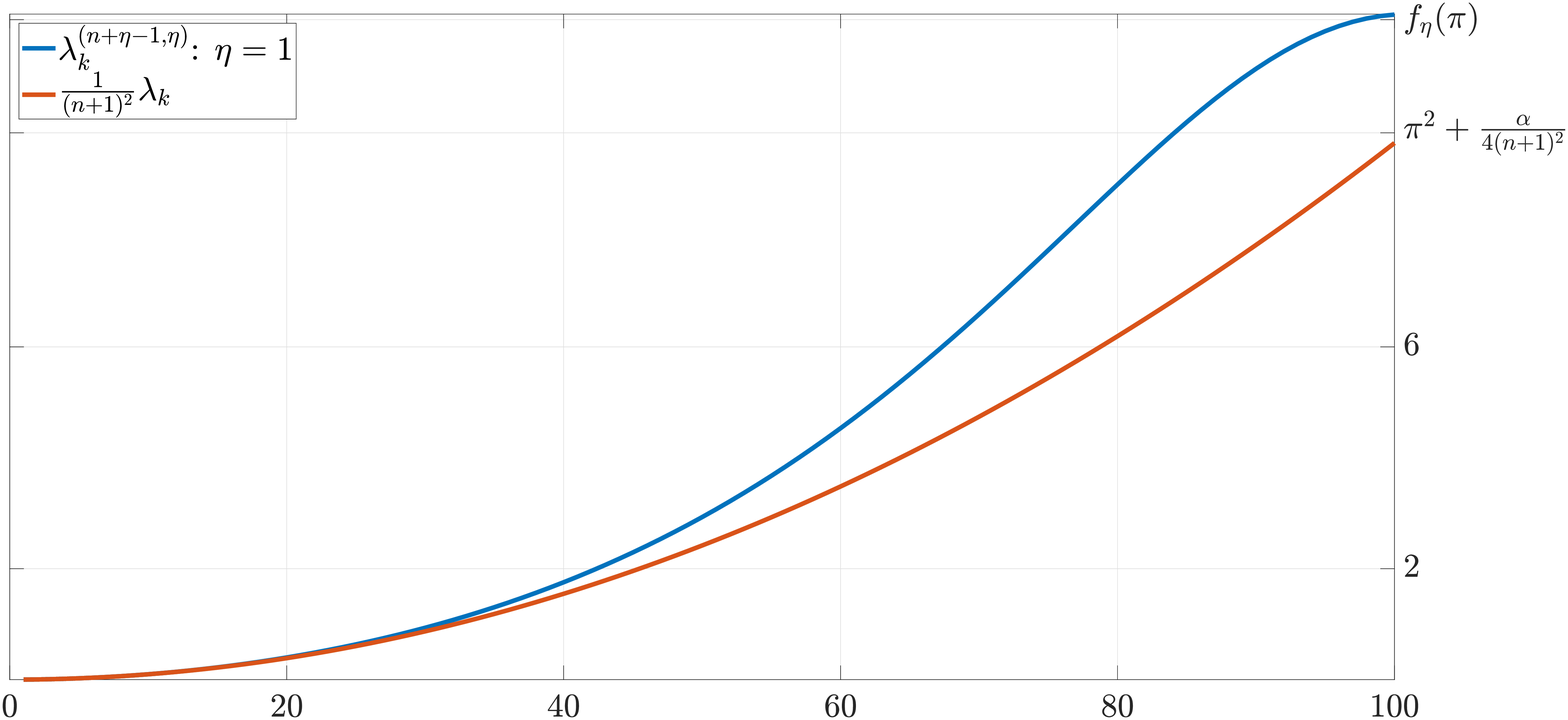}  
		\label{subfig:non_unif_IgA_eta_1}
	}
	\subfloat[Non-uniform IgA with $\eta=10$]{
		\centering
		\includegraphics[width=8cm]{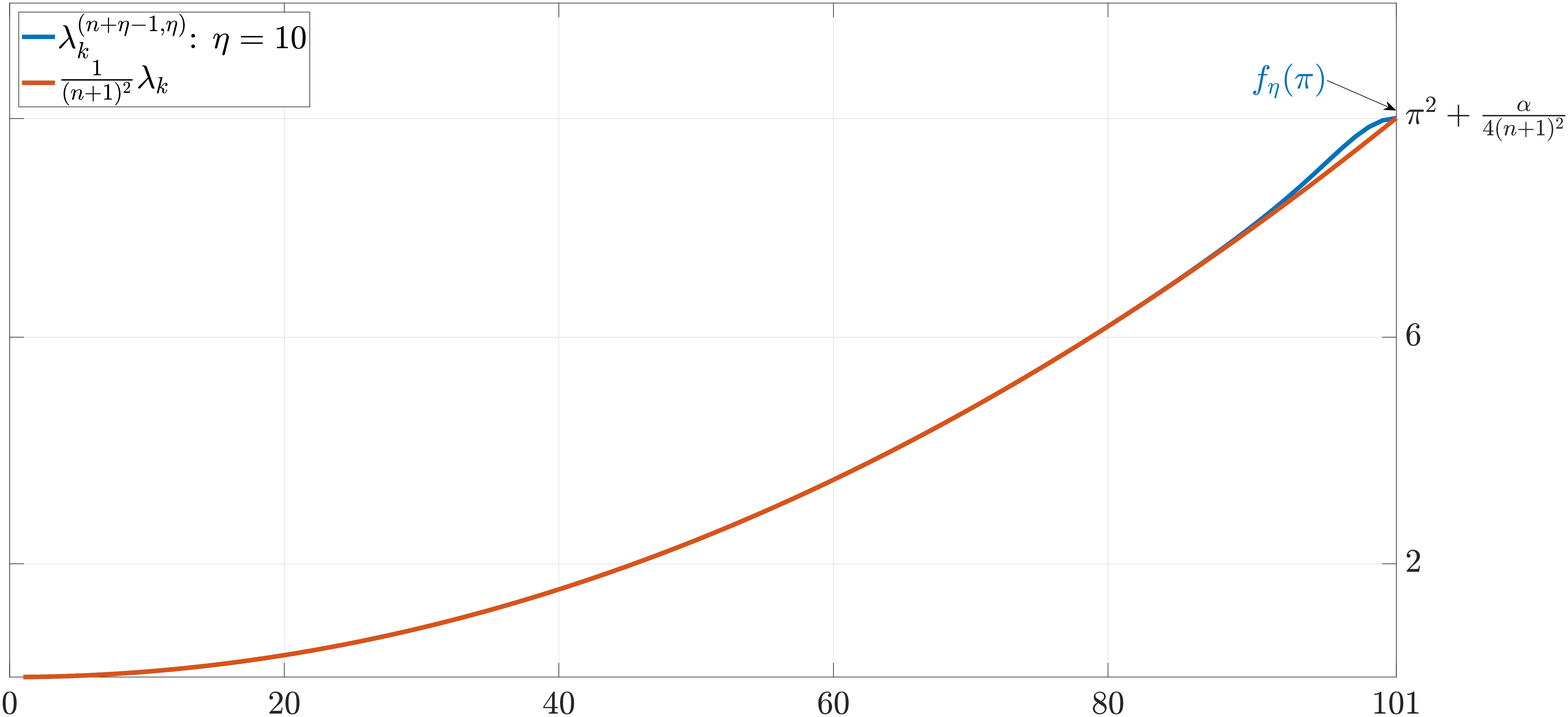}  
		\label{subfig:non_unif_IgA_eta_10}
	}
	\captionof{figure}{Graphic comparison between the eigenvalues distribution of the weighted discrete differential operators $\prescript{\textrm{e}^{\sqrt{\alpha}}}{1}{\hat{\mathcal{L}}_{\textnormal{dir},\alpha \tau(x)^2}^{(n+\eta-1,\eta)}}$ obtained by means of IgA discretization of order $\eta$ on uniform and non-uniform grids. The parameters $\alpha$ and $n$ are fixed, with $\alpha=1$ and $n=10^2$, while $\eta$ changes. Let us observe that in figures \ref{subfig:non_unif_IgA_eta_1}, \ref{subfig:non_unif_IgA_eta_10}, i.e., in the case of IgA discretization on the non-uniform grid given by \eqref{eq:non_uniform_grid}, the graph of the eigenvalue distribution seems to converge uniformly to the graph of the exact eigenvalues $(n+1)^{-2}\lambda_k$, as $\eta$ increases. The same phenomenon does not happen in the case of IgA discretization on uniform grid, as it is clear from figures \ref{subfig:unif_IgA_eta_1},\ref{subfig:unif_IgA_eta_10}. Let us notice that we did not take in consideration the outliers, see Definition \ref{def:outliers}.}\label{fig:comparison_non_uniform_IgA}
\end{figure}

\begin{figure}[th]
	\centering
	\subfloat[Relative errors for uniform central FD]{
		\includegraphics[height=35mm]{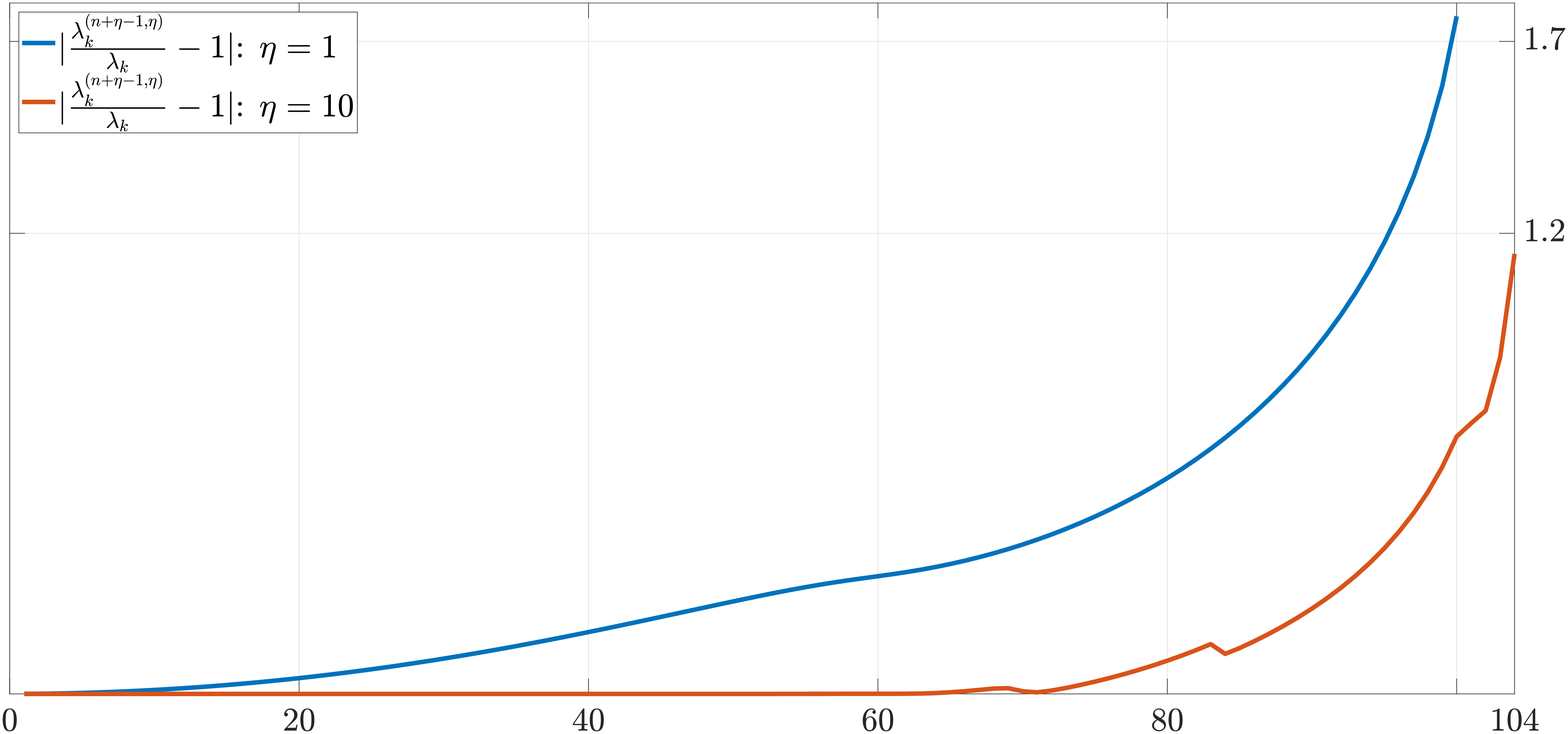}  
		\label{subfig:IgA_uniformGrid_error}
	}
	\subfloat[Relative errors for non-uniform central FD]{
		\centering
		\includegraphics[height=35mm]{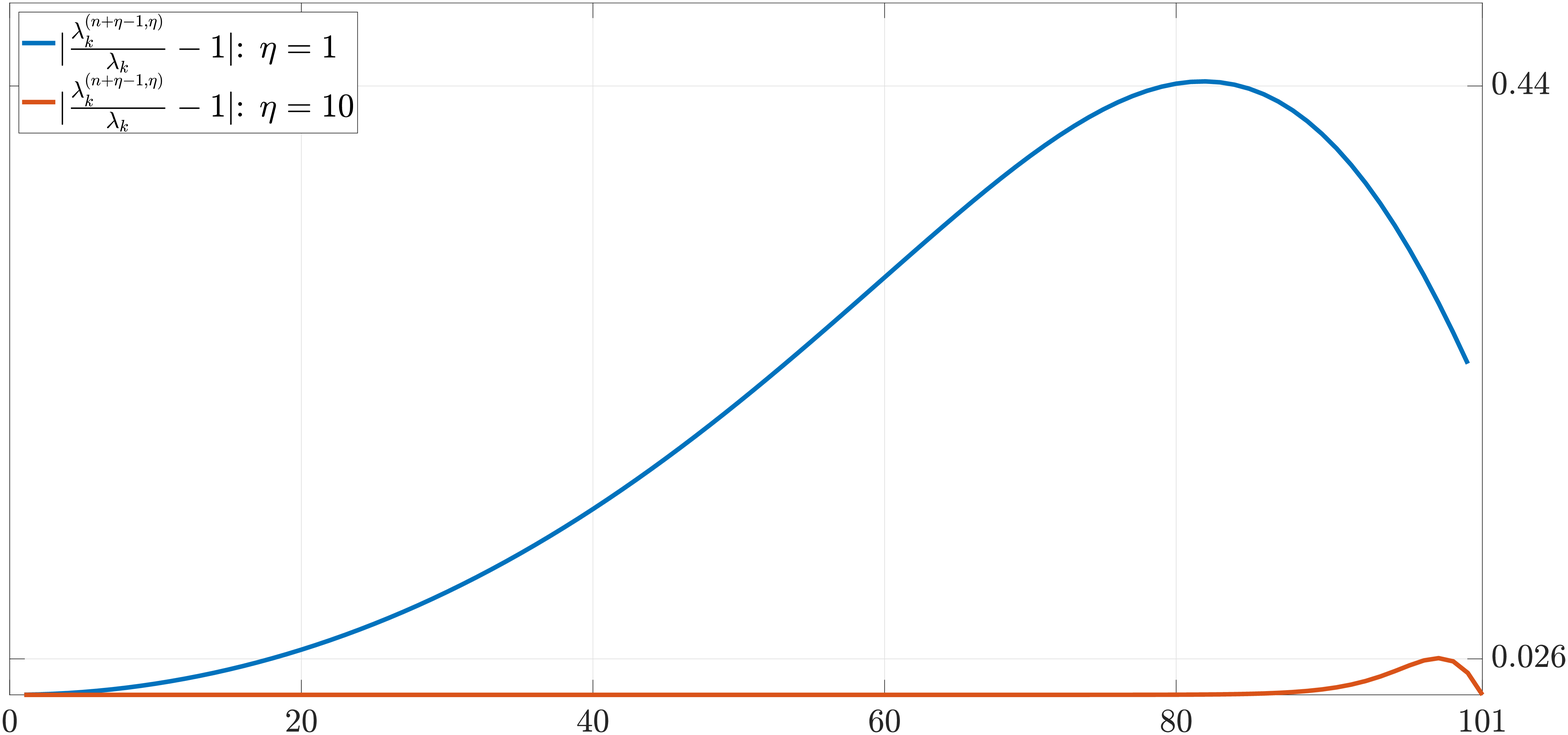}  
		\label{subfig:IgA_NonuniformGrid_error}
	}
	\captionof{figure}{Graphic comparison between the eigenvalues relative errors of the discrete differential operators $\prescript{\textrm{e}^{\sqrt{\alpha}}}{1}{\hat{\mathcal{L}}_{\textnormal{dir},\alpha \tau(x)^2}^{(n+\eta-1,\eta)}}$ obtained by means of IgA discretization of order $\eta$ on uniform and non-uniform grids, for different values of $\eta$. The parameters $\alpha$ and $n$ are fixed, with $\alpha=1$ and $n=10^2$. We notice from Subfigure \ref{subfig:IgA_NonuniformGrid_error}, where it is used the non-uniform grid given by \eqref{eq:non_uniform_grid}, that increasing the order $\eta$ produces a well-behaved uniform relative approximation. Let us notice that we did not take in consideration the outliers, see Definition \ref{def:outliers}.}\label{fig:comparison_IgA_err}
\end{figure}

\begin{table}[th]
	\centering
	\begin{tabular}{|c|c|c|c||c|c|} 
		\cline{3-6}
		\multicolumn{1}{l}{}           & \multicolumn{1}{l|}{} & \multicolumn{2}{c||}{Uniform grid $\eta=4$} & \multicolumn{2}{c|}{Nonuniform grid $\eta=4$}  \\ 
		\cline{3-6}
		\multicolumn{1}{l}{}           & \multicolumn{1}{l|}{} & $n=10^2$ & $n=10^3$                        & $n=10^2$  & $n=10^3$                           \\ 
		\hline
		\multirow{2}{*}{$\alpha=0.1$ } &  $|\frac{\max|\lambda^{(n,\eta)}_k/\lambda_k|}{\max| \prescript{}{\alpha}{\tilde{\omega}_{\eta,r}}(x)/x^2\pi^2|}-1|$                     & 0.3397   &           0.0881          &     0.0016      &        1.7241e-04                            \\ 
		\cline{2-6}
		&     $\bar{k}/n$                     &1 &         1                     &    0.9400       &              0.9300                   \\ 
		\hline\hline		
		\multirow{2}{*}{$\alpha=3$ }   &  $|\frac{\max|\lambda^{(n,\eta)}_k/\lambda_k|}{\max| \prescript{}{\alpha}{\tilde{\omega}_{\eta,r}}(x)/x^2\pi^2|}-1|$                     & 0.1735   &             0.0425          &     0.0016     &         1.7243e-04                         \\ 
		\cline{2-6}
		&        $\bar{k}/n$                  & 1        &     1                           &  0.9400         &                             0.9300       \\
		\hline
	\end{tabular}\caption{In this table we check numerically the validity of Theorem \ref{thm:necessary_cond_for_uniformity} for different values of $\alpha$ and $n$. The discretization has been made by means of IgA of order $\eta=4$. It can be seen that for every $\alpha$, as $n$ increases then the relative error between $\max_{k=1,\ldots,n}\left|\lambda_k\left(\prescript{\textrm{e}^{\sqrt{\alpha}}}{1}{\mathcal{L}_{\textnormal{dir},\alpha x^2}^{(n+\eta-1,\eta)}}\right)/\lambda_k\left(\Eulerc\right)\right|$ and $\max_{x\in[0,1]}\left|\prescript{}{\alpha}{\tilde{\omega}_{\eta,r}}(x)/x^2\pi^2\right|$ decreases, confirming \eqref{eq:limit_relative_error_FD_3_point}. In the table is reported as well the quotient $\bar{k}/n$, where $\bar{k}$ is the $k$-th eigenvalue which achieves the maximum relative error between $\lambda_k\left(\prescript{\textrm{e}^{\sqrt{\alpha}}}{1}{\mathcal{L}_{\textnormal{dir},\alpha x^2}^{(n+\eta-1,\eta)}}\right)$ and $\lambda_k\left(\Eulerc\right)$. We can notice that $\bar{k}/n$ is always bounded and it tends to a finite value in $(0,1]$ as $n$ increases. The approximation of $\prescript{}{\alpha}{\tilde{\omega}_{\eta,r}}$ is obtained by means of Algorithm \ref{alg:omega} with $r$ fixed, $r= n$. Let us notice that we did not take in consideration the outliers, see Definition \ref{def:outliers}.}
	\label{table:maximum_rel_error_IgA_eta_nec_cond}
\end{table}

\FloatBarrier

\subsection{The $L^1$ case}\label{ssec:L1}
We leave for a moment the case of regular SLPs. In \cite{Garoni18} it was addressed the issue of extending the spectral symbol analysis to the case of $\textnormal{L}^1$ coefficients. We show in the next example that the sampling of the spectral symbol does not provide accurate approximation of the eigenvalues of the weighted matrix discretization operator even in the sense of the absolute error.

Let us fix $p(x)=x^{-1/2} \in \textnormal{L}^1(0,1)$, $q(x) \equiv 0$ in \eqref{eq:S-L_general2} with Dirichlet BCs, namely
\begin{equation}\label{eq:L1}
\begin{cases}-\partial_x\left(x^{-1/2}\partial_xu(x)\right) = \lambda u(x) & x\in (0,1),\\
u(0)=u(1)=0.
\end{cases}
\end{equation}
Then the spectral symbol given by a $3$-points central FD scheme is 
$$
\omega(x,\theta)=x^{-1/2}(2-2\cos(\theta)), \qquad (x,\theta) \in (0,1]\times [0,\pi],
$$
see \cite[Theorem 10.5]{GS17} (the proof works fine even if $p(x) \in C([0,1])$ is replaced by $p(x) \in L^1([0,1])$). It is not difficult to prove that the monotone rearrangement $\tilde{\omega}: [0,1)\to [0,+\infty)$ is such that
\begin{equation}\label{eq:asymptotic_bound}
\tilde{\omega}_n^{(n)} = \tilde{\omega}\left(\frac{n}{n+1}\right) \sim 4\sqrt{n+1},  
\end{equation}
and that
\begin{equation}\label{eq:upper_bound}
\lambda_n\left(\prescript{1}{0}{\hat{\mathcal{L}}_{\textnormal{dir},x^{-1/2}}^{(n)}}\right)\leq \sqrt{2}\left(1+\frac{2}{\sqrt{3}}\right)\sqrt{n+1} < 4\sqrt{n+1}
\end{equation}
where the upper bound on the largest eigenvalue is proven by the Gershgorin theorems. The limit relation in Remark \ref{rem:ssymbol_sampling} still holds, 
 \begin{equation}\label{eq:limit_rel_L1}
 \lim_{n \to \infty} \frac{1}{n} \sum_{k=1}^{n}\lambda_k\left(\prescript{1}{0}{\hat{\mathcal{L}}_{\textnormal{dir},x^{-1/2}}^{(n)}}\right) = \frac{1}{\pi}\iint_{[0,1]\times [0,\pi]} x^{-1/2}(2-2\cos(\theta)) \,  dxd\theta = 4,
\end{equation}
but on the other hand,
\begin{equation}\label{eq:divergence}
\left\|\tilde{\omega}_{k}^{(n)}-\lambda_k\left(\prescript{1}{0}{\hat{\mathcal{L}}_{\textnormal{dir},x^{-1/2}}^{(n)}}\right)\right\|_\infty := \max_{k=1,\ldots,n}\left\{\left| \tilde{\omega}\left(\frac{k}{n+1}\right) - \lambda_k\left(\prescript{1}{0}{\hat{\mathcal{L}}_{\textnormal{dir},x^{-1/2}}^{(n)}}\right)\right|  \right\}\to \infty.
\end{equation}
In Figure \ref{fig:comparison_L1} and Table \ref{tab:comparison:L1} it is possible to see a summary of these last considerations.  

\begin{remark}
Let us observe that Corollary \ref{cor:discrete_Weyl_law} is still valid on every compact subset $K \Subset [0,1)$. For example, let $x_0 \in (0,1)$: then 
$$
\max\left\{\left| \tilde{\omega}\left(\frac{k}{n+1}\right) - \lambda_k\left(\prescript{1}{0}{\hat{\mathcal{L}}_{\textnormal{dir},x^{-1/2}}^{(n)}}\right)\right| \, : \, \frac{k}{n}\leq x_0 \right\}\to 0 \qquad \mbox{as } n \to \infty.
$$
Indeed, $\tilde{\omega}$ is absolute continuous on every compact subset which does not contain $x=1$.
\end{remark}

\begin{center}
	\includegraphics[width=15cm]{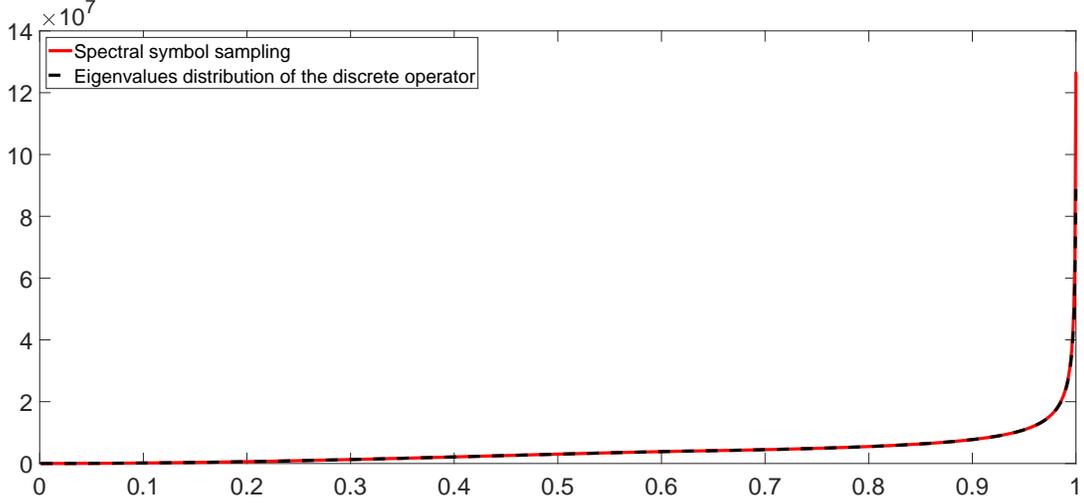}
	\captionof{figure}{Graphic comparison between the analytic approximation of the eigenvalues given by the uniform sampling of $(n+1)^2\tilde{\omega}_{r,k}$ (continous red line) and the eigenvalues of the discrete differntial operator $\lambda_k\left(\prescript{1}{0}{\mathcal{L}_{\textnormal{dir},x^{-1/2}}^{(n)}}\right)$ (dotted black line) of Probelm \eqref{eq:L1}.  On the $x$-axis is reported the quotient $k/n$, for $k=1,\ldots,n$. $\tilde{\omega}_r$ is computed according to Algorithm \ref{alg:omega} with $r=n=10^3$. The uniform sampling of $(n+1)^2\tilde{\omega}_{r,k}$ seems to approximate well the egienvalues $\lambda_k\left(\prescript{1}{0}{\mathcal{L}_{\textnormal{dir},x^{-1/2}}^{(n)}}\right)$, but it is only a false perception, as reported in Table \ref{tab:comparison:L1}. The sovrapposition of the graphs on compact sets $[0,x_0]\subset [0,1)$ is explained by Theorem \ref{thm:discrete_Weyl_law} and the limit \eqref{eq:discrete_Weyl_law2_2} (or equivalently \eqref{convergence_to_ICDF2}).}\label{fig:comparison_L1} 
\end{center}

\begin{table}[th]
	\centering
	\begin{tabular}{|l|c|c|c|} 
		\cline{2-4}
		\multicolumn{1}{l|}{}                                       & $r=n=10^2$  & $r=n=10^3$  & $r=n=2\cdot10^3$   \\ 
		\hline
		$\|\tilde{\omega}_{r,k}^{(n)}-\lambda_k\|_\infty$          & 11.7523     & 37.0283     & 52.3532            \\ 
		\hline\hline
		$\max_{k=1,\ldots,n}\{\tilde{\textbf{err}}_{r,k}^{(n)}\}$  & 0.4133      & 0.4136      & 0.4136             \\ 
		\hline\hline
		$\tilde{\omega}_{r,n}^{(n)}/\sqrt{n+1}$                    & 3.9990      & 4           & 4                  \\ 
		\hline\hline
		$\lambda_n/\sqrt{n+1}$                                     & 2.8296      & 2.8296      & 2.8296             \\ 
		\hline\hline
		$n^{-1}\sum_{k=1}^n \lambda_k$                             & 3.7663      & 3.9200      & 3.9428             \\ 
		\hdashline[1pt/1pt]
		$\int_0^1\tilde{\omega}(x)\,d\mu_1(x)$                     & 4           & 4           & 4                  \\
		\hline
	\end{tabular}
\caption{Comparison between an equispaced sampling of the (approximated) monotone rearrangement $\tilde{\omega}_r$ and the eigenvalues $\lambda_k$ of the discrete differential operator of Problem \eqref{eq:L1}. $\tilde{\omega}_r$ is computed according to Algorithm \ref{alg:omega} with $r=n$. In the first row it is calculated the absolute error: it increases as $n$ increases as stated in \eqref{eq:divergence}. In the second row it is calculated the maximum of the analytic relative error: it saturates at a lower bound $c>0$. In the third and fourth row are validated the estimates \eqref{eq:asymptotic_bound} and \eqref{eq:upper_bound}, respectively. In the sixth row both the left-hand side and the right-hand side of Equation \eqref{eq:limit_rel_L1} are compared.}\label{tab:comparison:L1}
\end{table}

\section{Theoretical results}\label{sec:theory}

\subsection{Proofs of results presented in Section \ref{sec:example}}\label{ssec:generalization_sec4}
We summarize and generalize the results of Section \ref{sec:example}. 

\begin{proposition}\label{prop:non_good_approximation}
	Let us consider the Sturm-Liouville eigenvalue problem
	\begin{equation}\label{eq:Euler-Cauchy_general}
	\begin{cases}-\partial_x\left(p(x)\partial_xu(x)\right) +q(x)u(x)  = \lambda w(x)u(x) & x\in (a,b),\\
	\sigma_1 u(a) - \sigma_2 p(a)\partial_xu(x)_{|x=a} =0 & \sigma_1^2+\sigma_2^2 >0,\\
	\zeta_1 u(b) +\zeta_2 p(b)\partial_xu(x)_{|x=b} =0 & \zeta_1^2+\zeta_2^2 >0,
	\end{cases}
	\end{equation}
	such that 
	\begin{enumerate}[(i)]
		\item $p,p',w,w',q,(pw)',(pw)'' \in C([a,b])$;
		\item $p,w>0$;
		\item $\sigma_1^2+\sigma_2^2 >0$, $\zeta_1^2+\zeta_2^2>0$.
	\end{enumerate}
	Discretize the above Problem \eqref{eq:Euler-Cauchy_general} by means of a numerical matrix method, and let $\prescript{b}{a}{\mathcal{L}_{\textnormal{BCs},p,q,w}^{(n,\eta)}}$ be the correspondent discrete operator, where $n$ is the mesh finesse parameter and $\eta$ is the order of approximation of the numerical method. Define $B=\int_a^b \sqrt{\frac{w(x)}{p(x)}}dx$. If:
	\begin{enumerate}[(a)]
		\item for every fixed $k\in \N$, 
		\begin{equation*}
		\lim_{n \to \infty} \lambda_k\left( \prescript{b}{a}{\mathcal{L}_{\textnormal{BCs},p(x),q(x),w(x)}^{(n,\eta)}}\right) = \lambda_k\left(\prescript{b}{a}{\mathcal{L}_{\textnormal{BCs},p(x),q(x),w(x)}}\right);
		\end{equation*}
		\item there exists $\omega : [a,b]\times [0,\pi] \to \R$, $\omega \in L^1\left([a,b]\times [0,\pi]\right)$ such that
		\begin{equation*}
		\left\{(n+1)^{-2}\prescript{b}{a}{\mathcal{L}_{\textnormal{BCs},p(x),q(x),w(x)}^{(n,\eta)}}\right\}_n \sim_{\lambda}\omega(x,\theta) \quad (x,\theta) \in [a,b]\times [0,\pi];
		\end{equation*}
		\item\label{item:tilde_omega_asymptotics} for every $\eta$, the monotone rearrangement $\tilde{\omega}_\eta$, as defined in \eqref{eq:rearrangment}, is such that
		\begin{equation*}
		\tilde{\omega}_\eta(x) \sim \frac{x^2\pi^2}{B^2}, \qquad \mbox{as } x\to 0.
		\end{equation*} 
	\end{enumerate}
	Then, for every fixed $k\in \N$ 
	$$
	\lim_{n \to \infty} \left|\frac{\tilde{\omega}_\eta\left(\frac{k}{n+1}\right)}{\lambda_k\left(\prescript{b}{a}{\hat{\mathcal{L}}_{\textnormal{BCs},p(x),q(x),w(x)}^{(n,\eta)}} \right)} -1 \right|=	\lim_{n \to \infty} \left|\frac{(n+1)^2\tilde{\omega}_\eta\left(\frac{k}{n+1}\right)}{\lambda_k\left(\prescript{b}{a}{\mathcal{L}_{\textnormal{BCs},p(x),q(x),w(x)}^{(n,\eta)}} \right)} -1 \right| =c_k\geq 0, \qquad \lim_{k \to \infty} c_k =0,
	$$
	with $c_k$ independent of $\eta$. The inequality is strict, i.e., $c_k >0$ for every $k$ such that
	$$
	\lambda_k\left(\prescript{B}{0}{\mathcal{L}_{\textnormal{BCs},V}}\right) \neq  \lambda_k\left(-\prescript{B}{0}{\Delta_{\textnormal{dir}}}\right),
	$$
	where $\prescript{B}{0}{\mathcal{L}_{\textnormal{BCs},V}}$ is the differential operator associated to the normal form \eqref{eq:S-L_normal_eig_prob} of Problem \eqref{eq:Euler-Cauchy_general} by the Liouville transform \eqref{eq:liouville_transform}, and $-\prescript{B}{0}{\Delta_{\textnormal{dir}}}$ is the (negative) Laplacian operator with Dirichlet BCs over the domain $[0,B]$.
	
	In particular,
	\begin{equation*}
	\lim_{n \to \infty}\analerr = c_k, 
	\end{equation*}
	where $\analerr$ is the analytical error defined in Definition \ref{def:num_anal_error}.
\end{proposition}  
\begin{proof}
By hypothesis, for every fixed $k$
$$
\lim_{n \to \infty}\lambda_k\left(\prescript{b}{a}{\mathcal{L}_{\textnormal{BCs},p(x),q(x),w(x)}^{(n,\eta)}}\right) = \lambda_k\left(\prescript{b}{a}{\mathcal{L}_{\textnormal{BCs},p(x),q(x),w(x)}}\right), 
$$ 
where $\lambda_k\left(\prescript{b}{a}{\mathcal{L}_{\textnormal{BCs},p(x),q(x),w(x)}}\right)$ are the eigenvalues of the continuous differential operator, and 
$$
\lambda_k\left(\prescript{b}{a}{\mathcal{L}_{\textnormal{BCs},p(x),q(x),w(x)}}\right) = \lambda_k\left(\prescript{B}{0}{\mathcal{L}_{\textnormal{BCs},V(y)}}\right),
$$
where $\prescript{B}{0}{\mathcal{L}_{\textnormal{BCs},V(y)}}$ is the differential operator associated to the normal form \eqref{eq:S-L_normal_eig_prob} of Problem \eqref{eq:Euler-Cauchy_general} through the Liouville transform. Let us observe that for $V(y)\equiv 0$ and Dirichlet BCs, then
$$
\lambda_k\left(\prescript{B}{0}{\mathcal{L}_{\textnormal{dir},V\equiv 0}}\right) =  \lambda_k\left(-\prescript{B}{0}{\Delta_{\textnormal{dir}}}\right)=\frac{k^2\pi^2}{B^2}.
$$
Therefore, from item \eqref{item:tilde_omega_asymptotics}
\begin{equation}
\lim_{n \to \infty} (n+1)^2\tilde{\omega}_\eta\left(\frac{k}{n+1}\right) = \frac{k^2\pi^2}{B^2} = \lambda_k\left(-\prescript{B}{0}{\Delta_{\textnormal{dir}}}\right).
\end{equation}
Then it is immediate to prove that if
$$
\lambda_k\left(\prescript{B}{0}{\mathcal{L}_{\textnormal{BCs},V(y)}}\right) \neq  \lambda_k\left(-\prescript{B}{0}{\Delta_{\textnormal{dir}}}\right),
$$
then
$$
\lim_{n \to \infty} \left|\frac{(n+1)^2\tilde{\omega}_\eta\left(\frac{k}{n+1}\right)}{\lambda_k\left(\prescript{b}{a}{\mathcal{L}_{\textnormal{BCs},p(x),q(x),w(x)}^{(n,\eta)}} \right)} -1 \right| =c_k>0.
$$
Moreover,
$$
\lambda_k\left(\prescript{B}{0}{\mathcal{L}_{\textnormal{BCs},V(y)}}\right) \sim \frac{k^2\pi^2}{B^2} \quad \mbox{for } k \to \infty,
$$
and then $c_k \to 0$ as $k \to \infty$.
\end{proof}

\begin{corollary}\label{cor:item_c}
If
\begin{equation*}
\omega_\eta (x,\theta)= \frac{p(x)}{w(x)(b-a)^2}f_\eta(\theta), \qquad (x,\theta) \in [a,b]\times [0,\pi],
\end{equation*}
with $f_\eta(\theta)$ nonnegative, nondecreasing and such that $f_\eta(\theta) \sim \theta^2$ as $\theta\to 0$,
then item (\ref{item:tilde_omega_asymptotics}) of Proposition \ref{prop:non_good_approximation} is satisfied.
\end{corollary}
\begin{proof}
From \eqref{eq:rearrangment2} and \eqref{eq:rearrangment}, for all $t \in \left[0, t_0\right]$, with
$$
t_0 = (b-a)^{-2}\min_{[a,b]}\left(p(y)/w(y)\right)\sup_{[0,\pi]}\left(f_\eta(\theta)\right), 
$$
we have that 
\begin{equation*}
\phi_\eta (t) = \frac{1}{\pi(b-a)} \int_a^b \left( \int_0^\pi \mathbbm{1}_{\Omega_y(t)} (\theta) d\theta\right)dy,  \quad \mbox{with }\Omega_y(t):=\left\{\theta \in [0,\pi] \, : \, 0\leq f_\eta(\theta)\leq \frac{(b-a)^2w(y)}{p(y)}t\right\}.
\end{equation*}
By the monotonicity of $f_\eta$, it holds that $\theta \to 0$ as $t \to 0$. For every $\epsilon>0$ there exists $\delta_\epsilon>0$ such that for every $t\in [0,\min\{t_0;t_0/\delta_\epsilon\}]$ then $(1-\epsilon)\theta^2< f_\eta(\theta)< (1+\epsilon)\theta^2$, and so
\begin{equation*}
\frac{1}{\pi(b-a)} \int_a^b \mu_1\left( \Omega_y^{+}(t) \right)dx \leq \phi_\eta (t) \leq \frac{1}{\pi(b-a)} \int_a^b \mu_1\left( \Omega_y^{-}(t) \right)dx,
\end{equation*}
with
\begin{equation*}
\Omega_y^{+}(t):=\left\{\theta \in [0,\pi] \, : \, \theta^2\leq \frac{(b-a)^2}{1+\epsilon}\frac{w(y)}{p(y)}t\right\}, \quad \Omega_y^{-}(t):=\left\{\theta \in [0,\pi] \, : \, \theta^2\leq \frac{(b-a)^2}{1-\epsilon}\frac{w(y)}{p(y)}t\right\}.
\end{equation*}
So we have that
\begin{equation*}
\frac{B}{\pi\sqrt{1+\epsilon}}\sqrt{t} \leq \phi_\eta(t) \leq \frac{B}{\pi\sqrt{1-\epsilon}}\sqrt{t}, \qquad B=\int_a^b \sqrt{\frac{w(y)}{p(y)}}dy.
\end{equation*}
By definition \eqref{eq:rearrangment}, $t\to 0$ as $x\to0$ and then it holds that
\begin{equation*}
(1-\epsilon)\frac{x^2\pi^2}{B^2}\leq \tilde{\omega}_\eta (x) \leq (1+\epsilon)\frac{x^2\pi^2}{B^2} \qquad \mbox{for }x \mbox{ small enough},
\end{equation*}
and the thesis follows.
\end{proof}

\begin{remark}\label{rem:not_good_approximation}
The matrix methods of subsections \ref{ssec:FD}, \ref{ssec:IsoG} satisfy the hypothesis of Proposition \ref{prop:non_good_approximation}, see theorems \ref{thm:FD_symbol}, \ref{thm:Galerkin_symbol} and corollaries \ref{cor:FD_uniform}, \ref{cor:Galerkin_uniform}. Therefore, in general, a uniform sampling of their spectral symbols does not provide an accurate approximation of the eigenvalues $\lambda_k\left( \prescript{b}{a}{\mathcal{L}_{\textnormal{BCs},p(x),q(x),w(x)}^{(n,\eta)}}\right)$ and $\lambda_k\left(\prescript{b}{a}{\mathcal{L}_{\textnormal{BCs},p(x),q(x),w(x)}}\right)$, in the sense of the relative error. See subsections \ref{ssec:example_uniform_3_points},\ref{ssec:FD_nonuniform} and \ref{ssec:galerkin} for numerical examples. On the other hand, 
$$
\left\|  \tilde{\omega}_\eta\left(\frac{k}{n+1}\right)  - (n+1)^{-2}\lambda_k\left( \prescript{b}{a}{\mathcal{L}_{\textnormal{BCs},p(x),q(x),w(x)}^{(n,\eta)}}\right)\right\|_\infty \to 0 \qquad \mbox{as } n\to \infty,
$$
if we exclude the outliers, see Corollary \ref{cor:discrete_Weyl_law} and figures \eqref{fig:eig_symbol_comparison} and \eqref{fig:comparison_L1}.
\end{remark}

\begin{theorem}\label{thm:necessary_cond_for_uniformity}
	Let us consider Problem \ref{eq:Euler-Cauchy_general} such that items (i)-(iv) of Proposition \ref{prop:non_good_approximation} are satisfied. Let $\prescript{b}{a}{\mathcal{L}_{\textnormal{BCs},p(x),q(x),w(x)}^{(n)}}$ be the discrete operator of Problem \ref{eq:Euler-Cauchy_general} obtained by means of any matrix method such that 
	\begin{enumerate}[(a)]
			\item \begin{equation*}
		\lim_{n \to \infty}	\lambda_k\left( \prescript{b}{a}{\mathcal{L}_{\textnormal{BCs},p(x),q(x),w(x)}^{(n)}}\right)=   \lambda_k\left( \prescript{b}{a}{\mathcal{L}_{\textnormal{BCs},p(x),q(x),w(x)}}\right) \qquad \mbox{for every fixed }k;
		\end{equation*}
		\item 
		$$
		\lambda_k \left( \prescript{b}{a}{\mathcal{L}_{\textnormal{BCs},p(x),q(x),w(x)}^{(n)}} \right) \in \R\qquad \mbox{for every } k=1,\dots,n;
		$$
		\item \begin{equation*}
		\left\{(n+1)^{-2} \prescript{b}{a}{\mathcal{L}_{\textnormal{BCs},p(x),q(x),w(x)}^{(n)}}\right\}_n \sim_{\lambda} \omega(x,\theta), \qquad (x,\theta) \in [a,b]\times [0,\pi],
		\end{equation*}
		with $\omega \in \textnormal{L}^\infty\left([a,b]\times [0,\pi]\right)$ and real;
		\item the monotone rearrangement $\tilde{\omega} : [0,1]\to [\min R_\omega, \max R_\omega]$, defined as in equation \eqref{eq:rearrangment}, is piecewise Lipschitz.
	\end{enumerate}
Then
	\begin{equation*}
	\lim_{n \to \infty}\left\{\max_{k=1,\ldots,n}\left| \frac{\lambda_k\left( \prescript{b}{a}{\mathcal{L}_{\textnormal{BCs},p(x),q(x),w(x)}^{(n)}}\right)}{\lambda_k\left( \prescript{b}{a}{\mathcal{L}_{\textnormal{BCs},p(x),q(x),w(x)}}\right)}  -1\right|\right\}  \geq \max_{x\in[0,1]}\left|\frac{\tilde{\omega}(x)}{x^2\pi^2/B^2} - 1 \right|.
	\end{equation*}
Moreover, if for $n$ large enough there are not outliers as in Definition \ref{def:outliers}, i.e., if for every $n\geq N$ 
	$$
	\nexists k \quad \mbox{such that}\quad \lambda_k\left( \prescript{b}{a}{\mathcal{L}_{\textnormal{BCs},p,q,w}^{(n)}}\right) \notin R_\omega,
	$$
then
	\begin{equation*}
	\lim_{n \to \infty}\left\{\max_{k=1,\ldots,n}\left| \frac{\lambda_k\left( \prescript{b}{a}{\mathcal{L}_{\textnormal{BCs},p(x),q(x),w(x)}^{(n)}}\right)}{\lambda_k\left( \prescript{b}{a}{\mathcal{L}_{\textnormal{BCs},p(x),q(x),w(x)}}\right)}  -1\right|\right\}  = \max_{x\in[0,1]}\left|\frac{\tilde{\omega}(x)}{x^2\pi^2/B^2} - 1  \right|.
	\end{equation*}
\end{theorem}
\begin{proof}
Since $\omega$ is bounded then $R_\omega$ is compact and $\min R_\omega, \max R_\omega$ are well-defined. Let $x \in [0,1]$.

Without loss of generality, let us suppose moreover that $\lambda_k\left(\prescript{b}{a}{\mathcal{L}_{\textnormal{BCs},p(x),q(x),w(x)}^{(n)}}\right)$ are distinct for every $k$, 
$$
\lambda_1\left(\prescript{b}{a}{\mathcal{L}_{\textnormal{BCs},p(x),q(x),w(x)}^{(n)}}\right) < \lambda_2\left(\prescript{b}{a}{\mathcal{L}_{\textnormal{BCs},p(x),q(x),w(x)}^{(n)}}\right)<\ldots < \lambda_n\left(\prescript{b}{a}{\mathcal{L}_{\textnormal{BCs},p(x),q(x),w(x)}^{(n)}}\right),
$$
so we can uniquely define the permutation index function $\sigma_n : \left\{1,\ldots, n \right\} \to \left\{1,\ldots, n \right\}$ such that the new index $\sigma_n(k)$ is reordered according to the distance of $(n+1)^{-2}\lambda_k\left(\prescript{b}{a}{\mathcal{L}_{\textnormal{BCs},p(x),q(x),w(x)}^{(n)}}\right)$ from $\tilde{\omega}(x)$, in ascending order. Namely, $\sigma_n(k)=j_n$ such that
$$
\lambda_k\left(\prescript{b}{a}{\mathcal{L}_{\textnormal{BCs},p(x),q(x),w(x)}^{(n)}}\right) = \bar{\lambda}_{j_n}\left(\prescript{b}{a}{\mathcal{L}_{\textnormal{BCs},p(x),q(x),w(x)}^{(n)}}\right) \qquad \mbox{with } j_n\in \{1,\ldots, n\},
$$
where 
$$
\left\{\bar{\lambda}_1\left(\prescript{b}{a}{\mathcal{L}_{\textnormal{BCs},p,q,w}^{(n)}}\right),\ldots, \bar{\lambda}_n\left(\prescript{b}{a}{\mathcal{L}_{\textnormal{BCs},p,q,w}^{(n)}}\right)\right\} \equiv \left\{\lambda_1\left(\prescript{b}{a}{\mathcal{L}_{\textnormal{BCs},p,q,w}^{(n)}}\right),\ldots, \lambda_n\left(\prescript{b}{a}{\mathcal{L}_{\textnormal{BCs},p,q,w}^{(n)}}\right)\right\}
$$
and 
$$
\cdots< \left|\frac{\bar{\lambda}_{j_n-1}\left(\prescript{b}{a}{\mathcal{L}_{\textnormal{BCs},p,q,w}^{(n)}}\right)}{(n+1)^2} - \tilde{\omega}(x) \right|<  \left|\frac{\bar{\lambda}_{j_n}\left(\prescript{b}{a}{\mathcal{L}_{\textnormal{BCs},p,q,w}^{(n)}}\right)}{(n+1)^{2}} - \tilde{\omega}(x) \right|< \left|\frac{\bar{\lambda}_{j_n+1}\left(\prescript{b}{a}{\mathcal{L}_{\textnormal{BCs},p,q,w}^{(n)}}\right)}{(n+1)^{2}} - \tilde{\omega}(x) \right|<\cdots.
$$
For every $n$, choose $k=k(n)$ such that $\sigma_n(k)=1$. By Theorem \ref{thm:clustering&spectral_attraction} we have that
\begin{equation*}
\lim_{n \to \infty} \frac{\lambda_{k(n)}\left(\prescript{b}{a}{\mathcal{L}_{\textnormal{BCs},p(x),q(x),w(x)}^{(n)}}\right)}{(n+1)^2}=\frac{\bar{\lambda}_{1}\left(\prescript{b}{a}{\mathcal{L}_{\textnormal{BCs},p(x),q(x),w(x)}^{(n)}}\right)}{(n+1)^2}= \tilde{\omega}(x), 
\end{equation*}
and by Theorem \ref{thm:discrete_Weyl_law} it holds that
\begin{equation*}
\lim_{n \to \infty} \frac{k(n)}{n} = \lim_{n \to \infty} \frac{\left|\left\{ i=1,\ldots,n \, : \, \frac{\lambda_i\left(\prescript{b}{a}{\mathcal{L}_{\textnormal{BCs},p(x),q(x),w(x)}^{(n)}}\right)}{(n+1)^2}\leq \tilde{\omega}(x) \right\}   \right|}{n} = x.
\end{equation*}
Therefore,
\begin{align*}
\lim_{n \to \infty} \frac{\lambda_{k(n)}\left(\prescript{b}{a}{\mathcal{L}_{\textnormal{BCs},p,q,w}^{(n)}}\right)}{\lambda_{k(n)}\left(\prescript{b}{a}{\mathcal{L}_{\textnormal{BCs},p,q,w}}\right)}&= \lim_{n \to \infty} \frac{\lambda_{k(n)}\left(\prescript{b}{a}{\mathcal{L}_{\textnormal{BCs},p,q,w}^{(n)}}\right)/(n+1)^2}{\lambda_{k(n)}\left(\prescript{b}{a}{\mathcal{L}_{\textnormal{BCs},p,q,w}}\right)/k(n)^2}\cdot \frac{(n+1)^2}{k(n)^2}\\
&=\frac{\tilde{\omega}(x)}{\pi^2/B^2}\cdot \frac{1}{x^2},
\end{align*}
and then for every $\epsilon>0$
\begin{equation*}
\max_{k=1,\ldots,n}\left| \frac{\lambda_k\left( \prescript{b}{a}{\mathcal{L}_{\textnormal{BCs},p(x),q(x),w(x)}^{(n)}}\right)}{\lambda_k\left( \prescript{b}{a}{\mathcal{L}_{\textnormal{BCs},p(x),q(x),w(x)}}\right)} \right|\geq (1-\epsilon) \max_{x\in[0,1]}\left|\frac{\tilde{\omega}(x)}{x^2\pi^2/B^2}\right| \qquad \mbox{definetely}.
\end{equation*}
The thesis follows at once.
\end{proof}

\begin{corollary}\label{cor:necessary_cond_for_uniformity}
In the same hypothesis of Theorem \ref{thm:necessary_cond_for_uniformity}, if
\begin{equation}\label{eq:unif_nec_cond}
\max_{x\in[0,1]}\left|\frac{\tilde{\omega}(x)}{x^2\pi^2/B^2} -1\right|>0
\end{equation}
then the numerical method used for the discretization of Problem \ref{eq:Euler-Cauchy_general} can not provide a relative uniform approximation of the eigenvalues of the continuous operator $ \prescript{b}{a}{\mathcal{L}_{\textnormal{BCs},p,q,w}}$. In particular, it is enough that
$$
\sup_{(x,\theta)\in[a,b]\times [0,\pi]}\omega(x,\theta)=\tilde{\omega}(1)\neq \frac{\pi^2}{B^2}.
$$
\end{corollary}
\begin{proof}
Immediate from Theorem \ref{thm:necessary_cond_for_uniformity}.
\end{proof}

\begin{corollary}
In the same hypothesis of Theorem \ref{thm:necessary_cond_for_uniformity}, for every diffeomorphism $\tau : [a,b] \to [a,b]$, the numerical methods in subsections \ref{ssec:FD}, \ref{ssec:IsoG} verify condition \eqref{eq:unif_nec_cond}, for any fixed order of approximation $\eta$.
\end{corollary}
\begin{proof}
For every fixed $\eta$, it is easy to check that both the methods satisfy the hypothesis of Theorem \ref{thm:necessary_cond_for_uniformity}, in particular $\tilde{\omega}_\eta \in C^1([0,1])$. By the regularity of the functions $f_\eta$, it holds that  $\tilde{\omega}_\eta$ is twice differentiable almost everywhere and $\tilde{\omega}_\eta^{-1}=\phi_\eta$, with $\phi_\eta$ as in \eqref{eq:rearrangment2}. At the point $x_0= \phi_\eta\left(t_0\right)$, with
$$
t_0=(b-a)^{-2}\min_{[a,b]}\left\{\frac{p(\tau(y))}{w(\tau(y))(\tau'(y))^2}\right\}\sup_{[0,\pi]}\left(f_\eta(\theta)\right),
$$
there exists $\epsilon>0$ such that
$$
\tilde{\omega}_\eta''(x_0) \neq  \tilde{\omega}_\eta''(x) \qquad \mbox{for every } x \in (x_0, x_0+\epsilon).
$$
It can be proved by direct computation, using the same approach as in Corollary \ref{cor:item_c}, we skip the details.

Therefore, even if $\tilde{\omega}_\eta(x) = \frac{x^2\pi^2}{B^2}$ for every $x \in [0,x_0]$, there exists a non-negligible compact set $I\subseteq [x_0,1]$ such that $\tilde{\omega}_\eta(x) \neq \frac{x^2\pi^2}{B^2}$.
\end{proof}

On the contrary, if the order of approximation is let free to increase, then the maximum of the relative error seems to decrease to zero, provided that the discretization is made on an appropriate non-uniform grid. See subsections \ref{ssec:FD_nonuniform} and \ref{ssec:galerkin}.

\subsection{Further generalizations}\label{ssec:generalization}
As we stated at the beginning of the Introduction, the results of this section can be generalized to different kind of differential operators in any dimension. In particular, Theorem \ref{thm:necessary_cond_for_uniformity} basically relies just on Theorem \ref{thm:discrete_Weyl_law} and the Weyl's law for Sturm-Liouville operators. Therefore, if we know the asymptotic distribution of the discrete spectrum of the operator $\mathcal{L}$, we can measure the lower bound of the maximum relative approximation error by means of the spectral symbol which characterizes the matrix method used for the discretization. 

As a plain example, consider the following Dirichlet boundary value problem on the unit square,
\begin{equation}\label{eq:Dirichlet_2dim}
\begin{cases}
-\Delta u(x) = \lambda u(x) & \mbox{for } x \in (0,1)^2 \subset \R^2,\\
u(x)=0 & \mbox{for } x \in \partial (0,1)^2,
\end{cases}
\end{equation}
The eigenvalues are given by $\lambda\left(-\Delta_{\textnormal{dir}}\right) = \pi^2\left(i^2+j^2\right)$ for $i,j\in \N$. If we indicate with $\tilde{\Omega}(x)$ the inverse cumulative distribution function for the eigenvalue distribution of $\left\{n^{-2}\lambda_k\left(-\Delta_{\textnormal{dir}}\right)\, : \, k=1,\ldots,n^2 \right\}_n$, even if there is not a closed formula as in the one dimensional case for all $x\in[0,1]$, clearly we have that $\tilde{\Omega}(1)=2\pi^2$.

On the other hand, discretizing problem \eqref{eq:Dirichlet_2dim} by means of $3$-points FD as described in \cite[Chapter 7.3]{GS18}, we get 
\begin{equation*}
\left\{ (n+1)^{-2}\lambda_n\left( -\Delta_{\textnormal{dir}}^{(n^2)}\right) \right\} \sim_{\lambda} \omega(\theta_1,\theta_2)= 4 - 2\cos(\theta_1) - 2\cos(\theta_2), \qquad (\theta_1,\theta_2) \in [0,\pi]^2.
\end{equation*}
Even in the discrete case there is not a closed formula for $\tilde{\omega}$, but nevertheless we have that
\begin{align*}
\lim_{n \to \infty}\left\{ \max_{k=1,\ldots,n} \left| \frac{\lambda_k\left(-\Delta_{\Omega,\textnormal{dir}}^{(n)}\right)}{\lambda_k\left(-\Delta_{\Omega,\textnormal{dir}}\right)}  -1\right|  \right\}&= \max_{x\in[0,1]} \left| \frac{\tilde{\omega}(x)}{\tilde{\Omega}(x)}  -1\right|\\
&\geq \left| \frac{\tilde{\omega}(1)}{\tilde{\Omega}(1)}  -1\right|\\
&= \left| \frac{\max_{[0,\pi]^2}\omega(\theta_1,\theta_2)}{\tilde{\Omega}(1)}  -1\right|\\
&= 1- \frac{8}{2\pi^2}.
\end{align*}

\section{Conclusions}\label{sec:conlcusions}
Although in the present paper for simplicity all the examples and the theory were developed mostly among the setting of regular Sturm-Liouville problems, a generalization to a wider class of differential operators in dimension $d\geq 1$ is feasible by means of the techniques presented in this paper. 

Given a differential operator $\mathcal{L}$ discretized by means of a numerical scheme, the sampling of the spectral symbol calculated by the theory of GLT sequences does not provide in general an accurate approximation of the eigenvalues $\lambda_k\left(\mathcal{L}\right)$. 
Nevertheless, by Theorem \ref{thm:necessary_cond_for_uniformity}  the knowledge of the spectral symbol provides to a numerical discretization scheme a necessary condition for the uniform spectral approximation of $\mathcal{L}$, in the sense of the relative error. It can measure how far the discretization method is from a uniform approximation of all the modes of the differential operator and this can be useful for engineering applications, see for example \cite{HER}. Moreover, the condition seems to become sufficient if the discretization method is paired with a suitable (non-uniform) grid and an increasing refinement of the order of approximation of the method. In light of this, it becomes a priority to devise new specific discretization schemes with
mesh-dependent order of approximation which guarantee a good balance between convergence to zero of the relative spectral error and computational costs. 

Finally, in reference with Theorem \ref{thm:discrete_Weyl_law}, corollaries \ref{cor:FD_uniform}, \ref{cor:Galerkin_uniform}, and \cite[Remark 15]{KLSS18}, since the spectral symbol is deeply related to the Weyl's asymptotic distribution of the eigenvalues of the differential operator, this connection can be exploited to give better estimates of the Weyl function of generic elliptic operators on manifolds with bounded geometry, through a smart discretization of the elliptic operator itself and the analysis of the associated spectral symbol generated by the discretization scheme.


\appendix

\section{Proofs of Subsection \ref{ssec:FD}}\label{ssec:proof_FD}
\textbf{Theorem \ref{thm:FD_symbol}}
\begin{proof}
	The proof of item \eqref{item_spectral_conv_thm:FD_symbol} is long and technical, and we avoid to present it here. Let us just mention that it can be proved by a straightforward generalization of standard techniques, see \cite[Theorem 1]{C69} and \cite{Gary65,CFL}. About item \eqref{item_spectral_symbol_thm:FD_symbol}, let us preliminarily observe that in case of $p(x)\equiv  1$ and $\tau(x)=x$, then $\prescript{a}{b}{\mathcal{L}^{(n,\eta)}_{\textnormal{dir},p(x)\equiv 1}}$ is a symmetric Toeplitz matrix defined by
	$$
	\left(\prescript{a}{b}{\mathcal{L}^{(n,\eta)}_{\textnormal{dir},p(x)\equiv 1}}\right)_{i,j} = \frac{(n+1)^2}{(b-a)^2}d_{\eta,|i-j|},
	$$
	where $d_{\eta,k}$ are the coefficients defined in \eqref{FD_coefficients}, see \cite[Corollary 2.2]{Li05} and \cite[Equation (27)]{KO99}. Therefore,
	\begin{equation*}
	\left\{(n+1)^{-2}\prescript{a}{b}{\mathcal{L}^{(n,\eta)}_{\textnormal{dir},p(x)\equiv 1}}\right\}_n \sim_{\lambda} \frac{f_\eta (\theta)}{(b-a)^2}, \qquad \theta \in [0,\pi].
	\end{equation*}
	
	Let us now define an approximation of $\prescript{a}{b}{\mathcal{L}^{(n,\eta)}_{\textnormal{dir},p(\bar{x})}}$, namely
	\begin{equation*}
	\prescript{a}{b}{\mathfrak{L}^{(n,\eta)}_{\textnormal{dir},p(\bar{x})}} := \begin{bmatrix}
	\frac{p(\tau(x_1))}{\tau'(x_1)^2} & 0 & &  0\\
	0 & \frac{p(\tau(x_2))}{\tau'(x_2)^2} & 0 & \\
	& \ddots & \ddots & \ddots  \\
	0& & 0 & \frac{p(\tau(x_n))}{\tau'(x_n)^2}
	\end{bmatrix}\prescript{a}{b}{\mathcal{L}^{(n,\eta)}_{\textnormal{dir},p(x)\equiv 1}}.
	\end{equation*}
	$\prescript{a}{b}{\mathfrak{L}^{(n,\eta)}_{\textnormal{dir},p(\bar{x})}}$ is symmetric and $\left\|(n+1)^{-2}\prescript{a}{b}{\mathfrak{L}^{(n,\eta)}_{\textnormal{dir},p(\bar{x})}}\right\|\leq \max_{[a,b]}\left\{ p(\tau(x))/\tau'(x)^2\right\}\max_{[0,\pi]}\left\{|f_\eta (\theta)|\right\}$. Moreover, by \cite[\textbf{GLT 3-4} p. 160]{GS17} and Proposition \ref{prop:GLT1}, it holds that
	\begin{equation}\label{eq:thmFD_1}
	\left\{(n+1)^{-2}\prescript{a}{b}{\mathfrak{L}^{(n,\eta)}_{\textnormal{dir},p(\bar{x})}} \right\}_n \sim_\lambda \frac{p(\tau(x))}{(b-a)^2\tau'(x)^2}f_\eta(\theta), \qquad (x,\theta)\in [a,b]\times[0,\pi]. 
	\end{equation}
	
	By the regularity of $p(x)$ and $\tau(x)$, it is not difficult to prove that
	\begin{equation}\label{eq:thmFD_2}
	(n+1)^{-2}\left(\prescript{a}{b}{\mathcal{L}^{(n,\eta)}_{\textnormal{dir},p(\bar{x})}} - \prescript{a}{b}{\mathfrak{L}^{(n,\eta)}_{\textnormal{dir},p(\bar{x})}}\right) = Y^{(n)},
	\end{equation} 
	such that 
	$$
	\|Y^{(n)}\|\leq c, \qquad n^{-1} \|Y^{(n)}\|_1 \to 0.
	$$
	Combining now \eqref{eq:thmFD_1} and \eqref{eq:thmFD_2}, by  \cite[Property \textbf{S 4} p. 156]{GS17} we get that 
	$$
	\left\{(n+1)^{-2}\prescript{a}{b}{\mathcal{L}^{(n,\eta)}_{\textnormal{dir},p(\bar{x})}} \right\}_n \sim_\lambda \frac{p(\tau(x))}{(b-a)^2\tau'(x)^2}f_\eta(\theta), \qquad (x,\theta)\in [a,b]\times[0,\pi].
	$$
	Finally, it is immediate to check that $\left\|(n+1)^{-2} Q^{(n)}\right\|\leq cn^{-2}$ and again, by \cite[properties \textbf{GLT 3-4} p. 160]{GS17} and \cite[Property \textbf{S 4} p. 156]{GS17} we conclude that
	$$
	\left\{(n+1)^{-2}\prescript{a}{b}{\mathcal{L}^{(n,\eta)}_{\textnormal{dir},p(\bar{x}),q(\bar{x}),w(\bar{x})}} \right\}_n = \left\{ \left(W^{(n)}\right)^{-1}\left(\frac{\prescript{a}{b}{\mathcal{L}^{(n,\eta)}_{\textnormal{dir},p(\bar{x})}}}{(n+1)^2} + \frac{Q^{(n)}}{(n+1)^2}\right)\right\}_n \sim_{\lambda} \frac{p(\tau(x))}{(b-a)^2w(\tau(x))\tau'(x)^2}f_\eta(\theta).
	$$
\end{proof}

\noindent \textbf{Corollary \ref{cor:FD_uniform}}.
\begin{proof}
	$f_\eta(\theta)$ is obviously $C^\infty([0,\pi])$. Let us begin to prove that $f_\eta(\theta)\sim \theta^2$ as $\theta \to 0$ and that it is monotone nonnegative on $[0,\pi]$. By the Taylor expansion at $\theta=0$ we get
	\begin{align*}
	f_\eta(\theta) &= d_{\eta,0} + 2 \sum_{k=1}^\eta d_{\eta,k}\cos(k\theta)\\
	&= d_{\eta,0} + \sum_{\substack{k=-\eta\\k\neq 0}}^\eta d_{\eta,k}\cos(k\theta)\\
	&= d_{\eta,0} + \sum_{\substack{k=-\eta\\k\neq 0}}^\eta d_{\eta,k}\left(1 - \frac{(k\theta)^2}{2}\right) + o(\theta^4)\\
	&=  -\theta^2\sum_{\substack{k=-\eta\\k\neq 0}}^\eta  (-1)^{k} \frac{\eta!\eta!}{(\eta-k)!(\eta+k)!} + o(\theta^4)\\
	&= -\theta^2 \left[(-1)^\eta\frac{\eta!\eta!}{(2\eta)!}\sum_{\substack{m=0\\m\neq \eta}}^{2\eta} (-1)^{m}\binom{2\eta}{m}\right] + o(\theta^4)\\
	&= \theta^2 + o(\theta^4).
	\end{align*} 
	Moreover, let us observe that
	\begin{align*}
	&f_\eta(\theta)= d_{\eta,0}+2\sum_{k=1}^\eta d_{\eta,k}\cos(k\theta) = d_{\eta,0}+ 2 \sum_{k=1}^\eta \left|d_{\eta,k}\right|\cos(k(-\theta+\pi)),\\
	&f_\eta'(\theta)=  2 \sum_{k=1}^\eta k\left |d_{\eta,k}\right|\sin(k(-\theta+\pi)).
	\end{align*}
	Define then 
	$$
	g (\psi) = \sum_{k=1}^\eta a_{k} \sin(k\psi), \qquad \mbox{with }\psi=-\theta+\pi \in [0,\pi], \qquad \mbox{and } a_{k} = 2k\left|d_{\eta,k}\right|.
	$$
	It is immediate to check that $a_1 \geq a_2 \geq \ldots \geq a_n > 0$ and that
	$$
	(2k)a_{2k} \leq (2k-1)a_{2k-1} \quad \forall k \geq 1.
	$$	
	By \cite[Theorem 1]{AS74} we can conclude that $g(\psi)>0$ on $(0,\pi)$ and then $f_\eta'(\theta)>0$ on $(0,\pi)$. Since $f_\eta(0)=0$, we deduce that $f_\eta(\theta)\geq 0$ on $[0,\pi]$. 
	
	The second part of the thesis is an immediate consequence of identities \eqref{FD_coefficients}. Indeed, for every fixed $k\geq 1$ it holds that
	\begin{equation*}
	\sum_{k=1}^\eta \left|d_{\eta,k} \right|\leq \sum_{k=1}^\eta\frac{2}{k^2}\leq \frac{\pi^2}{3} \qquad \forall \eta,k\geq 1,
	\end{equation*}
	and
	\begin{equation*}
	\frac{2}{k^2}-\left| d_{\eta,k}\right|   = o(1) \quad \mbox{as }\eta\to \infty, \qquad  \sum_{k=1}^{\eta} (-1)^k\left(\frac{2}{k^2} - \left|d_{\eta,k} \right|\right)\to 0 \quad \mbox{as }\eta\to \infty.
	\end{equation*}
	Since $\sum_{k=1}^{\infty} \frac{(-1)^{k}}{k^2} = -\frac{\pi^2}{12}$, we conclude that
	$$
	-\lim_{\eta\to \infty} 2\sum_{k=1}^\eta d_{\eta,k} = -2\sum_{k=1}^\infty (-1)^{k}\frac{2}{k^2}= \frac{\pi^2}{3}.
	$$
	Therefore, for every fixed $\theta \in [0,\pi]$,
	\begin{equation*}
	\lim_{\eta\to \infty} f_\eta(\theta) = \frac{\pi^2}{3} + 4\sum_{k=1}^\infty \frac{(-1)^k}{k^2}\cos(k\theta)= \theta^2,
	\end{equation*}
	being $\left\{\pi^2/3\right\}\cup \left\{(-1)^k4/k^2\right\}_{k\geq1}$ the Fourier coefficients of $\theta^2$ on $[0,\pi]$, and the convergence is uniform.
\end{proof}


\begin{thebibliography}{99}
	\bibitem{AGM09} L. Aceto, P. Ghelardoni, and C. Magherini, {\em "Boundary Value Methods as an extension of Numerov's method for Sturm–Liouville eigenvalue estimates}. Appl. Numer. Math. 59(7) (2009): 1644--1656.
	
	
	\bibitem{A95}	S. D. Algazin, {\em Calculating the eigenvalues of ordinary differential equations}. Comput. Math. Math. Phys. 35(4) (1995): 477--482.
	
	\bibitem{AS05} P. Amodio, I. Sgura, {\em High-order finite difference schemes for the solution of second-order BVPs}. J. Comput. Appl. Math. 176(1) (2005): 59--76.
	
	\bibitem{AS11} P Amodio, G. Settani, {\em A matrix method for the solution of Sturm-Liouville problems}. JNAIAM 6(1-2) (2011): 1--13.
	
	\bibitem{AS74} R. Askey, J. Steinig, {\em Some positive trigonometric sums}. Trans. Amer. Math. Soc. 187 (1974): 295--307.	
				
	\bibitem{BBCHS06} Y. Bazilevs, L. Beirao da Veiga, J. A. Cottrell, T. J. Hughes, G. Sangalli, {\em Isogeometric analysis: approximation, stability and error estimates for h-refined meshes}. ‎Math. Models Methods Appl. Sci. 16(07) (2006): 1031--1090.
	
	\bibitem{BS18} D. Bianchi, S. Serra-Capizzano, {\em Spectral analysis of finite-dimensional approximations of 1d waves in non-uniform grids}. Calcolo 55(47) (2018). 
	
	\bibitem{BBGM15} J. M. Bogoya, A. B\"{o}ttcher, S. M. Grudsky, E. A. Maximenko, {\em Eigenvalues of Hermitian Toeplitz matrices with smooth simple-loop symbols}. J. Math. Anal. Appl. 442(2): 1308--1334.
	
	\bibitem{BS13} A. B\"{o}ttcher, B. Silbermann. {\em Analysis of Toeplitz operators}. Springer Science\&Business Media (2013)
	
	
	\bibitem{BIK} D. Burago, S. Ivanov, Y. Kurylev, {\em A graph discretization of the Laplace-Beltrami operator}. J. Spectr. Theory 4(4) (2014): 675--714.
	
	\bibitem{C69} A. Carasso, {\em Finite-difference methods and the eigenvalue problem for nonselfadjoint Sturm-Liouville operators.} Math. Comp. 23(108) (1969): 717--729.
	\bibitem{CP79} G. Chiti, and C. Pucci, {\em Rearrangements of functions and convergence in Orlicz spaces.} Appl. Anal. 9(1) (1979): 23--27.
	
	
	\bibitem{BC96} A. Boumenir, B. Chanane, {\em Eigenvalues of S-L systems using sampling theory}. Appl. Anal. 62(3-4) (1996): 323--334.
	
	\bibitem{CHB} J. A. Cottrell, T. J. Hughes, Y. Bazilevs. {\em Isogeometric analysis: toward integration of CAD and FEA.} John Wiley \& Sons, 2009. 
	
	
	\bibitem{CFL} R. Courant, K. Friedrichs, H. Lewy, {\em \"{U}ber die partiellen Differenzengleichungen der mathematischen Physik}. Math. Ann. 100(1) (1928): 32--74.
	
	\bibitem{DBFS93} F. Di Benedetto, G. Fiorentino, S. Serra-Capizzano, {\em CG preconditioning for Toeplitz matrices}. Comput. Math. Appl. 25(6) (1993): 35--45,
	
	\bibitem{DGMS15} M. Donatelli, C. Garoni, C. Manni, S. Serra-Capizzano, {\em Robust
		and optimal multi-iterative techniques for IgA Galerkin linear systems}. CMAME 284 (2015): 230-–264.
	\bibitem{DMS16} M. Donatelli, M. Mazza, S. Serra-Capizzano, {\em Spectral analysis and structure preserving
		preconditioners for fractional diffusion equations}. J. Comput. Phys. 307 (2016): 262-–279.
	
\bibitem{DFFMS18} M. Dumbser, F. Fambri, I. Furci, M. Mazza, S. Serra-Capizzano, M. Tavelli, {\em Staggered discontinuous Galerkin methods for the incompressible Navier–Stokes equations: Spectral analysis and computational results}. Numer. Linear Algebra Appl. 25(5) (2018): e2151. 

\bibitem{EFGMSS18} S. E. Ekstr\"{o}m, I. Furci, C. Garoni, C. Manni, S. Serra-Capizzano, H. Speleers, {\em Are the eigenvalues of the B‐spline isogeometric analysis approximation of $-\Delta u = \lambda u$ known in almost closed form?} Numer. Linear Algebra Appl. 25(5) (2018): e2198. 

	\bibitem{EMZ16} S. Ervedoza, A. Marica, E. Zuazua, {\em Numerical meshes ensuring uniform observability of onedimensional
		waves: construction and analysis}. IMA J. Numer. Anal. 36 (2016): 503-–542.
	
	
	
	\bibitem{E05} W. N. Everitt, {\em A catalogue of Sturm-Liouville differential equations}. In Sturm-Liouville Theory. Birkhäuser Basel, 2005: 271--331.
	
	\bibitem{Garoni18} C. Garoni, {\em Spectral distribution of PDE discretization matrices from isogeometric analysis: the case of $L^1$ coefficients and non-regular geometry}. J. Spectr. Theor. 8 (2018): 297--313. 
	
	\bibitem{GS17} C. Garoni, S. Serra-Capizzano, {\em Generalized Locally Toeplitz Sequences: Theory and Applications}. Springer, Cham (2017).
	\bibitem{GS18} C. Garoni, S. Serra-Capizzano, {\em Generalized Locally Toeplitz Sequences: Theory and Applications, Volume II}. Springer, Cham (2018).
	
	
	
	\bibitem{GSERSH18} C. Garoni, H. Speleers, S.-E. Ekström, A. Reali, S. Serra-Capizzano,
	T. J.-R. Hughes, {\em Symbol-based analysis of finite element and isogeometric	B-spline discretizations of eigenvalue problems: Exposition and review}. Arch. Computat. Methods Eng. (2019): 1--52.
	
	\bibitem{Gary65} J. Gary, {\em Computing eigenvalues of ordinary differential equations by finite differences}. Math. Comp. 19(91) (1965): 365--379.
	
	\bibitem{Szego84} U. Grenander, G. Szego. {\em Toeplitz Forms and their Applications}, 2nd ed. Chelsea, New York, 1984. 
	
	\bibitem{HER} T. J.-R. Hughes, J. A. Evans, A. Reali, {\em Finite element and NURBS approximations of eigenvalue, boundary-value, and initial-value problems}. CMAME 272 (2014): 290--320.
	
	\bibitem{IZ99} J. A. Infante, E. Zuazua, {\em Boundary observability for the space semi discretizations of the 1-d wave
		equation}. Math. Model. Num. Ann. 33 (1999): 407-–438. 
	
	\bibitem{Li05} J. Li, {\em General explicit difference formulas for numerical differentiation}. J. Comput. Appl. Math. 183(1) (2005): 29--52.
	
	\bibitem{LL18} V. Limic, N. Limi\'{c}, {\em Equidistribution, uniform distribution: a probabilist's perspective}. Probab. Surv. 15 (2018): 131--155.
	
	
	\bibitem{LBB} U. Von Luxburg, M. Belkin, O. Bousquet, {\em Consistency of spectral clustering}. Ann. Statist. (2008): 555-586.
	
	
	\bibitem{KO99} I. R. Khan, R. Ohba, {\em Closed-form expressions for the finite difference approximations of first and higher derivatives based on Taylor series}. J. Comput. Appl. Math. 107(2) (1999): 179--193.
	
	\bibitem{KLSS18}A. Kunoth, T. Lyche, G. Sangalli, S. Serra-Capizzano. {\em Splines and PDEs: From Approximation Theory to Numerical Linear Algebra}. Vol. 2219 Springer (2018).
	
	\bibitem{KNT17} V. Kravchenko, L. J. Navarro, S. M. Torba, {\em Representation of solutions to the one-dimensional Schr\"{o}dinger equation in terms of Neumann series of Bessel functions}. Appl. Math. Comput. 314 (2017): 173--192.
	
	\bibitem{KT15} V. Kravchenko, S. M. Torba, {\em Analytic approximation of transmutation operators and applications to highly accurate solution of spectral problems}. J. Comput. Appl. Math. 275 (2015): 1--26.
	
	\bibitem{Kuipers-Niederreiter} L. Kuipers, H. Niederreiter. {\em Uniform distribution of sequences}. John Wiley \& Sons, Inc., New York (1974).
	
	
	\bibitem{PHA81} J. Paine, F. R. de Hoog, R. S. Anderssen, {\em On the correction of finite difference eigenvalue approximations for Sturm-Liouville problems}. Computing 26(2) (1981): 123--139.
	
	\bibitem{P93} J. D. Pryce, {\em Numerical solution of Sturm-Liouville problems}.  Oxford University Press, 1993.
	
	\bibitem{PDC} V. Puzyrev, Q. Deng, V. Calo, {\em Spectral approximation properties of isogeometric analysis with variable continuity}. CMAME 334 (2018): 22--39.
	
	
	
	\bibitem{RBV10} L. Rosasco, M. Belkin, E. De Vito, {\em On learning with integral operators}. J. Mach. Learn. Res. 11 (2010): 905--934.
	
		
	
	\bibitem{S12} I. S. Sargsjan, {\em Sturm—Liouville and Dirac Operators}. Vol. 59. Springer Science \& Business Media, 2012.
	
	\bibitem{Serra98} S. Serra-Capizzano, {\em An ergodic theorem for classes of preconditioned matrices}. Linear Algebra Appl. 282(1--3) (1998): 161--183.
	
	\bibitem{Serra03} S. Serra-Capizzano, {\em Generalized locally Toeplitz sequences: spectral analysis and applications to discretized partial differential equations}. Linear Algebra Appl. 366 (2003): 371--402.
	\bibitem{Serra06} S. Serra-Capizzano, {\em The GLT class as a generalized Fourier analysis and applications}. Linear Algebra Appl. 419 (2006) 180–-233.
	
	
	\bibitem{SG05} G. Sewell, {\em The numerical solution of ordinary and partial differential equations}. Vol. 75. John Wiley\&Sons, 2005.
	
	\bibitem{Smith85} G. D. Smith, {\em Numerical Solution of Partial Differential Equations: Finite Difference Methods, 3rd edn.} Clarendon Press, Oxford (1985).
	
	\bibitem{T86} G. Talenti, {\em Rearrangements of functions and partial differential equations.} Nonlinear Diffusion Problems. Springer, Berlin, Heidelberg (1986): 153--178.
	
	\bibitem{Ti98} P. Tilli, {\em Locally Toeplitz sequences: spectral properties and applications}. Linear Algebra Appl. 278(1-3) (1998): 91--120.
	
	\bibitem{Tyrty96} E.E. Tyrtyshnikov, {\em A unifying approach to some old and new theorems on distribution and
		clustering}. Linear Algebra Appl. 232 (1996): 1--43. 
	
	\bibitem{Tyrty98} E.E. Tyrtyshnikov, N. Zamarashkin, {\em Spectra of multilevel Toeplitz matrices: advanced Theory
		via simple matrix relationships}. Linear Algebra Appl. 270 (1998): 15--27. 
	\bibitem{W58} H. Widom, {\em On the eigenvalues of certain Hermitian operators}. Trans. Amer. Math. Soc. 88(2) (1958): 491--522. 	
	
	
	
	\bibitem{Z15} A. Zettl, {\em Sturm-liouville theory}. No. 121. American Mathematical Soc., 2005.
	
	\bibitem{ZD98} D. Zwillinger, {\em Handbook of differential equations}.  Gulf Professional Publishing, 1998.
	



\end{thebibliography}
\end{document}